\documentclass[12pt]{amsart}

\usepackage{amssymb}

\usepackage{enumitem}

\usepackage{graphicx}

\makeatletter
\@namedef{subjclassname@2010}{%
  \textup{2010} Mathematics Subject Classification}
\makeatother

\usepackage[T1]{fontenc}

\usepackage{xcolor}
\usepackage[colorlinks, linkcolor=blue, citecolor=blue, urlcolor=blue]{hyperref}

\newtheorem{theorem}{Theorem}[section]
\newtheorem{corollary}[theorem]{Corollary}
\newtheorem{fact}[theorem]{Fact}
\newtheorem{lemma}[theorem]{Lemma}
\newtheorem{proposition}[theorem]{Proposition}
\newtheorem{question}[theorem]{Question}

\newtheorem*{xclm}{Claim}

\theoremstyle{definition}
\newtheorem{definition}[theorem]{Definition}

\newtheorem*{xrem}{Remark}

\numberwithin{equation}{section}

\DeclareMathOperator{\cals}{\mathcal{S}}
\DeclareMathOperator{\calspdf}{\mathcal{S}_{pdfin}}
\DeclareMathOperator{\calsf}{\mathcal{S}_{fin}}
\DeclareMathOperator{\calsc}{\mathcal{S}_{ctbl}}
\DeclareMathOperator{\dom}{dom}
\DeclareMathOperator{\ran}{ran}
\DeclareMathOperator{\mov}{mov}
\DeclareMathOperator{\pdfin}{pdfin}
\DeclareMathOperator{\fin}{fin}
\DeclareMathOperator{\seq}{seq}
\DeclareMathOperator{\orb}{orb}
\DeclareMathOperator{\symg}{sym_{\mathcal{G}}}
\DeclareMathOperator{\fixg}{fix_{\mathcal{G}}}
\DeclareMathOperator{\cov}{cov}
\DeclareMathOperator{\hght}{ht}
\DeclareMathOperator{\scc}{succ}
\DeclareMathOperator{\prd}{pred}
\DeclareMathOperator{\proj}{pr}

\begin{document}


\baselineskip=17pt


\title[Factorials of Infinite Cardinals]{Factorials of infinite cardinals in $\mathsf{ZF}$}

\author[G. Shen]{Guozhen Shen}
\address{Academy of Mathematics and Systems Science\\
Chinese Academy of Sciences\\
No.55 Zhongguancun East Road\\
Beijing 100190\\
People's Republic of China}
\address{University of Chinese Academy of Sciences\\
No.19(A) Yuquan Road\\
Beijing 100049\\
People's Republic of China}
\email{shen\_guozhen@outlook.com}

\author[J. Yuan]{Jiachen Yuan}
\email{yuanjiachen15@mails.ucas.ac.cn}

\date{}

\begin{abstract}
For a set $x$, let $\cals(x)$ be the set of all permutations of $x$.
We study several aspects of this notion in $\mathsf{ZF}$.
The main results are as follows:
\begin{enumerate}
  \item $\mathsf{ZF}$ proves that for all sets $x$, if $\cals(x)$ is Dedekind infinite,
        then there are no finite-to-one maps from $\cals(x)$ into $\calsf(x)$,
        where $\calsf(x)$ is the set of all permutations of $x$
        which move only finitely many elements.
  \item $\mathsf{ZF}$ proves that for all sets $x$, the cardinality of $\cals(x)$
        is strictly greater than that of $[x]^2$.
  \item It is consistent with $\mathsf{ZF}$ that there exists an infinite set $x$ such that
        the cardinality of $\cals(x)$ is strictly less than that of $[x]^3$.
  \item It is consistent with $\mathsf{ZF}$ that there exists an infinite set $x$ such that
        there is a finite-to-one map from $\cals(x)$ into $x$.
\end{enumerate}
\end{abstract}

\subjclass[2010]{Primary 03E10, 03E25}

\keywords{$\mathsf{ZF}$, cardinal, permutation, finite-to-one map, Dedekind finite}

\maketitle

\section{Introduction}
In \cite{DawsonHoward1976}, Dawson and Howard defined $\mathfrak{a}!$, the factorial of a cardinal $\mathfrak{a}$,
as the cardinality of the set of all permutations of a set which is of cardinality~$\mathfrak{a}$.
In $\mathsf{ZFC}$, $\mathfrak{a}!=2^\mathfrak{a}$ for all infinite cardinals $\mathfrak{a}$.
However, Dawson and Howard proved that, without $\mathsf{AC}$ (i.e., the axiom of choice),
we cannot conclude any relationship between $\mathfrak{a}!$
and $2^\mathfrak{a}$ for an arbitrary infinite cardinal $\mathfrak{a}$.
On the other hand, they proved in $\mathsf{ZF}$ that
for all cardinals $\mathfrak{a}\geqslant3$, $\mathfrak{a}<\mathfrak{a}!$.
Recently, in \cite{SonpanowVejjajiva2018}, Sonpanow and Vejjajiva generalized this result by proving in $\mathsf{ZF}$
that for all infinite cardinals $\mathfrak{a}$ and all natural numbers $n$, $\mathfrak{a}^n<\mathfrak{a}!$.

In \cite{Forster2003}, Forster proved in $\mathsf{ZF}$ that for all infinite sets $x$,
there are no finite-to-one maps from $\wp(x)$ into $x$, where $\wp(x)$ is the power set of $x$.
In \cite{SonpanowVejjajiva2017}, Sonpanow and Vejjajiva gave a condition that makes Forster's theorem,
with $\wp(x)$ replaced by $\cals(x)$, provable in $\mathsf{ZF}$:
They showed in $\mathsf{ZF}$ that for all infinite sets $x$,
if there exists a permutation $f$ of $x$ without fixed points such that $f\circ f=\mathrm{id}_x$,
where $\mathrm{id}_x$ is the identity permutation of $x$,
then there are no finite-to-one maps from $\cals(x)$ into $x$.

In this paper, we thoroughly investigate the properties of $\mathfrak{a}!$ for infinite cardinals $\mathfrak{a}$.
Our first main result is a common generalization of the results mentioned above:
We prove in $\mathsf{ZF}$ that for all sets $x$, if $\cals(x)$ is Dedekind infinite,
then there are no finite-to-one maps from $\cals(x)$ into $\calsf(x)$.
Actually, we prove a more general result:
We say that a set $x$ is power Dedekind finite if the power set of $x$ is Dedekind finite.
For a set $x$, let $\calspdf(x)$ be the set of all permutations of $x$
which move only power Dedekind finitely many elements.
We prove in $\mathsf{ZF}$ that for all sets $x$, if $\cals(x)$ is Dedekind infinite,
then there are no Dedekind finite to one maps from $\cals(x)$ into $\calspdf(x)$.

Many statements concerning $\cals(x)$, including Sonpanow and Vejjajiva's two results stated above,
can be deduced as corollaries of this theorem. Among these corollaries,
we shall mention the following generalization of Dawson and Howard's result:
For all cardinals $\mathfrak{a}$, $[\mathfrak{a}]^2<\mathfrak{a}!$.
On the other hand, a Shelah-type permutation model is constructed in order to
show that the following statement is consistent with $\mathsf{ZF}$:
There exists a Dedekind infinite set $x$ such that the cardinality of $\cals(x)$ is
strictly less than that of~$[x]^3$ and such that there is a surjection from $x$ onto $\cals(x)$.

Finally, we construct a new permutation model in which there is an infinite set $x$
such that there exists a finite-to-one map from $\cals(x)$ into $x$.
This result shows that Forster's theorem,
with $\wp(x)$ replaced by $\cals(x)$, cannot be proved in $\mathsf{ZF}$.
Our results also settle several open problems from \cite{Tachtsis2018}.

\section{Preliminaries}
Throughout this paper, we shall work in $\mathsf{ZF}$
(i.e., the Zermelo-Fraenkel set theory without the axiom of choice).
In this section, we indicate briefly our use of some terminology and notation.
The cardinal of $x$, which we shall denote by $|x|$, is the least ordinal $\alpha$ equinumerous to $x$,
if $x$ is well-orderable, and the set of all sets $y$ of least rank which are equinumerous to~$x$,
otherwise (cf.~\cite[III.2.2]{Levy1979}). We shall use lower case German letters
$\mathfrak{a}$, $\mathfrak{b}$, $\mathfrak{c}$, $\mathfrak{d}$ for cardinals.
For a function $f$, we use $\dom(f)$ for the domain of $f$, $\ran(f)$ for the range of $f$,
$f[x]$ for the image of $x$ under $f$, $f^{-1}[x]$ for the inverse image of $x$ under $f$,
and $f\upharpoonright x$ for the restriction of $f$ to $x$.

We write $x\preccurlyeq y$ to express that there is an injection from $x$ into $y$,
and $x\preccurlyeq^\ast y$ to express that there is a surjection from a subset of $y$ onto $x$.
For all cardinals $\mathfrak{a}$, $\mathfrak{b}$, $\mathfrak{a}\leqslant\mathfrak{b}$
($\mathfrak{a}\leqslant^\ast\mathfrak{b}$) means that there are sets $x$, $y$ such that
$|x|=\mathfrak{a}$, $|y|=\mathfrak{b}$, and $x\preccurlyeq y$ ($x\preccurlyeq^\ast y$).
We use $\mathfrak{a}\nleqslant\mathfrak{b}$ ($\mathfrak{a}\nleqslant^\ast\mathfrak{b}$) to denote
the negation of $\mathfrak{a}\leqslant\mathfrak{b}$ ($\mathfrak{a}\leqslant^\ast\mathfrak{b}$).
If $f$ is an injection from $x$ into $y$ and $g$ is an injection from $y$ into $z$,
then $g\circ f$, the composition of $g$ and $f$, is an injection from $x$ into $z$.
Hence, if $\mathfrak{a}\leqslant\mathfrak{b}$ and $\mathfrak{b}\leqslant\mathfrak{c}$ then
$\mathfrak{a}\leqslant\mathfrak{c}$. It is the same case when we replace $\leqslant$ by $\leqslant^\ast$.
It is also easily verifiable that for all cardinals $\mathfrak{a}$, $\mathfrak{b}$,
if $\mathfrak{a}\leqslant\mathfrak{b}$ then $\mathfrak{a}\leqslant^\ast\mathfrak{b}$,
and if $\mathfrak{a}\leqslant^\ast\mathfrak{b}$ then $2^\mathfrak{a}\leqslant2^\mathfrak{b}$.

We shall frequently use expressions like ``from ... one can explicitly define ...''
in our formulations, for which we make the following convention.
Let $\varphi(p_1,\dots,p_m,x_0,\dots,x_n)$ and $\psi(p_1,\dots,p_m,x_0,\dots,x_n,y)$
be formulas of set theory with no free variables other than indicated.
When we say that \emph{from} $x_0,\dots,x_n$ \emph{such that}
$\varphi(p_1,\dots,p_m,x_0,\dots,x_n)$, \emph{one can explicitly define a} $y$
\emph{such that} $\psi(p_1,\dots,p_m,x_0,\dots,x_n,y)$, we mean the following:
\begin{quote}
There exists a class function $G$ without free variables such that if $\varphi(p_1,\dots,p_m,x_0,\dots,x_n)$,
then $(x_0,\dots,x_n)$ is in the domain of $G$ and $\psi(p_1,\dots,p_m,x_0,\dots,x_n,G(x_0,\dots,x_n))$.
\end{quote}
For example, according to this convention, the second part of Theorem~\ref{cbt}
states that there exists a class function $G$ without free variables such that
if $f$ is an injection from $x$ into $y$ and $g$ is an injection from $y$ into $x$,
then $G(f,g)$ is defined and is a bijection from $x$ onto $y$.

\begin{theorem}[Cantor-Bernstein]\label{cbt}
If $\mathfrak{a}\leqslant\mathfrak{b}$ and $\mathfrak{b}\leqslant\mathfrak{a}$ then we have $\mathfrak{a}=\mathfrak{b}$.
Moreover, from an injection $f:x\to y$ and an injection $g:y\to x$,
one can explicitly define a bijection $h:x\twoheadrightarrow y$.
\end{theorem}
\begin{proof}
Cf.~\cite[III.2.8]{Levy1979}.
\end{proof}

For all cardinals $\mathfrak{a}$, $\mathfrak{b}$, $\mathfrak{a}<\mathfrak{b}$
means that $\mathfrak{a}\leqslant\mathfrak{b}$ but not $\mathfrak{b}\leqslant\mathfrak{a}$.
By Theorem~\ref{cbt}, if $\mathfrak{a}<\mathfrak{b}$ and $\mathfrak{b}\leqslant\mathfrak{c}$,
or if $\mathfrak{a}\leqslant\mathfrak{b}$ and $\mathfrak{b}<\mathfrak{c}$, then $\mathfrak{a}<\mathfrak{c}$.

\subsection{Dedekind finiteness and power Dedekind finiteness}
It is well-known that, if $\mathsf{ZF}$ is consistent,
we cannot prove in $\mathsf{ZF}$ that every infinite set includes a denumerable subset,
and we cannot even prove in $\mathsf{ZF}$ that the power set of an infinite set includes a denumerable subset.
This suggests us to introduce the following definition.

\begin{definition}
A set $x$ is \emph{Dedekind infinite} (\emph{power Dedekind infinite})
if $\omega\preccurlyeq x$ ($\omega\preccurlyeq\wp(x)$);
otherwise $x$ is \emph{Dedekind finite} (\emph{power Dedekind finite}).
A cardinal $\mathfrak{a}$ is \emph{Dedekind infinite} (\emph{power Dedekind infinite})
if $\aleph_0\leqslant\mathfrak{a}$ ($\aleph_0\leqslant2^\mathfrak{a}$);
otherwise $\mathfrak{a}$ is \emph{Dedekind finite} (\emph{power Dedekind finite}).
\end{definition}

\begin{xrem}
The name ``power Dedekind finite'' was first introduced by Blass
in a manuscript which is not intended for publication (cf.~\cite{Blass2013}).
This notion was called ``III-finite'' by Levy in \cite{Levy1958},
``weakly Dedekind finite'' by Degen in \cite{Degen1994},
and ``$C$-finite'' by Herrlich in \cite{Herrlich2011}.
In \cite{Truss1974}, Truss denoted the class of all power Dedekind finite cardinals by $\Delta_4$.
\end{xrem}

It is obvious that all Dedekind infinite cardinals are power Dedekind infinite,
and all power Dedekind infinite cardinals are infinite.
The following result about Dedekind infinite cardinals is useful.

\begin{theorem}[Dedekind]\label{dedt}
For all cardinals $\mathfrak{a}$, if there are sets $x$, $y$ such that $|y|=|x|=\mathfrak{a}$
and $y$ is a proper subset of $x$, then $\mathfrak{a}$ is Dedekind infinite.
\end{theorem}
\begin{proof}
Cf.~\cite[III.1.20]{Levy1979}.
\end{proof}

For power Dedekind infinite cardinals, we have the following theorem.

\begin{theorem}[Kuratowski]\label{kurt}
For every cardinal $\mathfrak{a}$, $\mathfrak{a}$ is power Dedekind infinite
iff $\aleph_0\leqslant^\ast\mathfrak{a}$ iff $2^{\aleph_0}\leqslant2^\mathfrak{a}$.
\end{theorem}
\begin{proof}
Cf.~\cite[pp.~94--95]{Tarski1924} or \cite[Proposition~5.4]{Halbeisen2017}.
\end{proof}

Using Kuratowski's theorem, one can easily prove that
the class of all power Dedekind finite sets is
closed under unions (cf.~\cite[Theorem~1]{Truss1974}).

\begin{definition}
A function $f$ is a \emph{finite} (\emph{Dedekind finite}; \emph{power Dedekind finite}) \emph{to one map}
if for all $z\in\ran(f)$, $f^{-1}[\{z\}]$ is finite (Dedekind finite; power Dedekind finite).
We write $x\preccurlyeq_{\mathrm{fto}}y$ ($x\preccurlyeq_{\mathrm{dfto}}y$; $x\preccurlyeq_{\mathrm{pdfto}}y$)
to express that there is a finite (Dedekind finite; power Dedekind finite) to one map from $x$ into $y$,
and $\mathfrak{a}\leqslant_{\mathrm{fto}}\mathfrak{b}$ ($\mathfrak{a}\leqslant_{\mathrm{dfto}}\mathfrak{b}$;
$\mathfrak{a}\leqslant_{\mathrm{pdfto}}\mathfrak{b}$) to express that there are sets $x$, $y$ such that
$|x|=\mathfrak{a}$, $|y|=\mathfrak{b}$, and $x\preccurlyeq_{\mathrm{fto}}y$
($x\preccurlyeq_{\mathrm{dfto}}y$; $x\preccurlyeq_{\mathrm{pdfto}}y$).
We use $\mathfrak{a}\nleqslant_{\mathrm{fto}}\mathfrak{b}$ ($\mathfrak{a}\nleqslant_{\mathrm{dfto}}\mathfrak{b}$;
$\mathfrak{a}\nleqslant_{\mathrm{pdfto}}\mathfrak{b}$) to denote
the negation of $\mathfrak{a}\leqslant_{\mathrm{fto}}\mathfrak{b}$
($\mathfrak{a}\leqslant_{\mathrm{dfto}}\mathfrak{b}$; $\mathfrak{a}\leqslant_{\mathrm{pdfto}}\mathfrak{b}$).
\end{definition}

All injections are finite-to-one maps, and hence if $\mathfrak{a}\leqslant\mathfrak{b}$
then $\mathfrak{a}\leqslant_{\mathrm{fto}}\mathfrak{b}$. It is also obvious that
if $\mathfrak{a}\leqslant_{\mathrm{fto}}\mathfrak{b}$ then $\mathfrak{a}\leqslant_{\mathrm{pdfto}}\mathfrak{b}$,
and if $\mathfrak{a}\leqslant_{\mathrm{pdfto}}\mathfrak{b}$ then $\mathfrak{a}\leqslant_{\mathrm{dfto}}\mathfrak{b}$.
If $f$ is a finite-to-one map from $x$ into $y$ and $g$ is a finite-to-one map from $y$ into $z$,
then $g\circ f$ is a finite-to-one map from $x$ into $z$. Hence, if $\mathfrak{a}\leqslant_{\mathrm{fto}}\mathfrak{b}$
and $\mathfrak{b}\leqslant_{\mathrm{fto}}\mathfrak{c}$ then $\mathfrak{a}\leqslant_{\mathrm{fto}}\mathfrak{c}$.
The following three facts are Fact~2.8 and Corollaries~2.9~\&~2.11 of \cite{Shen2017}, respectively.

\begin{fact}\label{sh02}
If $f$ is a Dedekind finite (power Dedekind finite) to one map from $x$ into $y$,
and $g$ is a Dedekind finite (power Dedekind finite) to one map from $y$ into $z$,
then $g\circ f$ is a Dedekind finite (power Dedekind finite) to one map from $x$ into $z$.
Hence, if $\mathfrak{a}\leqslant_{\mathrm{dfto}}\mathfrak{b}$ and $\mathfrak{b}\leqslant_{\mathrm{dfto}}\mathfrak{c}$
then $\mathfrak{a}\leqslant_{\mathrm{dfto}}\mathfrak{c}$ (if $\mathfrak{a}\leqslant_{\mathrm{pdfto}}\mathfrak{b}$
and $\mathfrak{b}\leqslant_{\mathrm{pdfto}}\mathfrak{c}$ then $\mathfrak{a}\leqslant_{\mathrm{pdfto}}\mathfrak{c}$).
\end{fact}

\begin{fact}\label{sh03}
If $\mathfrak{a}$ is Dedekind infinite (power Dedekind infinite)
and $\mathfrak{a}\leqslant_{\mathrm{dfto}}\mathfrak{b}$
($\mathfrak{a}\leqslant_{\mathrm{pdfto}}\mathfrak{b}$) then
also $\mathfrak{b}$ is Dedekind infinite (power Dedekind infinite).
\end{fact}

\begin{fact}\label{sh04}
If $\mathfrak{a}^n$ is Dedekind infinite (power Dedekind infinite) then
also $\mathfrak{a}$ is Dedekind infinite (power Dedekind infinite).
\end{fact}

\subsection{Some special cardinals}
For a permutation $f$ of $x$, we write $\mov(f)$ for the set
$\{z\in x\mid f(z)\neq z\}$ (i.e., the elements of $x$ moved by $f$).

\begin{definition}
Let $x$ be an arbitrary set and let $\mathfrak{a}=|x|$.
\begin{enumerate}
\item $\cals(x)=\{f\mid f\text{ is a permutation of }x\}$;
      $\mathfrak{a}!=|\cals(x)|$.
\item $\calspdf(x)=\{f\in\cals(x)\mid\mov(f)\text{ is power Dedekind finite}\}$;\\
      $\calspdf(\mathfrak{a})=|\calspdf(x)|$.
\item $\calsf(x)=\{f\in\cals(x)\mid\mov(f)\text{ is finite}\}$;
      $\calsf(\mathfrak{a})=|\calsf(x)|$.
\item $\pdfin(x)=\{y\subseteq x\mid y\text{ is power Dedekind finite}\}$;\\
      $\pdfin(\mathfrak{a})=|\pdfin(x)|$.
\item $\fin(x)=\{y\subseteq x\mid y\text{ is finite}\}$;
      $\fin(\mathfrak{a})=|\fin(x)|$.
\item $\seq(x)=\{f\mid f\text{ is a function from some }n\in\omega\text{ into }x\}$;\\
      $\seq(\mathfrak{a})=|\seq(x)|$.
\item $\seq^{1\text{-}1}(x)=\{f\mid f\text{ is an injection from some }n\in\omega\text{ into }x\}$;\\
      $\seq^{1\text{-}1}(\mathfrak{a})=|\seq^{1\text{-}1}(x)|$.
\end{enumerate}
\end{definition}

Below we list some basic properties of these cardinals. We first note that
$\calsf(\mathfrak{a})\leqslant\calspdf(\mathfrak{a})\leqslant\mathfrak{a}!\leqslant\mathfrak{a}^\mathfrak{a}$
and that $\fin(\mathfrak{a})\leqslant^\ast\seq^{1\text{-}1}(\mathfrak{a})\leqslant\seq(\mathfrak{a})$.
The next two facts are Facts~2.13~\&~2.14 of \cite{Shen2017}, respectively.

\begin{fact}\label{sh05}
For all infinite cardinals $\mathfrak{a}$, both $\fin(\mathfrak{a})$ and $2^\mathfrak{a}$ are power Dedekind infinite.
\end{fact}

\begin{fact}\label{sh06}
If $\seq^{1\text{-}1}(\mathfrak{a})$ is Dedekind infinite then also $\mathfrak{a}$ is Dedekind infinite.
\end{fact}

\begin{fact}\label{sh10}
For every infinite cardinal $\mathfrak{a}$,
both $\calsf(\mathfrak{a})$ and $\mathfrak{a}!$ are power Dedekind infinite.
\end{fact}
\begin{proof}
Let $x$ be an arbitrary infinite set. Let $f$ be the function on $\calsf(x)$ such that
for all $t\in\calsf(x)$, $f(t)=0$, if $t=\mathrm{id}_x$, and $f(t)=|\mov(t)|-1$, otherwise.
Since $f$ is a surjection from $\calsf(x)$ onto $\omega$,
both $\calsf(x)$ and $\cals(x)$ are power Dedekind infinite sets.
\end{proof}

\begin{fact}\label{sh11}
For all power Dedekind finite cardinals $\mathfrak{a}$, $\mathfrak{a}!$ is Dedekind finite.
\end{fact}
\begin{proof}
By Fact~\ref{sh04}, $\mathfrak{a}^2$ is power Dedekind finite.
Since $\mathfrak{a}!\leqslant\mathfrak{a}^\mathfrak{a}\leqslant(2^\mathfrak{a})^\mathfrak{a}=2^{\mathfrak{a}^2}$,
$\mathfrak{a}!$ is Dedekind finite.
\end{proof}

\begin{fact}\label{sh01}
For all cardinals $\mathfrak{a}$, $\calsf(\mathfrak{a})\leqslant_{\mathrm{fto}}\fin(\mathfrak{a})$.
\end{fact}
\begin{proof}
For any set $x$, the function $g$ defined on $\calsf(x)$ given by
$g(t)=\mov(t)$ is a finite-to-one map from $\calsf(x)$ into $\fin(x)$.
\end{proof}

\begin{fact}\label{sh14}
For every set $x$, the function $f$ defined on $\calspdf(x)$ given by
$f(t)=\mov(t)$ is a Dedekind finite to one map from $\calspdf(x)$ into $\pdfin(x)$.
Hence for all cardinals $\mathfrak{a}$,
$\calspdf(\mathfrak{a})\leqslant_{\mathrm{dfto}}\pdfin(\mathfrak{a})$.
\end{fact}
\begin{proof}
Take an arbitrary $y\in\pdfin(x)$ and let $u=\{t\in\cals(x)\mid\mov(t)=y\}$.
It suffices to show that $u$ is Dedekind finite.
By Fact~\ref{sh11}, $\cals(y)$ is Dedekind finite.
Since the function $g$ defined on $u$ given by $g(t)=t\upharpoonright y$
is an injection from $u$ into $\cals(y)$, $u$ is also Dedekind finite.
\end{proof}

\begin{fact}\label{sh07}
For all cardinals $\mathfrak{a}$, $\aleph_0\cdot\mathfrak{a}\leqslant\seq(\mathfrak{a})$.
Hence for all non-zero cardinals $\mathfrak{a}$, $\seq(\mathfrak{a})$ is Dedekind infinite.
\end{fact}
\begin{proof}
For any set $x$, the function $g$ on $\omega\times x$ given by
$g(n,z)=(n+1)\times\{z\}$ is an injection from $\omega\times x$ into $\seq(x)$.
\end{proof}

\begin{fact}\label{sh08}
For all Dedekind finite cardinals $\mathfrak{a}$, $\seq(\mathfrak{a})\leqslant_{\mathrm{dfto}}\aleph_0$.
\end{fact}
\begin{proof}
For every Dedekind finite set $x$, by Fact~\ref{sh04}, $x^n$ is Dedekind finite for any $n\in\omega$,
and hence the function $f$ on $\seq(x)$ given by $f(t)=\dom(t)$
is a Dedekind finite to one map from $\seq(x)$ into $\omega$.
\end{proof}

\begin{lemma}\label{sh08a}
For all non-zero cardinals $\mathfrak{a}$, $\seq(\seq(\mathfrak{a}))=\seq(\mathfrak{a})$.
\end{lemma}
\begin{proof}
Cf.~\cite[Lemma~2]{Ellentuck1966}.
\end{proof}

\begin{lemma}\label{sh08b}
For any $\mathfrak{a}\neq0$,
$\seq(\mathfrak{a})=\seq^{1\text{-}1}(\mathfrak{a}+\aleph_0)=\aleph_0\cdot\seq^{1\text{-}1}(\mathfrak{a})$.
\end{lemma}
\begin{proof}
Let $x$ be a set disjoint from $\omega$ such that $|x|=\mathfrak{a}$.
Let $f$ be the function on $\seq(x)$ such that for all $t\in\seq(x)$,
$f(t)$ is the function defined on $\dom(t)$ given by
\[
f(t)(n)=
\begin{cases}
  t(n),                      & \text{if for all $k<n$, $t(k)\neq t(n)$;}\\
  \max\{k<n\mid t(k)=t(n)\}, & \text{otherwise.}
\end{cases}
\]
Clearly, for all $t\in\seq(x)$, $f(t)\in\seq^{1\text{-}1}(x\cup\omega)$.
Moreover, $f$ is injective, since for all $t\in\seq(x)$,
$t$ is recursively determined by $f(t)$ in the following way:
$\dom(t)=\dom(f(t))$, and for all $n\in\dom(t)$, $t(n)=f(t)(n)$,
if $f(t)(n)\in x$, and $t(n)=t(f(t)(n))$, if $f(t)(n)\in\omega$.
Hence $\seq(\mathfrak{a})\leqslant\seq^{1\text{-}1}(\mathfrak{a}+\aleph_0)$.

Obviously, $\seq(\aleph_0)=\aleph_0$, and hence there exists an injection $p$
from $\seq(\omega)\times\seq(\omega)\times\seq(\omega)$ into $\omega$.
Let $g$ and $h$ be functions on $\seq^{1\text{-}1}(x\cup\omega)$ such that
for all $t\in\seq^{1\text{-}1}(x\cup\omega)$,
$g(t)$ is the enumerating function of $t^{-1}[x]$ and
$h(t)$ is the enumerating function of $t^{-1}[\omega]$.
Then it is easy to verify that the function $u$
defined on $\seq^{1\text{-}1}(x\cup\omega)$ given by
\[
u(t)=\bigl(p\bigl(t\circ(h(t)),g(t),h(t)\bigr),t\circ(g(t))\bigr)
\]
is an injection from $\seq^{1\text{-}1}(x\cup\omega)$ into
$\omega\times\seq^{1\text{-}1}(x)$, which implies that
$\seq^{1\text{-}1}(\mathfrak{a}+\aleph_0)\leqslant\aleph_0\cdot\seq^{1\text{-}1}(\mathfrak{a})$.
Finally, by Fact~\ref{sh07} and Lemma~\ref{sh08a}, we have that
$\aleph_0\cdot\seq^{1\text{-}1}(\mathfrak{a})\leqslant
\aleph_0\cdot\seq(\mathfrak{a})\leqslant\seq(\seq(\mathfrak{a}))=\seq(\mathfrak{a})$.
\end{proof}

\begin{corollary}\label{sh09}
For all cardinals $\mathfrak{a}$, $\seq(\mathfrak{a})=\seq^{1\text{-}1}(\mathfrak{a})$
iff either $\mathfrak{a}=0$ or $\mathfrak{a}$ is Dedekind infinite.
\end{corollary}
\begin{proof}
If $\seq(\mathfrak{a})=\seq^{1\text{-}1}(\mathfrak{a})$ and $\mathfrak{a}\neq0$,
then, by Fact~\ref{sh07}, $\seq^{1\text{-}1}(\mathfrak{a})$ is Dedekind infinite,
and therefore, by Fact~\ref{sh06}, $\mathfrak{a}$ is also Dedekind infinite.
For the other direction, since $\seq(0)=\seq^{1\text{-}1}(0)$,
assume that $\mathfrak{a}$ is Dedekind infinite.
Let $x$ be a set disjoint from $\omega$ such that $|x|=\mathfrak{a}$,
and let $f$ be an injection from $\omega$ into $x$.
Let $g$ be the function on $x\cup\omega$ such that for all $z\in x\setminus\ran(f)$, $g(z)=z$,
and such that for all $n\in\omega$, $g(f(n))=f(2n)$ and $g(n)=f(2n+1)$.
Since $g$ is a bijection from $x\cup\omega$ onto $x$, $\mathfrak{a}+\aleph_0=\mathfrak{a}$
and therefore, by Lemma~\ref{sh08b},
$\seq(\mathfrak{a})=\seq^{1\text{-}1}(\mathfrak{a}+\aleph_0)=\seq^{1\text{-}1}(\mathfrak{a})$.
\end{proof}

\subsection{Some notation on permutations}
For a permutation $f$ of $x$, we define $f^{(n)}$ by recursion on $n$ as follows:
$f^{(0)}=\mathrm{id}_x$; $f^{(n+1)}=f\circ f^{(n)}$. Let $f\in\cals(x)$ and let $z\in x$.
The \emph{orbit} of $z$ under $f$, which we shall denote by $\orb(f,z)$,
is the countable set $\{v\in x\mid\exists n\in\omega(v=f^{(n)}(z)\text{ or }z=f^{(n)}(v))\}$.
An orbit $\orb(f,z)$ is said to be \emph{trivial} if $f(z)=z$;
otherwise it is \emph{non-trivial}.
Clearly, $\orb(f,z)$ is trivial if and only if $\orb(f,z)=\{z\}$.
It is also obvious that the orbits of $f$ form a partition of $x$
and the non-trivial orbits of $f$ form a partition of $\mov(f)$.

\begin{fact}\label{sh12}
Let $x$ be a set and let $\mathfrak{a}=|x|$.
For a permutation $f$ of $x$, let $\mathfrak{b}$ be the cardinality of the set of all non-trivial orbits of $f$.
Then $2^\mathfrak{b}\leqslant\mathfrak{a}!$.
\end{fact}
\begin{proof}
Let $y=\{\orb(f,z)\mid z\in\mov(f)\}$. Then $\mathfrak{b}=|y|$.
Let $g$ be the function on $\wp(y)$ such that for all $u\subseteq y$,
$g(u)$ is the permutation of $x$ given by
\[
g(u)(z)=
\begin{cases}
  f(z), & \text{if $z\in\bigcup u$;}\\
  z,    & \text{otherwise.}
\end{cases}
\]
It is easily verifiable that $g$ is an injection from $\wp(y)$ into $\cals(x)$.
\end{proof}

Let $f\in\cals(x)$ be such that all orbits of $f$ are finite and let $y\subseteq x$.
The permutation of $y$ \emph{induced} by $f$, which we shall denote by $f\rhd y$,
is defined as follows: For all $z\in y$, $(f\rhd y)(z)=f^{(n+1)}(z)$,
where $n$ is the least natural number such that $f^{(n+1)}(z)\in y$.
Note that for every $z\in y$, $\orb(f\rhd y,z)=\orb(f,z)\cap y$.
Note also that all orbits of a permutation in $\calspdf(x)$ are finite.

For $t\in\seq^{1\text{-}1}(x)$, we use $(t(0);\dots;t(n-1))_x$, where $n=\dom(t)$,
to denote the permutation of $x$ which moves $t(0)$ to $t(1)$, $t(1)$ to $t(2)$, \dots,
$t(n-2)$ to $t(n-1)$, and $t(n-1)$ to $t(0)$, and fixes all other elements of $x$.
In particular, for two distinct elements $z$, $v$ of $x$,
$(z;v)_x$ is the transposition that interchanges $z$ and $v$.

\begin{fact}\label{sh13}
For all cardinals $\mathfrak{a}$,
$\seq^{1\text{-}1}(\mathfrak{a})\leqslant_{\mathrm{fto}}\calsf(\mathfrak{a})$
and $\seq(\mathfrak{a})\leqslant\aleph_0\cdot\calsf(\mathfrak{a})$.
\end{fact}
\begin{proof}
If $\mathfrak{a}=0$ then $\calsf(\mathfrak{a})=\seq^{1\text{-}1}(\mathfrak{a})=\seq(\mathfrak{a})=1$.
Otherwise, let $x$ be a set such that $|x|=\mathfrak{a}$ and let $z\in x$.
Let $f$ be the function on $\seq^{1\text{-}1}(x)$ such that
for all $t\in\seq^{1\text{-}1}(x)$, $f(t)=(t(0);\dots;t(n-1))_x$, if $z\in\ran(t)$,
and $f(t)=(t(0);\dots;t(n-1);z)_x$, otherwise, where $n=\dom(t)$.
Then $f$ is a finite-to-one map from $\seq^{1\text{-}1}(x)$ into $\calsf(x)$,
and hence $\seq^{1\text{-}1}(\mathfrak{a})\leqslant_{\mathrm{fto}}\calsf(\mathfrak{a})$.
Let $g$ be the function defined on $\seq^{1\text{-}1}(x)$ given by
\[
g(t)=
\begin{cases}
  (t^{-1}(z),f(t)), & \text{if $z\in\ran(t)$;}\\
  (\dom(t)+1,f(t)), & \text{otherwise.}
\end{cases}
\]
Then it is easily verifiable that $g$ an injection
from $\seq^{1\text{-}1}(x)$ into $\omega\times\calsf(x)$, which implies that
$\seq^{1\text{-}1}(\mathfrak{a})\leqslant\aleph_0\cdot\calsf(\mathfrak{a})$.
Now, by Lemma~\ref{sh08b}, we get that
$\seq(\mathfrak{a})=\aleph_0\cdot\seq^{1\text{-}1}(\mathfrak{a})\leqslant\aleph_0\cdot\calsf(\mathfrak{a})$.
\end{proof}

\section{Permutations that move power Dedekind finitely many elements}
In this section, we prove our first main result that for all cardinals $\mathfrak{a}$,
if $\mathfrak{a}!$ is Dedekind infinite, then $\mathfrak{a}!\nleqslant_{\mathrm{dfto}}\calspdf(\mathfrak{a})$.
Our proof is based on ideas in \cite{Shen2017}, which are originally from \cite{Specker1954}.
The strategy is as follows:

Assume towards a contradiction that there exists a set $x$ such that there is a Dedekind finite to one map
from $\cals(x)$ into $\calspdf(x)$ and such that there is an injection $h$ from $\omega$ into $\cals(x)$.
We first prove a kind of Cantor's theorem for $\cals(x)$, and then, by transfinitely iterating this theorem,
we extend $h$ to an injection $H$ from $\mathrm{Ord}$
(i.e., the proper class of all ordinals) into $\cals(x)$, which is a contradiction.

\begin{theorem}\label{ctfs}
From functions $f:x\to\cals(x)$ and $g:\ran(f)\to\cals(x)$
such that for all $t\in\ran(f)$, $\emptyset\neq\mov(g(t))\subseteq f^{-1}[\{t\}]$,
one can explicitly define a $u\in\cals(x)\setminus\ran(f)$.
\end{theorem}
\begin{proof}
We use Cantor's diagonal construction:
Let $u$ be the permutation of $\dom(f)$ such that for all $t\in\ran(f)$ and all $z\in f^{-1}[\{t\}]$,
\[
u(z)=
\begin{cases}
  z,       & \text{if $t\upharpoonright f^{-1}[\{t\}]=g(t)\upharpoonright f^{-1}[\{t\}]$;}\\
  g(t)(z), & \text{otherwise.}
\end{cases}
\]
For all $t\in\ran(f)$, $u\upharpoonright f^{-1}[\{t\}]\neq t\upharpoonright f^{-1}[\{t\}]$,
and hence $u\notin\ran(f)$.
\end{proof}

\begin{lemma}\label{sh15}
From an infinite ordinal $\alpha$, one can explicitly define an injection $f:\fin(\alpha)\to\alpha$.
\end{lemma}
\begin{proof}
Cf.~\cite[Theorem~5.19]{Halbeisen2017}.
\end{proof}

\begin{lemma}\label{sh16}
From a finite-to-one map $f:\alpha\to x$, where $\alpha$ is an infinite ordinal,
one can explicitly define an injection $g:\alpha\to x$.
\end{lemma}
\begin{proof}
Cf.~\cite[Lemma~3.3]{Shen2017}.
\end{proof}

In \cite[Theorem~2.2]{VejjajivaPanasawatwong2014}, Vejjajiva and Panasawatwong proved a lemma
which states that from a set $x$ and an injection $f:\alpha\to\pdfin(x)$, where $\alpha$ is
an infinite ordinal, one can explicitly define a surjection $g:x\twoheadrightarrow\alpha$.
This lemma was originally proved by Halbeisen and Shelah
(cf.~\cite[Theorem~3]{HalbeisenShelah1994}) for $\fin(x)$.
The key step of our proof is a corresponding lemma for $\calspdf(x)$:

\begin{lemma}\label{sh17}
From an injection $f:\alpha\to\calspdf(x)$, where $\alpha$ is an infinite ordinal,
one can explicitly define a pair of functions $(g,h)$ such that $g$ is a surjection from $x$ onto $\alpha$,
$h$ is a function from $\alpha$ into $\cals(x)$, and for all $\beta<\alpha$,
$\emptyset\neq\mov(h(\beta))\subseteq g^{-1}[\{\beta\}]$.
\end{lemma}
\begin{proof}
Let $\alpha$ be an infinite ordinal and let $f$ be an injection from $\alpha$ into $\calspdf(x)$.
Since $\alpha=\dom(f)$ and $x=\dom(f(0))$,
it suffices to explicitly define such a pair $(g,h)$ from $\alpha$, $x$, $f$.

Let $\Phi$ be the function defined on $\alpha$ given by $\Phi(\beta)=\mov(f(\beta))$.
Then by Facts~\ref{sh02}~\&~\ref{sh14}, $\Phi$ is a Dedekind finite to one map from $\alpha$ into $\pdfin(x)$.
Since all Dedekind finite subsets of $\alpha$ are finite,
$\Phi$ is finite-to-one.

Let $\sim$ be the equivalence relation on $x$ such that for all $z$, $v\in x$,
\[
z\sim v\text{ if and only if }\forall\beta<\alpha\bigl(z\in\Phi(\beta)\leftrightarrow v\in\Phi(\beta)\bigr).
\]
Let $\Pi=\{[z]_\sim\mid z\in\bigcup_{\beta<\alpha}\Phi(\beta)\}$,
where $[z]_\sim$ is the equivalence class of~$z$ with respect to the equivalence relation $\sim$.

\begin{xclm}
We can explicitly define a bijection $\Omega:\alpha\twoheadrightarrow\Pi$.
\end{xclm}
\begin{proof}[Proof of Claim]
For each $z\in\bigcup_{\beta<\alpha}\Phi(\beta)$, let $\eta_z=\min\{\beta<\alpha\mid z\in\Phi(\beta)\}$.
Let $\Psi$ be the function defined on $\bigcup_{\beta<\alpha}\Phi(\beta)$ given by
\[
\Psi(z)=\bigl\{\gamma<\alpha\bigm|z\in\Phi(\gamma)\wedge
\forall\beta<\gamma\bigl(\Phi(\eta_z)\cap\Phi(\beta)\neq\Phi(\eta_z)\cap\Phi(\gamma)\bigr)\bigr\}.
\]
Note that for all $z\in\bigcup_{\beta<\alpha}\Phi(\beta)$, $\eta_z=\min\Psi(z)$.
For every $z\in\bigcup_{\beta<\alpha}\Phi(\beta)$,
since the function that maps each $\gamma\in\Psi(z)$ to $\Phi(\eta_z)\cap\Phi(\gamma)$
is an injection from $\Psi(z)$ into $\wp(\Phi(\eta_z))$ and $\Phi(\eta_z)$ is power Dedekind finite,
$\Psi(z)$ is a finite subset of $\alpha$. We claim that for all $z$, $v\in\bigcup_{\beta<\alpha}\Phi(\beta)$,
\begin{equation}\label{sh18}
z\sim v\text{ if and only if }\Psi(z)=\Psi(v).
\end{equation}
In fact, obviously, if $z\sim v$ then $\eta_z=\eta_v$ and $\Psi(z)=\Psi(v)$.
For the other direction, assume that $\Psi(z)=\Psi(v)$.
Then $\eta_z=\min\Psi(z)=\min\Psi(v)=\eta_v$. Take an arbitrary $\beta<\alpha$.
Let $\gamma=\min\{\delta<\alpha\mid\Phi(\eta_z)\cap\Phi(\delta)=\Phi(\eta_z)\cap\Phi(\beta)\}$.
If $z\in\Phi(\beta)$, then $z\in\Phi(\gamma)$ and hence $\gamma\in\Psi(z)=\Psi(v)$,
which implies that $v\in\Phi(\gamma)$ and $v\in\Phi(\beta)$.
Similarly, if $v\in\Phi(\beta)$ then $z\in\Phi(\beta)$. Hence $z\sim v$.

Now, by \eqref{sh18}, the function
$\Lambda=\{([z]_\sim,\Psi(z))\mid z\in\bigcup_{\beta<\alpha}\Phi(\beta)\}$
is an injection from $\Pi$ into $\fin(\alpha)$.
By Lemma~\ref{sh15}, we can explicitly define an injection $p:\fin(\alpha)\to\alpha$.
Let $r$ be the well-ordering of $\Pi$ induced by $p\circ\Lambda$;
that is, $r=\{(c,d)\mid c,d\in\Pi\text{ and }p(\Lambda(c))\in p(\Lambda(d))\}$.
Let $\theta$ be the order type of $\langle\Pi,r\rangle$,
and let $\Theta$ be the unique isomorphism of $\langle\theta,\in\rangle$ onto $\langle\Pi,r\rangle$.
Then $p\circ\Lambda\circ\Theta$ is an injection from $\theta$ into $\alpha$.

Let $\Xi$ be the function on $\alpha$ given by
$\Xi(\beta)=\{\delta<\theta\mid\Theta(\delta)\subseteq\Phi(\beta)\}$.
For every $\beta<\alpha$, since $\Phi(\beta)$ is power Dedekind finite, $\Xi(\beta)\in\fin(\theta)$.
Since $\Phi(\beta)=\bigcup_{\delta\in\Xi(\beta)}\Theta(\delta)$ for any $\beta<\alpha$
and $\Phi$ is finite-to-one, $\Xi$ is a finite-to-one map from $\alpha$ into $\fin(\theta)$.
Then by Lemma~\ref{sh16}, we can explicitly define an injection $t:\alpha\to\fin(\theta)$.
By Lemma~\ref{sh15}, we can explicitly define an injection $q:\fin(\theta)\to\theta$.
Then $q\circ t$ is an injection from $\alpha$ into $\theta$.

Therefore, by Theorem~\ref{cbt}, we can explicitly define a bijection $u:\alpha\twoheadrightarrow\theta$.
Now the function $\Omega=\Theta\circ u$ is a bijection from $\alpha$ onto $\Pi$.
\end{proof}

Now we turn back to the construction of $(g,h)$.
For each $\gamma<\alpha$, let $\xi_\gamma=\min\{\beta<\alpha\mid\Omega(\gamma)\subseteq\Phi(\beta)\}$, and let
\[
\zeta_\gamma=\min\bigl\{\delta<\alpha\bigm|\Omega(\delta)\subseteq\Phi(\xi_\gamma)
\wedge\exists z\in\Omega(\gamma)\bigl(f(\xi_\gamma)(z)\in\Omega(\delta)\bigr)\bigr\}.
\]
Then we have
\begin{equation}\label{sh19}
\forall\gamma<\alpha\bigl(\mov\bigl(f(\xi_\gamma)\rhd(\Omega(\gamma)\cup\Omega(\zeta_\gamma))\bigr)\neq\emptyset\bigr).
\end{equation}
We define by recursion a function $F$ on an initial segment of $\mathrm{Ord}$ as follows:
\[
F(\beta)=\min\bigl\{\gamma<\alpha\bigm|\{\gamma,\zeta_\gamma\}\cap
\bigcup_{\delta<\beta}\{F(\delta),\zeta_{F(\delta)}\}=\emptyset\bigr\},
\]
as long as such a $\gamma<\alpha$ exists.
Clearly, $F$ is an injection from some ordinal into $\alpha$.
Let $A=\ran(F)$ and let $B=\bigcup_{\gamma\in A}\{\gamma,\zeta_\gamma\}$.
Then we have
\begin{equation}\label{sh20}
\forall\gamma,\delta\in A\bigl(\gamma\neq\delta\rightarrow
\{\gamma,\zeta_\gamma\}\cap\{\delta,\zeta_\delta\}=\emptyset\bigr).
\end{equation}
There cannot be a $\gamma\in\alpha\setminus B$ such that $\zeta_\gamma\notin B$,
since otherwise the least such $\gamma$ would be in the range of $F$, which is a contradiction.
Therefore
\begin{equation}\label{sh21}
\forall\gamma\in\alpha\setminus B\bigl(\zeta_\gamma\in B\bigr).
\end{equation}

Let $\asymp$ be the irreflexive and symmetric relation on $\alpha\setminus B$
such that for all $\gamma$, $\delta\in\alpha\setminus B$, $\gamma\asymp\delta$ if and only if
\[
\gamma\neq\delta\wedge\zeta_\gamma=\zeta_\delta\wedge\exists z\in\Omega(\gamma)\exists v\in\Omega(\delta)
\bigl(f(\xi_\gamma)(z)=f(\xi_\delta)(v)\in\Omega(\zeta_\gamma)\bigr).
\]
Let $\Gamma$ be the function defined on
$\{(\gamma,\delta)\mid\gamma,\delta\in\alpha\setminus B\text{ and }\gamma\asymp\delta\}$ given by
\[
\Gamma(\gamma,\delta)=\bigl\{(z,v)\in\Omega(\gamma)\times\Omega(\delta)
\bigm|f(\xi_\gamma)(z)=f(\xi_\delta)(v)\in\Omega(\zeta_\gamma)\bigr\}.
\]
Note that for all $\gamma$, $\delta\in\alpha\setminus B$,
if $\gamma\asymp\delta$ then $\Gamma(\gamma,\delta)$ is a non-empty injection
from a subset of $\Omega(\gamma)$ into $\Omega(\delta)$ and
the inverse of $\Gamma(\gamma,\delta)$ is $\Gamma(\delta,\gamma)$.
We define, by mutual recursion, two functions $G$ and $H$ on an initial segment of $\mathrm{Ord}$ as follows:
\begin{align*}
   G(\beta)=\min\bigl\{\gamma\in\alpha\setminus B & \bigm|\exists\delta\in\alpha\setminus B
            \bigl(\gamma\asymp\delta\wedge\{\gamma,\delta\}\cap(G[\beta]\cup H[\beta])=\emptyset\bigr)\bigr\};\\
   H(\beta)=\min\bigl\{\delta\in\alpha\setminus B & \bigm|G(\beta)\asymp\delta\wedge
            \{G(\beta),\delta\}\cap(G[\beta]\cup H[\beta])=\emptyset\bigr\},
\end{align*}
as long as
\[
\exists\gamma\in\alpha\setminus B\exists\delta\in\alpha\setminus B
\bigl(\gamma\asymp\delta\wedge\{\gamma,\delta\}\cap(G[\beta]\cup H[\beta])=\emptyset\bigr).
\]
Clearly, $G$ and $H$ are injections from some ordinal into
$\alpha\setminus B$ such that $\ran(G)\cap\ran(H)=\emptyset$.
Let $C=\ran(G)$ and let $D=\ran(G)\cup\ran(H)$.
For each $\gamma\in C$, let $\rho_\gamma=H(G^{-1}(\gamma))$.
Then we have that $D=\bigcup_{\gamma\in C}\{\gamma,\rho_\gamma\}$ and
\begin{equation}\label{sh22}
\forall\gamma\in C\bigl(\gamma\asymp\rho_\gamma\wedge\forall\delta\in C
\bigl(\gamma\neq\delta\rightarrow\{\gamma,\rho_\gamma\}\cap\{\delta,\rho_\delta\}=\emptyset\bigr)\bigr).
\end{equation}
There cannot be a $\gamma\in\alpha\setminus(B\cup D)$ such that
$\gamma\asymp\delta$ for some $\delta\in\alpha\setminus(B\cup D)$,
since otherwise the least such $\gamma$ would be in the range of $G$,
which is a contradiction. Therefore
\begin{equation}\label{sh23}
\forall\gamma,\delta\in\alpha\setminus(B\cup D)\bigl(\neg\,\gamma\asymp\delta\bigr).
\end{equation}

Now, in view of \eqref{sh20}, \eqref{sh22}, and \eqref{sh21},
we define a function $\Delta$ on $\alpha$ by setting, for $\beta<\alpha$,
\[
\Delta(\beta)=
\begin{cases}
  \beta,                                                                           & \text{if $\beta\in A\cup C$;}\\
  \text{the unique $\gamma\in A$ such that }\beta=\zeta_\gamma,                    & \text{if $\beta\in B\setminus A$;}\\
  \text{the unique $\gamma\in C$ such that }\beta=\rho_\gamma,                     & \text{if $\beta\in D\setminus C$;}\\
  \text{the unique $\gamma\in A$ such that }\zeta_\beta\in\{\gamma,\zeta_\gamma\}, & \text{if $\beta\notin B\cup D$.}
\end{cases}
\]
We claim that
\begin{equation}\label{sh24}
\Delta\text{ is a finite-to-one map from $\alpha$ into }A\cup C.
\end{equation}
Clearly, it suffices to show that for all $\gamma\in A$,
$\{\beta\in\alpha\setminus(B\cup D)\mid\zeta_\beta\in\{\gamma,\zeta_\gamma\}\}$ is finite.
For this purpose, in turn, it suffices to show that for all $\delta\in B$,
$\{\beta\in\alpha\setminus(B\cup D)\mid\zeta_\beta=\delta\}$ is finite.
Take an arbitrary $\delta\in B$, and let
$u=\{\beta\in\alpha\setminus(B\cup D)\mid\zeta_\beta=\delta\}$.
Let $t$ be the function defined on $u$ given by
\[
t(\beta)=\bigl\{w\in\Omega(\delta)\bigm|\exists z\in\Omega(\beta)\bigl(f(\xi_\beta)(z)=w\bigr)\bigr\}.
\]
Then for all $\beta\in u$, by the definition of $\zeta_\beta$,
$t(\beta)$ is a non-void subset of $\Omega(\delta)$.
We show that
\begin{equation}\label{sh25}
\forall\beta,\gamma\in u\bigl(\beta\neq\gamma\rightarrow t(\beta)\cap t(\gamma)=\emptyset\bigr).
\end{equation}
Assume towards a contradiction that for two distinct elements $\beta$, $\gamma$ of $u$,
$t(\beta)\cap t(\gamma)\neq\emptyset$. Let $w\in t(\beta)\cap t(\gamma)$.
Then we have $\exists z\in\Omega(\beta)(f(\xi_\beta)(z)=w)$
and $\exists v\in\Omega(\gamma)(f(\xi_\gamma)(v)=w)$, and therefore $\beta\asymp\gamma$,
contradicting \eqref{sh23}. Thus \eqref{sh25} is proved, and therefore
$t$ is an injection from $u$ into $\wp(\Omega(\delta))$.
Since $\Omega(\delta)$ is power Dedekind finite,
we get that $u$ is a finite subset of $\alpha$,
and thus \eqref{sh24} is proved.

By \eqref{sh24} and Lemma~\ref{sh16},
we can explicitly define an injection $q$ from $\alpha$ into $A\cup C$.
Let $p$ be the function defined on $\alpha$ given by
\[
p(\beta)=
\begin{cases}
  \Omega(q(\beta))\cup\Omega(\zeta_{q(\beta)}), & \text{if $q(\beta)\in A$;}\\
  \Omega(q(\beta))\cup\Omega(\rho_{q(\beta)}),  & \text{if $q(\beta)\in C$.}
\end{cases}
\]
Then for all $\beta<\alpha$, $p(\beta)$ is a non-void subset of $x$ and,
by \eqref{sh20} and \eqref{sh22},
\[
\forall\beta,\gamma<\alpha\bigl(\beta\neq\gamma\rightarrow p(\beta)\cap p(\gamma)=\emptyset\bigr).
\]

Now we define functions $g$ and $h$ as follows:
Define $g$ to be the function on $x$ such that for all $z\in x$,
\[
g(z)=
\begin{cases}
  \text{the unique $\beta<\alpha$ such that }z\in p(\beta), & \text{if $z\in\bigcup\ran(p)$;}\\
  0,                                                        & \text{otherwise.}
\end{cases}
\]
Then $g$ is a surjection from $x$ onto $\alpha$ such that for all $\beta<\alpha$,
$p(\beta)\subseteq g^{-1}[\{\beta\}]$.
Define $h$ to be the function on $\alpha$ such that for all $\beta<\alpha$,
if $q(\beta)\in A$, then $h(\beta)$ is the permutation of $x$ given by
\[
h(\beta)(z)=
\begin{cases}
  (f(\xi_{q(\beta)})\rhd p(\beta))(z), & \text{if $z\in p(\beta)$;}\\
  z,                                   & \text{otherwise,}
\end{cases}
\]
and if $q(\beta)\in C$, then $h(\beta)$ is the permutation of $x$ given by
\[
h(\beta)(z)=
\begin{cases}
  \Gamma(q(\beta),\rho_{q(\beta)})(z), & \text{if $z\in\dom(\Gamma(q(\beta),\rho_{q(\beta)}))$;}\\
  \Gamma(\rho_{q(\beta)},q(\beta))(z), & \text{if $z\in\ran(\Gamma(q(\beta),\rho_{q(\beta)}))$;}\\
  z,                                   & \text{otherwise.}
\end{cases}
\]
Then for all $\beta<\alpha$ such that $q(\beta)\in A$, by \eqref{sh19},
\[
\emptyset\neq\mov(h(\beta))\subseteq p(\beta)\subseteq g^{-1}[\{\beta\}].
\]
On the other hand, for all $\beta<\alpha$ such that $q(\beta)\in C$, by \eqref{sh22},
$q(\beta)\asymp\rho_{q(\beta)}$, and hence, by the definition of $\Gamma$,
$\Gamma(q(\beta),\rho_{q(\beta)})$ is a non-empty injection
from a subset of $\Omega(q(\beta))$ into $\Omega(\rho_{q(\beta)})$ and
the inverse of $\Gamma(q(\beta),\rho_{q(\beta)})$ is $\Gamma(\rho_{q(\beta)},q(\beta))$,
which implies that
\[
\emptyset\neq\mov(h(\beta))\subseteq p(\beta)\subseteq g^{-1}[\{\beta\}].
\]
To sum up, $h$ is a function from $\alpha$ into $\cals(x)$ such that for all $\beta<\alpha$,
$\emptyset\neq\mov(h(\beta))\subseteq g^{-1}[\{\beta\}]$, which completes the proof.
\end{proof}

Now we are ready to prove our first main theorem.

\begin{theorem}\label{tspdf}
For all cardinals $\mathfrak{a}$, if $\mathfrak{a}!$ is Dedekind infinite,
then we have $\mathfrak{a}!\nleqslant_{\mathrm{dfto}}\calspdf(\mathfrak{a})$,
and hence $\calspdf(\mathfrak{a})<\mathfrak{a}!$.
\end{theorem}
\begin{proof}
Assume towards a contradiction that there exists a cardinal $\mathfrak{a}$
such that $\mathfrak{a}!$ is Dedekind infinite and
such that $\mathfrak{a}!\leqslant_{\mathrm{dfto}}\calspdf(\mathfrak{a})$.
Let $x$ be a set such that $|x|=\mathfrak{a}$.
Let $h$ be an injection from $\omega$ into $\cals(x)$,
and let $\Phi$ be a Dedekind finite to one map from $\cals(x)$ into $\calspdf(x)$.
In what follows, we get a contradiction by constructing by recursion
an injection $H$ from the proper class $\mathrm{Ord}$ into the set $\cals(x)$.

For $n\in\omega$, we set $H(n)=h(n)$.
Now, we assume that $\alpha$ is an infinite ordinal and that
$H\upharpoonright\alpha$ is an injection from $\alpha$ into $\cals(x)$.
By Fact~\ref{sh02}, $\Phi\circ(H\upharpoonright\alpha)$ is
a Dedekind finite to one map from $\alpha$ into $\calspdf(x)$.
Since all Dedekind finite subsets of $\alpha$ are finite,
$\Phi\circ(H\upharpoonright\alpha)$ is finite-to-one.
Then by Lemma~\ref{sh16}, $\Phi\circ(H\upharpoonright\alpha)$
explicitly provides an injection $f:\alpha\to\calspdf(x)$.
By Lemma~\ref{sh17}, from $f$, we can explicitly define a pair of functions $(g,p)$
such that $g$ is a surjection from $x$ onto $\alpha$,
$p$ is a function from $\alpha$ into $\cals(x)$, and for all $\beta<\alpha$,
$\emptyset\neq\mov(p(\beta))\subseteq g^{-1}[\{\beta\}]$.
Therefore, $(H\upharpoonright\alpha)\circ g$ is a surjection from $x$ onto $H[\alpha]$,
and $p\circ(H\upharpoonright\alpha)^{-1}$ is a function from $H[\alpha]$ into~$\cals(x)$
such that for all $t\in H[\alpha]$,
\[
\emptyset\neq\mov\bigl(\bigl(p\circ(H\upharpoonright\alpha)^{-1}\bigr)(t)\bigr)
\subseteq\bigl((H\upharpoonright\alpha)\circ g\bigr)^{-1}[\{t\}].
\]
Therefore, by Theorem~\ref{ctfs}, we can explicitly define an $H(\alpha)\in\cals(x)\setminus H[\alpha]$
from $H\upharpoonright\alpha$ (and $\Phi$).
\end{proof}

\begin{corollary}\label{sh26}
For all sets $x$, if $\calspdf(x)\neq\cals(x)$ then $|\calspdf(x)|<|\cals(x)|$.
\end{corollary}
\begin{proof}
Let $x$ be a set such that $\calspdf(x)$ is a proper subset of $\cals(x)$.
Assume towards a contradiction that $|\calspdf(x)|=|\cals(x)|$.
Then by Theorem~\ref{dedt}, $\cals(x)$ is Dedekind infinite,
and hence, by Theorem~\ref{tspdf}, $|\calspdf(x)|<|\cals(x)|$,
which is a contradiction.
\end{proof}

In fact, even without assuming $|\calspdf(x)|=|\cals(x)|$,
$\calspdf(x)\neq\cals(x)$ implies that $\cals(x)$ is Dedekind infinite,
as shown in the following theorem, which is a kind of Kuratowski's theorem for $\cals(x)$.

\begin{theorem}\label{ktfs}
For all sets $x$, $\cals(x)$ is Dedekind infinite iff $\calspdf(x)\neq\cals(x)$
(i.e., there exists a permutation of $x$ which moves power Dedekind infinitely many elements)
iff $\wp(\omega)\preccurlyeq\cals(x)$.
\end{theorem}
\begin{proof}
Suppose that $\cals(x)$ is Dedekind infinite.
Assume towards a contradiction that $\calspdf(x)=\cals(x)$.
Let $f$ be an injection from $\omega$ into $\calspdf(x)$.
By Lemma~\ref{sh17}, there are functions $g$ and $h$ such that
$g$ is a surjection from $x$ onto $\omega$,
$h$ is a function from $\omega$ into $\cals(x)$,
and for all $n\in\omega$, $\emptyset\neq\mov(h(n))\subseteq g^{-1}[\{n\}]$.
Then the permutation $u$ of $x$ given by
\[
u(z)=h(g(z))(z)
\]
moves power Dedekind infinitely many elements, which is a contradiction.
Therefore, if $\cals(x)$ is Dedekind infinite, then $\calspdf(x)\neq\cals(x)$.

Now we suppose that there is a permutation $t$ of $x$
which moves power Dedekind infinitely many elements.
If there exists a $z\in x$ such that $\orb(t,z)$ is denumerable,
then $x$ is Dedekind infinite, and hence
$\wp(\omega)\approx\cals(\omega)\preccurlyeq\cals(x)$.
Otherwise, all orbits of $t$ are finite.
Let $y=\{\orb(t,z)\mid z\in\mov(t)\}$.
Since the function $q$ defined on $\mov(t)$ given by
$q(z)=\orb(t,z)$ is a finite-to-one map from $\mov(t)$ into $y$
and $\mov(t)$ is power Dedekind infinite, by Fact~\ref{sh03},
$y$ is also power Dedekind infinite.
Thus, by Theorem~\ref{kurt} and Fact~\ref{sh12},
$\wp(\omega)\preccurlyeq\wp(y)\preccurlyeq\cals(x)$.
Therefore $\calspdf(x)\neq\cals(x)$ implies that
$\wp(\omega)\preccurlyeq\cals(x)$, which completes the proof.
\end{proof}

\subsection{Some further results}
Next, we develop some further properties of~$\calspdf(\mathfrak{a})$.
By Theorem~\ref{ktfs}, for all cardinals $\mathfrak{a}$,
if $\mathfrak{a}!$ is Dedekind infinite, then $2^{\aleph_0}\leqslant\mathfrak{a}!$.
The following theorem is a generalization of this result.

\begin{theorem}\label{sh48}
For all cardinals $\mathfrak{a}$, if $\mathfrak{a}!$ is Dedekind infinite,
then we have $2^{\aleph_0}\cdot\calspdf(\mathfrak{a})\leqslant\mathfrak{a}!$.
\end{theorem}
\begin{proof}
Let $x$ be a set such that $|x|=\mathfrak{a}$.
Since $\cals(x)$ is Dedekind infinite,
by Theorem~\ref{ktfs}, there exists a permutation $f$ of $x$
which moves power Dedekind infinitely many elements.
We claim that there is a permutation $g$ of $x$
which has power Dedekind infinitely many non-trivial orbits.
In fact, if all orbits of $f$ are finite, then,
since the function $q$ defined on $\mov(f)$ given by $q(z)=\orb(f,z)$
is finite-to-one and $\mov(f)$ is power Dedekind infinite, by Fact~\ref{sh03},
$\{\orb(f,z)\mid z\in\mov(f)\}$ is power Dedekind infinite,
and hence it suffices to take $g=f$.
Otherwise, there exists a $z\in x$ such that $\orb(f,z)$ is denumerable,
and therefore $x$ is Dedekind infinite.
Hence, if $p$ is an injection from $\omega$ into $x$,
it suffices to take $g$ to be the permutation of $x$
which interchanges $p(2n)$ and $p(2n+1)$ for all $n\in\omega$
and which fixes all other elements of $x$.

Let $y$ be the set of all non-trivial orbits of $g$.
Clearly, for all $u\subseteq y$, $g\upharpoonright\bigcup u$
is a permutation of $\bigcup u$ without fixed points.
Since $y$ is power Dedekind infinite, by Theorem~\ref{kurt},
there is a surjection $h:y\twoheadrightarrow\omega\times\omega\times\omega$.
For each $t\in\calspdf(x)$, let
\[
m_t=\min\bigl\{n\in\omega\bigm|h\bigl[\{w\in y\mid\mov(t)\cap w\neq\emptyset\}\bigr]
\cap\bigl(\omega\times\{n\}\times\omega\bigr)=\emptyset\bigr\}.
\]
Such an $n\in\omega$ exists, since $\{w\in y\mid\mov(t)\cap w\neq\emptyset\}$
is power Dedekind finite and its image under $h$ is finite.

Now, let $\Phi$ be the function on $\wp(\omega)\times\calspdf(x)$ such that
for all $a\subseteq\omega$ and all $t\in\calspdf(x)$,
$\Phi(a,t)$ is the permutation of $x$ given by
\[
\Phi(a,t)(z)=
\begin{cases}
  t(z), & \text{if $z\in\mov(t)$;}\\
  g(z), & \text{if $z\in\bigcup h^{-1}[a\times\{m_t\}\times\omega]$;}\\
  z,    & \text{otherwise.}
\end{cases}
\]
Note that for all $a\subseteq\omega$ and all $t\in\calspdf(x)$,
$\mov(\Phi(a,t))$ is the union of $\mov(t)$ and $\bigcup h^{-1}[a\times\{m_t\}\times\omega]$.
Hence, for all $a\subseteq\omega$ and all $t\in\calspdf(x)$,
if $\Phi(a,t)\in\calspdf(x)$ then $a=\varnothing$ and $t=\Phi(a,t)$, and otherwise
\[
a=\bigl\{i\in\omega\bigm|\textstyle\bigcup h^{-1}[\{i\}\times\{k\}\times\omega]\subseteq\mov(\Phi(a,t))\bigr\}
\]
and $t$ is the permutation of $x$ given by
\[
t(z)=
\begin{cases}
  \Phi(a,t)(z), & \text{if $z\in\mov(\Phi(a,t))\setminus\bigcup h^{-1}[\omega\times\{k\}\times\omega]$;}\\
  z,            & \text{otherwise,}
\end{cases}
\]
where $k$ is the unique $n\in\omega$ such that
the intersection of $\mov(\Phi(a,t))$ and
$\bigcup h^{-1}[\omega\times\{n\}\times\omega]$ is power Dedekind infinite.
Therefore, $\Phi$ is an injection from $\wp(\omega)\times\calspdf(x)$ into $\cals(x)$,
and hence $2^{\aleph_0}\cdot\calspdf(\mathfrak{a})\leqslant\mathfrak{a}!$.
\end{proof}

\begin{corollary}\label{sh49}
For all cardinals $\mathfrak{a}$, if $\mathfrak{a}!$ is Dedekind infinite,
then we have $2^{\aleph_0}\cdot\seq(\mathfrak{a})\leqslant\mathfrak{a}!$.
\end{corollary}
\begin{proof}
For all cardinals $\mathfrak{a}$, if $\mathfrak{a}!$ is Dedekind infinite,
then, by Fact~\ref{sh13} and Theorem~\ref{sh48},
$2^{\aleph_0}\cdot\seq(\mathfrak{a})\leqslant2^{\aleph_0}\cdot\calsf(\mathfrak{a})
\leqslant2^{\aleph_0}\cdot\calspdf(\mathfrak{a})\leqslant\mathfrak{a}!$.
\end{proof}

\begin{lemma}\label{sh43}
From two permutations $f$, $g\in\calspdf(x)$, one can explicitly define
a permutation $h\in\calspdf(x)$ such that $\mov(h)=\mov(f)\cup\mov(g)$.
\end{lemma}
\begin{proof}
Let $f$, $g\in\calspdf(x)$. Let $y=\mov(f)\cap\mov(g)$, let
\[
u=\bigl\{z\in\mov(g)\setminus y\bigm|\orb(g,z)\setminus y=\{z\}\bigr\},
\]
and let $w=\mov(g)\setminus(y\cup u)$.
Then for all $z\in w$, $\orb(g,z)\setminus y\subseteq w$,
and thus $\orb(g\rhd w,z)=\orb(g,z)\cap w=\orb(g,z)\setminus y\neq\{z\}$,
which implies that $z\in\mov(g\rhd w)$. Hence $\mov(g\rhd w)=w$.
Note also that for all $z\in u$, $g(z)\in y$.
Now define $h$ to be the permutation of $x$ given by
\[
h(z)=
\begin{cases}
  f(z),         & \text{if $z\in\mov(f)\setminus g[u]$;}\\
  g^{-1}(z),    & \text{if $z\in g[u]$;}\\
  f(g(z)),      & \text{if $z\in u$;}\\
  (g\rhd w)(z), & \text{if $z\in w$;}\\
  z,            & \text{otherwise.}
\end{cases}
\]
Clearly, $\mov(h)=\mov(f)\cup\mov(g)$, and therefore $h\in\calspdf(x)$.
\end{proof}

\begin{lemma}\label{sh44}
For all cardinals $\mathfrak{a}$,
$\seq^{1\text{-}1}(\calspdf(\mathfrak{a}))\leqslant_{\mathrm{dfto}}\calspdf(\mathfrak{a})$.
\end{lemma}
\begin{proof}
Let $x$ be a set such that $|x|=\mathfrak{a}$.
By Lemma~\ref{sh43}, there exists a class function $G$ such that
for all $f$, $g\in\calspdf(x)$, $G(f,g)$ is defined and
is a permutation in $\calspdf(x)$ such that $\mov(G(f,g))=\mov(f)\cup\mov(g)$.
We define by recursion a function $\Phi$
from $\seq^{1\text{-}1}(\calspdf(x))$ into $\calspdf(x)$ as follows:
Take $\Phi(\emptyset)=\mathrm{id}_x$;
for all $n\in\omega$ and all $t\in\seq^{1\text{-}1}(\calspdf(x))$ with domain $n+1$,
we set $\Phi(t)=G(\Phi(t\upharpoonright n),t(n))$.
A routine induction shows that for all $t\in\seq^{1\text{-}1}(\calspdf(x))$,
\begin{equation}\label{sh45}
\mov(\Phi(t))=\bigcup_{i\in\dom(t)}\mov(t(i)).
\end{equation}
Now we show that $\Phi$ is a Dedekind finite to one map,
and thus complete the proof.
Take an arbitrary $h\in\calspdf(x)$ and let $y=\mov(h)\in\pdfin(x)$.
It suffices to show that
$u=\{t\in\seq^{1\text{-}1}(\calspdf(x))\mid\Phi(t)=h\}$ is Dedekind finite.
By \eqref{sh45}, for all $t\in u$ and all $i\in\dom(t)$, $\mov(t(i))\subseteq y$,
and hence $t(i)\upharpoonright y$ is a permutation of $y$.
Let $\Psi$ be the function on $u$ such that for all $t\in u$,
$\Psi(t)$ is the function defined on $\dom(t)$
given by $\Psi(t)(i)=t(i)\upharpoonright y$.
Clearly, $\Psi$ is an injection from $u$ into $\seq^{1\text{-}1}(\cals(y))$.
Since $y\in\pdfin(x)$, by Fact~\ref{sh11}, $\cals(y)$ is Dedekind finite,
and hence, by Fact~\ref{sh06},
$\seq^{1\text{-}1}(\cals(y))$ is Dedekind finite,
which implies that $u$ is also Dedekind finite.
\end{proof}

\begin{corollary}\label{sh46}
For all cardinals $\mathfrak{a}$, if $\mathfrak{a}!$ is Dedekind infinite,
then we have $\mathfrak{a}!\nleqslant_{\mathrm{dfto}}\seq^{1\text{-}1}(\calspdf(\mathfrak{a}))$.
\end{corollary}
\begin{proof}
This corollary follows from Lemma~\ref{sh44} and Theorem~\ref{tspdf}.
\end{proof}

\begin{theorem}\label{sh47}
For all cardinals $\mathfrak{a}$, if $\mathfrak{a}!$ is Dedekind infinite,
then we have $\mathfrak{a}!\nleqslant_{\mathrm{dfto}}\seq(\calspdf(\mathfrak{a}))$.
\end{theorem}
\begin{proof}
Assume towards a contradiction that there exists a cardinal $\mathfrak{a}$
such that $\mathfrak{a}!$ is Dedekind infinite and
such that $\mathfrak{a}!\leqslant_{\mathrm{dfto}}\seq(\calspdf(\mathfrak{a}))$.
If $\calspdf(\mathfrak{a})$ is Dedekind infinite, then, by Corollary~\ref{sh09},
$\seq(\calspdf(\mathfrak{a}))=\seq^{1\text{-}1}(\calspdf(\mathfrak{a}))$,
and thus $\mathfrak{a}!\leqslant_{\mathrm{dfto}}\seq^{1\text{-}1}(\calspdf(\mathfrak{a}))$,
contradicting Corollary~\ref{sh46}.
Otherwise, by Theorem~\ref{ktfs} and Fact~\ref{sh08},
$\aleph_0!=2^{\aleph_0}\leqslant\mathfrak{a}!\leqslant_{\mathrm{dfto}}\seq(\calspdf(\mathfrak{a}))\leqslant_{\mathrm{dfto}}\aleph_0$,
which is also a contradiction.
\end{proof}

\begin{corollary}\label{sh50}
For all cardinals $\mathfrak{a}$, if $\mathfrak{a}!$ is Dedekind infinite,
then we have $\mathfrak{a}!\nleqslant_{\mathrm{dfto}}\aleph_0\cdot\calspdf(\mathfrak{a})$,
and hence $\aleph_0\cdot\calspdf(\mathfrak{a})<\mathfrak{a}!$.
\end{corollary}
\begin{proof}
For all cardinals $\mathfrak{a}$, if $\mathfrak{a}!$ is Dedekind infinite,
then, by Fact~\ref{sh07} and Theorem~\ref{sh47},
$\mathfrak{a}!\nleqslant_{\mathrm{dfto}}\aleph_0\cdot\calspdf(\mathfrak{a})$,
and therefore, by Theorem~\ref{sh48},
$\aleph_0\cdot\calspdf(\mathfrak{a})<\mathfrak{a}!$.
\end{proof}

\begin{corollary}\label{sh51}
For all cardinals $\mathfrak{a}$, if $\mathfrak{a}!$ is Dedekind infinite,
then we have $\mathfrak{a}!\nleqslant_{\mathrm{dfto}}\seq(\mathfrak{a})$,
and hence $\seq(\mathfrak{a})<\mathfrak{a}!$.
\end{corollary}
\begin{proof}
This corollary follows from Fact~\ref{sh13} and Corollary~\ref{sh50}.
\end{proof}

\begin{corollary}\label{sh39}
For all non-zero cardinals $\mathfrak{a}$, $\mathfrak{a}!\neq\seq(\mathfrak{a})$.
\end{corollary}
\begin{proof}
For any non-zero cardinal $\mathfrak{a}$,
if $\mathfrak{a}!=\seq(\mathfrak{a})$, then,
by Fact~\ref{sh07}, $\mathfrak{a}!$ is Dedekind infinite,
contradicting Corollary~\ref{sh51}.
\end{proof}

We shall see in the next section that it is consistent with $\mathsf{ZF}$ that
there exists an infinite cardinal $\mathfrak{a}$ such that
$\mathfrak{a}!<\seq^{1\text{-}1}(\mathfrak{a})<\seq(\mathfrak{a})$.

\begin{corollary}\label{sh40}
For all cardinals $\mathfrak{a}$, if $\mathfrak{a}!$ is Dedekind infinite,
then we have $\mathfrak{a}!\nleqslant_{\mathrm{dfto}}\aleph_0\cdot\mathfrak{a}$,
and hence $\aleph_0\cdot\mathfrak{a}<\mathfrak{a}!$.
\end{corollary}
\begin{proof}
This corollary follows from Fact~\ref{sh07} and Corollary~\ref{sh51}.
\end{proof}

\begin{corollary}\label{sh41}
For all cardinals $\mathfrak{a}$, $\mathfrak{a}!\neq\aleph_0\cdot\mathfrak{a}$.
\end{corollary}
\begin{proof}
For every cardinal $\mathfrak{a}$, if $\mathfrak{a}!=\aleph_0\cdot\mathfrak{a}$,
then $\mathfrak{a}!$ is Dedekind infinite,
contradicting Corollary~\ref{sh40}.
\end{proof}

\subsection{Permutations that move finitely many elements}
Now, we focus our attention on cardinals bounded by $\calsf(\mathfrak{a})$.
The next theorem follows immediately from Theorem~\ref{tspdf}.

\begin{theorem}\label{tsfin}
For all cardinals $\mathfrak{a}$, if $\mathfrak{a}!$ is Dedekind infinite,
then we have $\mathfrak{a}!\nleqslant_{\mathrm{dfto}}\calsf(\mathfrak{a})$,
and hence $\calsf(\mathfrak{a})<\mathfrak{a}!$.
\end{theorem}

\begin{corollary}\label{sh27}
For all sets $x$, if $\calsf(x)\neq\cals(x)$ then $|\calsf(x)|<|\cals(x)|$.
\end{corollary}
\begin{proof}
Let $x$ be a set such that $\calsf(x)$ is a proper subset of $\cals(x)$.
Assume towards a contradiction that $|\calsf(x)|=|\cals(x)|$.
Therefore, by Theorem~\ref{dedt}, $\cals(x)$ is Dedekind infinite,
and hence, by Theorem~\ref{tsfin}, $|\calsf(x)|<|\cals(x)|$,
which is a contradiction.
\end{proof}

Let $x$ be an arbitrary set and let $\mathfrak{a}=|x|$.
For any $n\in\omega$, let $\mathcal{S}_n(x)$ denote the set of
all permutations of $x$ which move at most $n$ elements of $x$,
and let $\mathcal{S}_n(\mathfrak{a})$ denote the cardinal of $\mathcal{S}_n(x)$.

\begin{corollary}\label{sh28}
For all $n\in\omega\setminus\{0\}$ and all cardinals $\mathfrak{a}>n$, $\mathcal{S}_n(\mathfrak{a})<\mathfrak{a}!$.
\end{corollary}
\begin{proof}
Since $\mathfrak{a}>n$,
$\mathcal{S}_n(\mathfrak{a})\leqslant\mathcal{S}_n(\mathfrak{a})+1\leqslant\calsf(\mathfrak{a})\leqslant\mathfrak{a}!$.
Assume towards a contradiction that $\mathcal{S}_n(\mathfrak{a})=\mathfrak{a}!$.
By Theorem~\ref{dedt}, $\mathfrak{a}!$ is Dedekind infinite,
and hence, by Theorem~\ref{tsfin},
$\mathcal{S}_n(\mathfrak{a})\leqslant\calsf(\mathfrak{a})<\mathfrak{a}!$,
which is a contradiction.
\end{proof}

Let $x$ be an arbitrary set and let $\mathfrak{a}=|x|$.
For any $n\in\omega$, let $[x]^n$ denote the set of all $n$-element subsets of $x$,
and let $[\mathfrak{a}]^n$ denote the cardinal of $[x]^n$.

\begin{fact}\label{sh29}
For all cardinals $\mathfrak{a}$, $[\mathfrak{a}]^2+1=\mathcal{S}_2(\mathfrak{a})$.
\end{fact}
\begin{proof}
For any set $x$, the function $g$ defined on $\mathcal{S}_2(x)$ given by $g(t)=\mov(t)$
is a bijection from $\mathcal{S}_2(x)$ onto $[x]^2\cup\{\emptyset\}$.
\end{proof}

\begin{corollary}\label{sh30}
For all cardinals $\mathfrak{a}$, $[\mathfrak{a}]^2<\mathfrak{a}!$.
\end{corollary}
\begin{proof}
By Fact~\ref{sh29}, if $\mathfrak{a}\leqslant2$ then
$[\mathfrak{a}]^2<[\mathfrak{a}]^2+1=\mathcal{S}_2(\mathfrak{a})\leqslant\mathfrak{a}!$,
and if $\mathfrak{a}>2$ then, by Corollary~\ref{sh28},
$[\mathfrak{a}]^2\leqslant[\mathfrak{a}]^2+1=\mathcal{S}_2(\mathfrak{a})<\mathfrak{a}!$.
\end{proof}

In the next section, it will be shown that the following statement is consistent with $\mathsf{ZF}$:
There exists a Dedekind infinite cardinal $\mathfrak{a}$ such that $\mathfrak{a}!<[\mathfrak{a}]^3$,
$[\mathfrak{a}]^3\nleqslant_{\mathrm{dfto}}\mathfrak{a}!$, and $\mathfrak{a}!\leqslant^\ast\mathfrak{a}$.

\begin{lemma}\label{sh31}
For all cardinals $\mathfrak{a}$, $\bigl[[\mathfrak{a}]^2\bigr]^2+1\leqslant\mathcal{S}_5(\mathfrak{a})$.
\end{lemma}
\begin{proof}
For $\mathfrak{a}<8$, an easy calculation shows that
$\bigl[[\mathfrak{a}]^2\bigr]^2+1\leqslant\mathcal{S}_5(\mathfrak{a})$.
Now assume that $\mathfrak{a}\geqslant8$ and
let $x$ be a set such that $|x|=\mathfrak{a}$.
Let $z_i$, $v_i$ ($i<4$) be eight distinct elements of $x$.
Let $f$ be the function on $\bigl[[x]^2\bigr]^2\cup\{\emptyset\}$ such that
$f(\emptyset)=\mathrm{id}_x$ and such that
for any four distinct elements $a$, $b$, $c$, $d$ of $x$,
\[
f\Bigl(\bigl\{\{a,b\},\{c,d\}\bigr\}\Bigr)=\bigl(a;b\bigr)_x\circ\bigl(c;d\bigr)_x
\]
and
\[
f\Bigl(\bigl\{\{a,b\},\{a,c\}\bigr\}\Bigr)=\bigl(a;z_k;v_k\bigr)_x\circ\bigl(b;c\bigr)_x
\]
where $k<4$ is the least natural number such that $\{a,b,c\}\cap\{z_k,v_k\}=\emptyset$.
Then it is easy to verify that $f$ is an injection
from $\bigl[[x]^2\bigr]^2\cup\{\emptyset\}$ into $\mathcal{S}_5(x)$,
and hence $\bigl[[\mathfrak{a}]^2\bigr]^2+1\leqslant\mathcal{S}_5(\mathfrak{a})$.
\end{proof}

\begin{corollary}\label{sh32}
For all cardinals $\mathfrak{a}$, $\bigl[[\mathfrak{a}]^2\bigr]^2<\mathfrak{a}!$.
\end{corollary}
\begin{proof}
By Lemma~\ref{sh31}, if $\mathfrak{a}\leqslant5$ then
$\bigl[[\mathfrak{a}]^2\bigr]^2<\bigl[[\mathfrak{a}]^2\bigr]^2+1\leqslant\mathcal{S}_5(\mathfrak{a})\leqslant\mathfrak{a}!$,
and if $\mathfrak{a}>5$ then, by Corollary~\ref{sh28},
$\bigl[[\mathfrak{a}]^2\bigr]^2\leqslant\bigl[[\mathfrak{a}]^2\bigr]^2+1\leqslant\mathcal{S}_5(\mathfrak{a})<\mathfrak{a}!$.
\end{proof}

It will be shown in the next section that $\Bigl[\bigl[[\mathfrak{a}]^2\bigr]^2\Bigr]^2\leqslant\mathfrak{a}!$
for an arbitrary infinite cardinal $\mathfrak{a}$ cannot be proved in $\mathsf{ZF}$.

Let $x$ be an arbitrary set and let $\mathfrak{a}=|x|$.
Recall that for every $n\in\omega$, $x^n$~is the set of all functions from $n$ into $x$,
and $\mathfrak{a}^n$ is the cardinal of $x^n$.

\begin{lemma}\label{sh33}
For all $n\in\omega$ and all cardinals $\mathfrak{a}\geqslant2n(n+1)$,
$\mathfrak{a}^n\leqslant\mathcal{S}_{2n+1}(\mathfrak{a})$.
Moreover, from an $n\in\omega\setminus\{0\}$, a set $x$, and an injection $f:2n(n+1)\to x$,
one can explicitly define an injection $g:x^n\to\mathcal{S}_{2n+1}(x)\setminus\mathcal{S}_{2n}(x)$.
\end{lemma}
\begin{proof}
Let $n$ be a non-zero natural number,
and let $f$ be an injection from $2n(n+1)$ into $x$.
Without loss of generality, assume that $x\cap\omega=\emptyset$.
For any $i$, $j\leqslant n$ and $k<n-1$,
let $z_{i,j}=f(2ni+j)$ and let $v_{i,k}=f(2ni+n+k+1)$.
Then $z_{i,j}$, $v_{i,k}$ ($i$, $j\leqslant n$, $k<n-1$)
are pairwise distinct elements of $x$.
For each $t\in x^n$, let
\[
m_t=\min\bigl\{i\leqslant n\bigm|\ran(t)\cap
\bigl(\{z_{i,j}\mid j\leqslant n\}\cup\{v_{i,k}\mid k<n-1\}\bigr)=\emptyset\bigr\}.
\]

Let $h$ be the function on $x^n$ such that for all $t\in x^n$,
$h(t)$ is the function defined on $n$ given by
\[
h(t)(l)=
\begin{cases}
  t(l),                      & \text{if for all $k<l$, $t(k)\neq t(l)$;}\\
  \max\{k<l\mid t(k)=t(l)\}, & \text{otherwise.}
\end{cases}
\]
Then as in the proof of Lemma~\ref{sh08b},
$h$ is an injection from $x^n$ into the set
$\{u\mid u\text{ is an injection from }n\text{ into }x\cup(n-1)\}$.
Let $\Phi$ be the function on $x^n$ such that for all $t\in x^n$,
$\Phi(t)$ is the function defined on $n$ given by
\[
\Phi(t)(l)=
\begin{cases}
  h(t)(l),         & \text{if $h(t)(l)\in x$;}\\
  v_{m_t,h(t)(l)}, & \text{if $h(t)(l)\in n-1$.}
\end{cases}
\]
Clearly, for all $t\in x^n$, $\Phi(t)$ is an injection from $n$ into $x$.
Note that $\Phi$ need not be injective.

Now, let $g$ be the function defined on $x^n$ given by
\[
g(t)=\bigl(\Phi(t)(0);\dots;\Phi(t)(n-1);z_{m_t,0};\dots;z_{m_t,n}\bigr)_x.
\]
Clearly, for all $t\in x^n$, $g(t)\in\mathcal{S}_{2n+1}(x)\setminus\mathcal{S}_{2n}(x)$.
Moreover, $g$ is injective, since for all $t\in x^n$,
$t$ is uniquely determined by $g(t)$ in the following way:
First, $m_t$ is the unique $i\leqslant n$ such that
$\{z_{i,j}\mid j\leqslant n\}\subseteq\mov(g(t))$,
and $\Phi(t)$ is the function on $n$ such that
$\Phi(t)(l)=(g(t))^{(l+1)}(z_{m_t,n})$ for any $l<n$.
Then $h(t)$ is the function on $n$ such that
for all $l<n$, $h(t)(l)=\Phi(t)(l)$, if $\Phi(t)(l)\notin\{v_{m_t,k}\mid k<n-1\}$,
and $h(t)(l)$ is the unique $k<n-1$ for which $\Phi(t)(l)=v_{m_t,k}$, otherwise.
Finally, since $h$ is injective, $t$ is uniquely determined by $h(t)$,
and hence by $g(t)$.
\end{proof}

\begin{corollary}\label{sh35}
For all Dedekind infinite cardinals $\mathfrak{a}$, $\seq(\mathfrak{a})\leqslant\calsf(\mathfrak{a})$.
\end{corollary}
\begin{proof}
Let $x$ be a set such that $|x|=\mathfrak{a}$,
and let $f$ be an injection from $\omega$ into $x$.
By the second part of Lemma~\ref{sh33},
there exists a class function $G$ such that
for all $n\in\omega\setminus\{0\}$ and all injections $g:2n(n+1)\to x$,
$G(n,x,g)$ is defined and is an injection from $x^n$
into $\mathcal{S}_{2n+1}(x)\setminus\mathcal{S}_{2n}(x)$.
Then
\[
h=\bigl\{(\emptyset,\mathrm{id}_x)\bigr\}\cup\bigcup_{n\in\omega\setminus\{0\}}G\bigl(n,x,f\upharpoonright 2n(n+1)\bigr)
\]
is an injection from $\seq(x)$ into $\calsf(x)$, and hence $\seq(\mathfrak{a})\leqslant\calsf(\mathfrak{a})$.
\end{proof}

It follows from Corollary~\ref{sh35} and Theorem~\ref{tsfin} that
$\seq(\mathfrak{a})<\mathfrak{a}!$ for any Dedekind infinite cardinal $\mathfrak{a}$,
however, this result is a special case of Corollary~\ref{sh51}.
The next result was also proved in \cite[Theorem~2.3]{SonpanowVejjajiva2018}.

\begin{corollary}\label{sh34}
For all $n\in\omega$ and all infinite cardinals $\mathfrak{a}$, $\mathfrak{a}^n<\mathfrak{a}!$.
\end{corollary}
\begin{proof}
By Lemma~\ref{sh33} and Corollary~\ref{sh28},
$\mathfrak{a}^n\leqslant\mathcal{S}_{2n+1}(\mathfrak{a})<\mathfrak{a}!$.
\end{proof}

In \cite[Theorem~3.10]{SonpanowVejjajiva2017}, Sonpanow and Vejjajiva proved that
for all infinite sets $x$, if $x$ is \emph{almost even} in the sense that
there exists a permutation $f$ of~$x$ without fixed points such that $f\circ f=\mathrm{id}_x$,
then there are no finite-to-one maps from $\cals(x)$ into $x$.
This result is a special case of the next corollary.

\begin{corollary}\label{sh36}
For all infinite sets $x$, if there exists a permutation of $x$ without fixed points,
then for any $n\in\omega$, there are no power Dedekind finite to one maps from $\cals(x)$ into $x^n$.
\end{corollary}
\begin{proof}
Assume towards a contradiction that there exists an infinite set $x$
such that $\mov(f)=x$ for some $f\in\cals(x)$ and
such that $\cals(x)\preccurlyeq_{\mathrm{pdfto}}x^n$ for some $n\in\omega$.
By Fact~\ref{sh10}, $\cals(x)$ is power Dedekind infinite,
and therefore, by Fact~\ref{sh03}, $x^n$ is power Dedekind infinite,
which implies that, by Fact~\ref{sh04}, $x$ is power Dedekind infinite.
Then, since $\mov(f)=x$, $f\in\cals(x)\setminus\calspdf(x)$,
and hence, by Theorem~\ref{ktfs}, $\cals(x)$ is Dedekind infinite. Now, we have that
$\cals(x)\preccurlyeq_{\mathrm{pdfto}}x^n\subseteq\seq(x)$,
contradicting Corollary~\ref{sh51}.
\end{proof}

Now it is natural to ask whether we can prove in $\mathsf{ZF}$ that
for all infinite cardinals $\mathfrak{a}$,
$\mathfrak{a}!\nleqslant_{\mathrm{fto}}\mathfrak{a}$.
It turns out that the answer is no,
and this is one of the main results of the present paper.
At present, we only discuss the relationship between
$\mathfrak{a}!\leqslant_{\mathrm{fto}}\mathfrak{a}$
and $\mathfrak{a}!<\aleph_0\cdot\mathfrak{a}$.
Note that, by Corollary~\ref{sh41},
$\mathfrak{a}!\neq\aleph_0\cdot\mathfrak{a}$ for any cardinal $\mathfrak{a}$.

\begin{lemma}\label{sh37}
For all cardinals $\mathfrak{a}$, if $\mathfrak{a}!<\aleph_0\cdot\mathfrak{a}$
then $\mathfrak{a}!\leqslant_{\mathrm{fto}}\mathfrak{a}$.
\end{lemma}
\begin{proof}
Let $x$ be an arbitrary set.
If there is an injection $f:\cals(x)\to\omega\times x$,
then for all $z\in x$, $f^{-1}[\omega\times\{z\}]$ is finite,
since otherwise $\cals(x)$ would be Dedekind infinite,
contradicting Corollary~\ref{sh40}.
Hence the function that maps each $t\in\cals(x)$ to
the second component of $f(t)$ is a finite-to-one map from $\cals(x)$ into $x$,
and therefore $\cals(x)\preccurlyeq_{\mathrm{fto}}x$.
\end{proof}

\begin{lemma}\label{sh52}
All infinite subsets of $\wp(\omega)$ are power Dedekind infinite.
\end{lemma}
\begin{proof}
Cf.~\cite[Lemma~5]{Blass2013}.
\end{proof}

As a consequence of this lemma, we get that
for all subsets $x$ of $\wp(\omega)$,
$\pdfin(x)=\fin(x)$ and $\calspdf(x)=\calsf(x)$.

\begin{theorem}\label{sh38}
For all cardinals $\mathfrak{a}\leqslant2^{\aleph_0}$,
$\mathfrak{a}!<\aleph_0\cdot\mathfrak{a}$ iff $\mathfrak{a}!\leqslant_{\mathrm{fto}}\mathfrak{a}$.
\end{theorem}
\begin{proof}
Assume that $\mathfrak{a}\leqslant2^{\aleph_0}$ and
let $x$ be a subset of $\wp(\omega)$ such that $|x|=\mathfrak{a}$.
By Lemma~\ref{sh37}, $\mathfrak{a}!<\aleph_0\cdot\mathfrak{a}$
implies that $\mathfrak{a}!\leqslant_{\mathrm{fto}}\mathfrak{a}$.
For the other direction, assume that there is a finite-to-one map $f:\cals(x)\to x$.
By Corollary~\ref{sh40}, $\cals(x)$ is Dedekind finite,
and hence, by Theorem~\ref{ktfs}, $\calspdf(x)=\cals(x)$,
which implies that, by Lemma~\ref{sh52}, $\calsf(x)=\cals(x)$.

Let $r$ be the lexicographic ordering of $x$; that is,
\[
r=\bigl\{(z,v)\in x\times x\bigm|\exists n\in\omega\bigl(n\notin z\wedge n\in v\wedge n\cap z=n\cap v\bigr)\bigr\}.
\]
Let $s$ be the relation on $\cals(x)$ defined by
\[
(t,u)\in s\leftrightarrow\exists z\in x\bigl((t(z),u(z))\in r\wedge\forall v\in x\bigl((v,z)\in r\rightarrow t(v)=u(v)\bigr)\bigr).
\]
We claim that $s$ orders $\cals(x)$.
In fact, it is easy to verify that $s$ is irreflexive and transitive.
For trichotomy, let $t$, $u$ be two distinct elements of $\cals(x)$.
Then $\{z\in x\mid t(z)\neq u(z)\}$ is a non-void subset of $\mov(t)\cup\mov(u)$.
Since $\calsf(x)=\cals(x)$, $\mov(t)$ and $\mov(u)$ are finite,
and thus $\{z\in x\mid t(z)\neq u(z)\}$ has a least element $w$ with respect to $r$.
Now, if $(t(w),u(w))\in r$ then $(t,u)\in s$, and if $(u(w),t(w))\in r$ then $(u,t)\in s$.

Let $g$ be the function on $x$ such that for all $z\in x$,
$g(z)$ is the unique isomorphism of $\langle f^{-1}[\{z\}],s\rangle$
onto some natural number.
Then the function $h$ defined on $\cals(x)$ given by
\[
h(t)=\bigl(g(f(t))(t),f(t)\bigr)
\]
is an injection from $\cals(x)$ into $\omega\times x$,
and hence $\mathfrak{a}!\leqslant\aleph_0\cdot\mathfrak{a}$,
which implies that, by Corollary~\ref{sh41},
$\mathfrak{a}!<\aleph_0\cdot\mathfrak{a}$.
\end{proof}

\section{Permutation models}
In this section, we shall give a brief introduction to permutation models
(cf.~\cite[Chap.~8]{Halbeisen2017} or \cite[Chap.~4]{Jech1973}),
and derive some consistency results from a few well-known permutation models.
Permutation models are not models of $\mathsf{ZF}$;
they are models of the weaker theory $\mathsf{ZFA}$
(i.e., the Zermelo-Fraenkel set theory with atoms).
$\mathsf{ZFA}$ is characterized by the fact that
it admits objects other than sets, \emph{atoms}.
Atoms are objects which do not have any elements
but which are distinct from the void set;
they are not sets but can be members of sets.
The development of $\mathsf{ZFA}$ is essentially the same
as that of $\mathsf{ZF}$, and all proofs in the previous sections
can be carried out in $\mathsf{ZFA}$.

In $\mathsf{ZFA}$, for any transitive set $x$,
we define $\mathrm{V}^x_\alpha$ by recursion on $\alpha$ as follows:
$\mathrm{V}^x_0=x$; $\mathrm{V}^x_{\alpha+1}=\wp(\mathrm{V}^x_\alpha)$;
$\mathrm{V}^x_\alpha=\bigcup_{\beta<\alpha}\mathrm{V}^x_\beta$ when $\alpha$ is a limit ordinal.
Further, let $\mathrm{V}^x=\bigcup_{\alpha\in\mathrm{Ord}}\mathrm{V}^x_\alpha$.

Let $A$ be the set of atoms.
The axiom of foundation of $\mathsf{ZFA}$ guarantees
that $\mathrm{V}^A$ is the class of all objects.
The class $\mathrm{V}^\emptyset$ is a model
of $\mathsf{ZF}$ and is called the \emph{kernel}.
Note that all ordinals belong to the kernel.
Now every permutation $\pi$ of $A$ extends to an $\in$-automorphism of $\mathrm{V}^A$ by
\[
\pi(x)=\pi[x].
\]
A routine induction shows that for any permutation $\pi$ of $A$
and any $x$ in the kernel, we have $\pi(x)=x$.

Let $\mathcal{G}$ be a permutation group of $A$
(i.e., a group of permutations of $A$).
For each $x\in\mathrm{V}^A$, let
\[
\symg(x)=\bigl\{\pi\in\mathcal{G}\bigm|\pi(x)=x\bigr\};
\]
$\symg(x)$ is a subgroup of $\mathcal{G}$.
We say that a set $\mathfrak{F}$ of subgroups of $\mathcal{G}$ is
a \emph{normal filter} on $\mathcal{G}$ if for all subgroups $H$, $K$ of $\mathcal{G}$,
\begin{enumerate}[label=\upshape(\roman*)]
  \item $\mathcal{G}\in\mathfrak{F}$;
  \item if $H\in\mathfrak{F}$ and $H\subseteq K$ then $K\in\mathfrak{F}$;
  \item if $H\in\mathfrak{F}$ and $K\in\mathfrak{F}$ then $H\cap K\in\mathfrak{F}$;
  \item if $\pi\in\mathcal{G}$ and $H\in\mathfrak{F}$ then $\pi H\pi^{-1}\in\mathfrak{F}$;\label{sh53}
  \item for each $a\in A$, $\symg(a)\in\mathfrak{F}$.\label{sh54}
\end{enumerate}

Let $\mathfrak{F}$ be a normal filter on $\mathcal{G}$.
We say that $x\in\mathrm{V}^A$ is \emph{symmetric}
(with respect to $\mathfrak{F}$) if $\symg(x)\in\mathfrak{F}$.
By \ref{sh53}, for all $x\in\mathrm{V}^A$ and all $\pi\in\mathcal{G}$,
$x$ is symmetric if and only if $\pi(x)$ is symmetric.
By \ref{sh54}, each $a\in A$ is symmetric.
We say that $x\in\mathrm{V}^A$ is \emph{hereditarily symmetric}
(with respect to $\mathfrak{F}$) if $x$ as well as
each element of its transitive closure is symmetric.
Note that for all $x\in\mathrm{V}^A$ and all $\pi\in\mathcal{G}$,
$x$ is hereditarily symmetric if and only if $\pi(x)$ is hereditarily symmetric.
Note also that $A$ is hereditarily symmetric.

The \emph{permutation model} $\mathcal{V}$ (determined by $\mathfrak{F}$)
consists of all hereditarily symmetric objects.
Then each permutation $\pi\in\mathcal{G}$, when extended to $\mathrm{V}^A$ as above,
maps $\mathcal{V}$ onto itself and in fact is an $\in$-automorphism of $\mathcal{V}$.
Now it is easy to verify that $\mathcal{V}$ is a transitive model of $\mathsf{ZFA}$
containing $A$ and all elements of the kernel.

Most of the well-known permutation models are of the following simple type:
Let $\mathcal{G}$ be a permutation group of $A$.
A family $\mathcal{I}$ of subsets of $A$, for example $\mathcal{I}=\fin(A)$,
is a \emph{normal ideal} if for all subsets $B$, $C$ of $A$,
\begin{enumerate}
  \item $\emptyset\in\mathcal{I}$;
  \item if $B\in\mathcal{I}$ and $C\subseteq B$ then $C\in\mathcal{I}$;
  \item if $B\in\mathcal{I}$ and $C\in\mathcal{I}$ then $B\cup C\in\mathcal{I}$;
  \item if $\pi\in\mathcal{G}$ and $B\in\mathcal{I}$ then $\pi[B]\in\mathcal{I}$;
  \item for each $a\in A$, $\{a\}\in\mathcal{I}$.
\end{enumerate}
For each subset $B$ of $A$, let
\[
\fixg(B)=\bigl\{\pi\in\mathcal{G}\bigm|\forall a\in B\bigl(\pi(a)=a\bigr)\bigr\};
\]
$\fixg(B)$ is a subgroup of $\mathcal{G}$.
Define $\mathfrak{F}$ to be the filter on $\mathcal{G}$
generated by the subgroups $\{\fixg(B)\mid B\in\mathcal{I}\}$.
$\mathfrak{F}$ is a normal filter,
and so it determines a permutation model $\mathcal{V}$;
we say that $\mathcal{V}$ is the permutation model
\emph{determined} by $\mathcal{G}$ and $\mathcal{I}$.
Note that $x$ is symmetric (with respect to $\mathfrak{F}$)
if and only if there exists a $B\in\mathcal{I}$ such that
\[
\fixg(B)\subseteq\symg(x);
\]
we say that such a $B\in\mathcal{I}$ is a \emph{support} of $x$.
Note also that $\mathcal{I}\in\mathcal{V}$.

Although permutation models are not models of $\mathsf{ZF}$,
they indirectly give, via the Jech-Sochor theorem
(cf.~\cite[Theorem~17.2]{Halbeisen2017} or \cite[Theorem~6.1]{Jech1973}),
models of $\mathsf{ZF}$. The Jech-Sochor theorem provides embeddings
of arbitrarily large initial segments of permutation models into $\mathsf{ZF}$ models.
All statements whose consistency we prove in the present paper
depend only on a very small initial segment of the permutation model,
so they are preserved by the embedding and we thus
obtain their consistency with $\mathsf{ZF}$.

\subsection{The basic Fraenkel model}
Let the set $A$ of atoms be denumerable,
let $\mathcal{G}=\cals(A)$, and let $\mathcal{I}=\fin(A)$.
The permutation model determined by $\mathcal{G}$ and $\mathcal{I}$
is called the \emph{basic Fraenkel model}
(cf.~\cite[pp.~195--196]{Halbeisen2017} or \cite[\S4.3]{Jech1973}),
and is denoted by $\mathcal{V}_\mathrm{F}$ ($\mathrm{F}$ for Fraenkel).

In $\mathcal{V}_\mathrm{F}$, $A$ is \emph{amorphous} (cf.~\cite[Lemma~8.2]{Halbeisen2017});
that is, $A$ is infinite but every infinite subset of $A$ is co-finite.
Since it is obvious that all amorphous sets are power Dedekind finite,
we have that $A$ is power Dedekind finite, and therefore,
by Fact~\ref{sh11}, $\cals(A)$ is Dedekind finite.

Moreover, it is easy to verify that $A$ is \emph{strongly amorphous},
in the sense that $A$ is amorphous and for any partition $P$ of $A$,
all but finitely many elements of $P$ are singletons.
Hence, by the following fact, $\calsf(A)=\cals(A)$,
which implies that the existence of an infinite set $x$
such that $\calsf(x)=\cals(x)$ is consistent with $\mathsf{ZF}$.
Therefore, in Corollary~\ref{sh27}, the requirement that $\calsf(x)\neq\cals(x)$
cannot be replaced by the requirement that $x$ is infinite.

\begin{fact}\label{sh42}
For all strongly amorphous sets $x$, $\calsf(x)=\cals(x)$.
\end{fact}
\begin{proof}
Let $x$ be a strongly amorphous set and let $f$ be a permutation of $x$.
If there is a $z\in x$ such that $\orb(f,z)$ is denumerable,
then $x$ is Dedekind infinite, contradicting the assumption that $x$ is amorphous.
Hence all orbits of $f$ are finite. Since the orbits of $f$ form a partition of $x$,
all but finitely many orbits of $f$ are singletons,
which implies that $f\in\calsf(x)$.
\end{proof}

\begin{lemma}\label{sh55}
Let $A$ be the set of atoms of $\mathcal{V}_\mathrm{F}$ and let $\mathfrak{a}=|A|$.
In $\mathcal{V}_\mathrm{F}$,
\begin{enumerate}[label=\upshape(\roman*)]
  \item $[\mathfrak{a}]^2\nleqslant_{\mathrm{fto}}\seq(\mathfrak{a})$;\label{sh61}
  \item $\seq^{1\text{-}1}(\mathfrak{a})\nleqslant\mathfrak{a}!$;\label{sh62}
  \item $\mathcal{S}_3(\mathfrak{a})\nleqslant2^{\mathfrak{a}+\aleph_0}$;\label{sh63}
  \item $[\mathfrak{a}]^3\nleqslant_{\mathrm{fto}}(\mathfrak{a}+\aleph_0)!$;\label{sh64}
  \item $([\mathfrak{a}]^2)^2\nleqslant(\mathfrak{a}+\aleph_0)!$.\label{sh65}
\end{enumerate}
\end{lemma}
\begin{proof}
(i) Assume towards a contradiction that there exists
a finite-to-one map $f\in\mathcal{V}_\mathrm{F}$ from $[A]^2$ into $\seq(A)$.
Let $B\in\fin(A)$ be a support of $f$.
Let us fix two distinct elements $a$, $b$ of
$A\setminus B$ and consider the sequence $t=f(\{a,b\})$.
If there is an $n\in\dom(t)$ such that $t(n)\in A\setminus(B\cup\{a,b\})$,
then take an arbitrary $c\in A\setminus(B\cup\{a,b,t(n)\})$ and let $\pi=(t(n);c)_A$.
Note that $\pi\in\fixg(B\cup\{a,b\})$ but $\pi$ moves $t$,
contradicting the assumption that $B$ is a support of $f$.
Hence $t\in\seq(B\cup\{a,b\})$.
If there is an $m\in\dom(t)$ such that $t(m)\in\{a,b\}$,
then $\sigma=(a;b)_A$ is a member of $\fixg(B)$ such that
$\sigma(\{a,b\})=\{a,b\}$ and $\sigma(t)\neq t$, which is also a contradiction.
Thus we have
\begin{equation}\label{sh56}
\forall a,b\in A\setminus B\bigl(a\neq b\rightarrow f(\{a,b\})\in\seq(B)\bigr).
\end{equation}
Now, for any $p$, $q\in[A\setminus B]^2$,
since it is easy to see that there exists a permutation $\tau\in\fixg(B)$
such that $\tau(p)=q$, by \eqref{sh56}, we have $f(p)=f(q)$.
Therefore, $f$ maps all elements of $[A\setminus B]^2$ to the same element of $\seq(B)$,
contradicting the fact that $[A\setminus B]^2$ is infinite and $f$ is finite-to-one.

(ii) Assume towards a contradiction that there exists
an injection $g\in\mathcal{V}_\mathrm{F}$ from $\seq^{1\text{-}1}(A)$ into $\cals(A)$.
Let $C\in\fin(A)$ be a support of $g$.
Without loss of generality, assume that $C\neq\emptyset$.
Let us fix an arbitrary $t\in\seq^{1\text{-}1}(C)$ and consider the permutation $u=g(t)$.
If there exists a $c\in\mov(u)\setminus C$,
then take an arbitrary $d\in A\setminus(C\cup\{c,u(c)\})$ and let $\pi=(c;d)_A$.
Note that $\pi\in\fixg(C\cup\{u(c)\})$ but $\pi$ moves $u$, which is a contradiction.
Thus we get that for all $t\in\seq^{1\text{-}1}(C)$, $\mov(g(t))\subseteq C$.
Hence the function $f$ defined on $\seq^{1\text{-}1}(C)$
given by $f(t)=g(t)\upharpoonright C$ is an injection
from $\seq^{1\text{-}1}(C)$ into $\cals(C)$.
Thus, if we take $n=|C|$, then $n\neq0$ and
$\seq^{1\text{-}1}(n)\leqslant n!$, which is absurd.

(iii) Assume towards a contradiction that there exists
an injection $h\in\mathcal{V}_\mathrm{F}$ from $\mathcal{S}_3(A)$ into $\wp(A\cup\omega)$.
Let $D\in\fin(A)$ be a support of $h$.
Take three distinct elements $a$, $b$, $c$ of $A\setminus D$,
let $\pi=(a;b;c)_A$, and let $\sigma=(b;a;c)_A$.
Then $\pi$, $\sigma\in\fixg(D)$, and hence $\pi(h)=\sigma(h)=h$.
Since $\pi(\pi)=\sigma(\pi)=\pi$, we get $\pi(h(\pi))=\sigma(h(\pi))=h(\pi)$.
Hence, if $a\in h(\pi)$ then $b=\pi(a)\in h(\pi)$,
and if $b\in h(\pi)$ then $a=\sigma(b)\in h(\pi)$;
that is, $a\in h(\pi)\leftrightarrow b\in h(\pi)$.
Thus, if we set $\tau=(a;b)_A$, then $\tau(h(\pi))=h(\pi)$.
Since $\tau\in\fixg(D)$, $\tau(h)=h$,
and hence $h(\pi)=\tau(h(\pi))=h(\tau(\pi))=h(\sigma)$,
contradicting that $h$ is injective.

(iv) Assume towards a contradiction that there exists
a finite-to-one map $f\in\mathcal{V}_\mathrm{F}$ from $[A]^3$ into $\cals(A\cup\omega)$.
Let $B\in\fin(A)$ be a support of $f$.
Let us now fix three distinct elements $a$, $b$, $c$ of
$A\setminus B$ and consider the permutation $u=f(\{a,b,c\})$.
If there is a $d\in\mov(u)\setminus(B\cup\omega\cup\{a,b,c\})$,
then take an arbitrary $e\in A\setminus(B\cup\{a,b,c,d,u(d)\})$ and let $\pi=(d;e)_A$.
Note that $\pi\in\fixg(B\cup\{a,b,c\})$, $\pi(d)\neq d$, and $\pi(u(d))=u(d)$.
Hence $\pi$ moves $u$, contradicting the assumption that $B$ is a support of $f$.
Therefore $\mov(u)\subseteq B\cup\omega\cup\{a,b,c\}$.
If there is a $v\in\mov(u)\cap\{a,b,c\}$,
then, since $\{a,b,c\}\setminus\{v,u(v)\}\neq\emptyset$,
take a $w\in\{a,b,c\}\setminus\{v,u(v)\}$ and let $\sigma=(v;w)_A$.
Note that $\sigma\in\fixg(B)$, $\sigma(\{a,b,c\})=\{a,b,c\}$,
and $\sigma(u)\neq u$, which is also a contradiction.
Therefore $\mov(u)\subseteq B\cup\omega$. Thus we have
\begin{equation}\label{sh57}
\forall t\in[A\setminus B]^3\bigl(\mov(f(t))\subseteq B\cup\omega\bigr).
\end{equation}
Now, for any $p$, $q\in[A\setminus B]^3$,
since it is easy to see that there exists a permutation $\tau\in\fixg(B)$
such that $\tau(p)=q$, by \eqref{sh57}, we have $f(p)=f(q)$.
Therefore, $f$ maps all elements of $[A\setminus B]^3$ to the same element of $\cals(A\cup\omega)$,
contradicting the fact that $[A\setminus B]^3$ is infinite and $f$ is finite-to-one.

(v) Assume towards a contradiction that there exists
an injection $g\in\mathcal{V}_\mathrm{F}$ from $[A]^2\times[A]^2$ into $\cals(A\cup\omega)$.
Let $C\in\fin(A)$ be a support of $g$.
Take four distinct elements $a_0$, $a_1$, $b_0$, $b_1$ of $A\setminus C$,
and let $u=g(\{a_0,a_1\},\{b_0,b_1\})$.
If there is a $c\in\mov(u)\setminus(C\cup\omega\cup\{a_0,a_1,b_0,b_1\})$,
then take an arbitrary $d\in A\setminus(C\cup\{a_0,a_1,b_0,b_1,c,u(c)\})$ and let $\pi=(c;d)_A$.
Then we have that $\pi\in\fixg(C\cup\{a_0,a_1,b_0,b_1\})$, $\pi(c)\neq c$, and $\pi(u(c))=u(c)$.
Thus $\pi$ moves $u$, contradicting the assumption that $C$ is a support of $g$.
Therefore we have
\begin{equation}\label{sh58}
\mov(u)\subseteq C\cup\omega\cup\{a_0,a_1,b_0,b_1\}.
\end{equation}
We claim that
\begin{equation}\label{sh59}
\forall i\leqslant1\bigl(u(a_i)=a_{1-i}\text{ and }u(b_i)=b_{1-i}\bigr).
\end{equation}
In fact, if $u(a_i)\notin\{a_0,a_1\}$,
then $(a_0;a_1)_A\in\fixg(C)$ fixes $(\{a_0,a_1\},\{b_0,b_1\})$ but moves $u$,
contradicting the assumption that $C$ is a support of $g$.
Thus we have $u(a_i)\in\{a_0,a_1\}$.
Moreover, $u(a_i)\neq a_i$, since otherwise,
if we take an arbitrary $e\in A\setminus(C\cup\{a_0,a_1,b_0,b_1\})$,
then, by \eqref{sh58}, $u(e)=e$, and thus $(a_i;e)_A\in\fixg(C)$
fixes $u$ but moves $(\{a_0,a_1\},\{b_0,b_1\})$, contradicting that $g$ is injective.
Hence $u(a_i)=a_{1-i}$. Similarly $u(b_i)=b_{1-i}$, and \eqref{sh59} is proved.
Therefore, if we set $\sigma=(a_0;b_0)_A\circ(a_1;b_1)_A$,
then, by \eqref{sh59}, $\sigma(u)=u$,
but $\sigma(\{a_0,a_1\},\{b_0,b_1\})=(\{b_0,b_1\},\{a_0,a_1\})\neq(\{a_0,a_1\},\{b_0,b_1\})$,
contradicting again the assumption that $g$ is injective.
\end{proof}

Now we derive some consistency results from Lemma~\ref{sh55}.

\begin{proposition}\label{sh60}
The following statements are consistent with $\mathsf{ZF}$:
\begin{enumerate}[label=\upshape(\roman*)]
  \item There exists an infinite cardinal $\mathfrak{a}$ such that
        $\mathfrak{a}!$ and $\seq^{1\text{-}1}(\mathfrak{a})$ are incomparable
        and such that $\mathfrak{a}!$ and $\seq(\mathfrak{a})$ are incomparable.\label{sh67}
  \item There exists a Dedekind infinite cardinal $\mathfrak{b}$ such that
        $\mathfrak{b}!$ and $2^\mathfrak{b}$ are incomparable,
        $\mathfrak{b}!$ and $[\mathfrak{b}]^3$ are incomparable,
        and $[\mathfrak{b}]^3\nleqslant_{\mathrm{fto}}\mathfrak{b}!$.\label{sh68}
  \item There exists a Dedekind infinite cardinal $\mathfrak{c}$ such that
        $([\mathfrak{c}]^2)^2\nleqslant\mathfrak{c}!$.\label{sh66}
  \item There exists a Dedekind infinite cardinal $\mathfrak{d}$ such that
        $\Bigl[\bigl[[\mathfrak{d}]^2\bigr]^2\Bigr]^2\nleqslant\mathfrak{d}!$.\label{sh69}
\end{enumerate}
\end{proposition}
\begin{proof}
By the Jech-Sochor theorem, it suffices to show that
there are such cardinals in $\mathcal{V}_\mathrm{F}$.
Let $A$ be the set of atoms of $\mathcal{V}_\mathrm{F}$ and let $\mathfrak{a}=|A|$.

(i) Note that $\seq^{1\text{-}1}(\mathfrak{a})\leqslant\seq(\mathfrak{a})$
and that, by Corollary~\ref{sh30}, $[\mathfrak{a}]^2\leqslant\mathfrak{a}!$.
By Lemma~\ref{sh55}\ref{sh61}, $[\mathfrak{a}]^2\nleqslant\seq(\mathfrak{a})$,
and hence $\mathfrak{a}!\nleqslant\seq(\mathfrak{a})$
and $\mathfrak{a}!\nleqslant\seq^{1\text{-}1}(\mathfrak{a})$.
By Lemma~\ref{sh55}\ref{sh62}, $\seq^{1\text{-}1}(\mathfrak{a})\nleqslant\mathfrak{a}!$,
and thus $\seq(\mathfrak{a})\nleqslant\mathfrak{a}!$, which completes the proof of (i).

(ii) Let $\mathfrak{b}=\mathfrak{a}+\aleph_0$.
Note that $\mathfrak{b}$ is Dedekind infinite, $[\mathfrak{b}]^3\leqslant2^\mathfrak{b}$,
and $\mathcal{S}_3(\mathfrak{b})\leqslant\mathfrak{b}!$.
By Lemma~\ref{sh55}\ref{sh63}, $\mathcal{S}_3(\mathfrak{b})\nleqslant2^\mathfrak{b}$,
and thus $\mathfrak{b}!\nleqslant2^\mathfrak{b}$ and $\mathfrak{b}!\nleqslant[\mathfrak{b}]^3$.
By Lemma~\ref{sh55}\ref{sh64}, $[\mathfrak{b}]^3\nleqslant_{\mathrm{fto}}\mathfrak{b}!$,
and hence $[\mathfrak{b}]^3\nleqslant\mathfrak{b}!$ and $2^\mathfrak{b}\nleqslant\mathfrak{b}!$,
which completes the proof of (ii).

(iii) Let $\mathfrak{c}=\mathfrak{a}+\aleph_0$.
By Lemma~\ref{sh55}\ref{sh65}, $([\mathfrak{c}]^2)^2\nleqslant\mathfrak{c}!$.

(iv) Let $\mathfrak{d}=\mathfrak{a}+\aleph_0$.
For any set $x$, since all elements of $x^2$ are $2$-element subsets of $2\times x$,
we have $x^2\subseteq[2\times x]^2$.
Since it is easy to verify that $2\times y\preccurlyeq[y]^2$ for any infinite set $y$,
we get that
\[
\bigl([\mathfrak{d}]^2\bigr)^2\leqslant\bigl[2\cdot[\mathfrak{d}]^2\bigr]^2\leqslant\Bigl[\bigl[[\mathfrak{d}]^2\bigr]^2\Bigr]^2.
\]
Now $\Bigl[\bigl[[\mathfrak{d}]^2\bigr]^2\Bigr]^2\nleqslant\mathfrak{d}!$ follows from \ref{sh66}.
\end{proof}

\begin{xrem}
It is provable in $\mathsf{ZF}$ that
for all infinite cardinals $\mathfrak{a}$ and all natural numbers $n$,
$\mathfrak{a}^n\leqslant\seq^{1\text{-}1}(\mathfrak{a})$ (cf.~\cite[Lemma~2.5]{Rubin1971}).
Proposition~\ref{sh60}\ref{sh67} shows that, in Corollary~\ref{sh34},
we cannot replace $\mathfrak{a}^n$ by $\seq^{1\text{-}1}(\mathfrak{a})$.
Proposition~\ref{sh60}\ref{sh66} shows that we cannot generalize Corollary~\ref{sh30}
by proving that $([\mathfrak{a}]^2)^2<\mathfrak{a}!$,
even for Dedekind infinite cardinals $\mathfrak{a}$;
it also shows that, in Theorem~\ref{sh47},
we cannot conclude that $\seq(\calspdf(\mathfrak{a}))<\mathfrak{a}!$,
since $([\mathfrak{a}]^2)^2\leqslant\seq(\calspdf(\mathfrak{a}))$.
Proposition~\ref{sh60}\ref{sh69} is the consistency result stated after Corollary~\ref{sh32}.
\end{xrem}

\begin{proposition}\label{sh70}
The following statement is consistent with $\mathsf{ZF}$:
There is a Dedekind infinite cardinal $\mathfrak{b}$ such that
$\seq(\mathfrak{b})<[\mathfrak{b}]^2$ and
$[\mathfrak{b}]^2\nleqslant_{\mathrm{fto}}\seq(\mathfrak{b})$.
\end{proposition}
\begin{proof}
Let $A$ be the set of atoms of $\mathcal{V}_\mathrm{F}$,
let $\mathfrak{a}=|A|$, and let $\mathfrak{b}=\seq(\mathfrak{a})$.
Then by Fact~\ref{sh07}, $\mathfrak{b}$ is Dedekind infinite,
and by Lemma~\ref{sh08a}, $\seq(\mathfrak{b})=\mathfrak{b}$.
By Lemma~\ref{sh55}\ref{sh61}, $[\mathfrak{a}]^2\nleqslant_{\mathrm{fto}}\mathfrak{b}$,
and hence $[\mathfrak{b}]^2\nleqslant_{\mathrm{fto}}\mathfrak{b}$ and $\mathfrak{b}<[\mathfrak{b}]^2$.
Therefore, we get that $\seq(\mathfrak{b})=\mathfrak{b}<[\mathfrak{b}]^2$
and $[\mathfrak{b}]^2\nleqslant_{\mathrm{fto}}\mathfrak{b}=\seq(\mathfrak{b})$.
\end{proof}

\subsection{The ordered Mostowski model}
Let the set $A$ of atoms be denumerable,
and let $<_\mathrm{M}$ be an ordering of $A$ with order type that of the rational numbers.
Let $\mathcal{G}$ be the group of all automorphisms of $\langle A,<_\mathrm{M}\rangle$
and let $\mathcal{I}=\fin(A)$.
The permutation model determined by $\mathcal{G}$ and $\mathcal{I}$
is called the \emph{ordered Mostowski model}
(cf.~\cite[pp.~198--202]{Halbeisen2017} or \cite[\S4.5]{Jech1973}),
and is denoted by $\mathcal{V}_\mathrm{M}$ ($\mathrm{M}$ for Mostowski).

Clearly, the relation $<_\mathrm{M}$ belongs to
the model $\mathcal{V}_\mathrm{M}$ (cf.~\cite[Lemma~8.10]{Halbeisen2017}).
In $\mathcal{V}_\mathrm{M}$, $A$ is infinite
but power Dedekind finite (cf.~\cite[Lemma~8.13]{Halbeisen2017}),
and therefore, by Fact~\ref{sh11}, $\cals(A)$ is Dedekind finite.

\begin{fact}\label{sh71}
Let $A$ be the set of atoms of $\mathcal{V}_\mathrm{M}$.
In $\mathcal{V}_\mathrm{M}$, $\calsf(A)=\cals(A)$.
\end{fact}
\begin{proof}
Let $f\in\mathcal{V}_\mathrm{M}$ be a permutation of $A$,
and let $B\in\fin(A)$ be a support of $f$.
If there exists an $a\in\mov(f)\setminus B$,
then take a $\pi\in\fixg(B\cup\{f(a)\})$ such that $\pi(a)\neq a$.
Thus $\pi$ moves $f$, contradicting the assumption that $B$ is a support of $f$.
Therefore $\mov(f)\subseteq B$, and hence $f\in\calsf(A)$.
\end{proof}

\begin{lemma}\label{sh72}
For all non-zero cardinals $\mathfrak{a}$,
if there are $x$, $r$ such that $|x|=\mathfrak{a}$ and $r$ is an ordering of $x$,
then $\calsf(\mathfrak{a})\leqslant\seq^{1\text{-}1}(\mathfrak{a})\leqslant\seq(\mathfrak{a})=\aleph_0\cdot\calsf(\mathfrak{a})$.
Moreover, if in addition $\mathfrak{a}$ is a Dedekind finite cardinal,
then we have that $\calsf(\mathfrak{a})<\seq^{1\text{-}1}(\mathfrak{a})<\seq(\mathfrak{a})=\aleph_0\cdot\calsf(\mathfrak{a})$.
\end{lemma}
\begin{proof}
Let $\mathfrak{a}$ be a non-zero cardinal,
let $x$ be a set such that $|x|=\mathfrak{a}$,
and let $r$ be an ordering of $x$.
Let $f$ be the function defined on $\calsf(x)$ such that for all $t\in\calsf(x)$,
$f(t)$ is the unique isomorphism of $\langle\mov(t),r\rangle$ onto some natural number.
Then the function $g$ defined on $\calsf(x)$ given by
\[
g(t)=t\circ(f(t))^{-1}
\]
is an injection from $\calsf(x)$ into $\seq^{1\text{-}1}(x)$, since for all $t\in\calsf(x)$,
$f(t)$ is the unique isomorphism of $\langle\ran(g(t)),r\rangle$ onto some natural number,
and $t$ is the permutation of $x$ given by
\[
t(z)=
\begin{cases}
  g(t)(f(t)(z)), & \text{if $z\in\ran(g(t))$;}\\
  z,             & \text{otherwise.}
\end{cases}
\]
Note that $g$ is not surjective, since every sequence of length $1$ is not in the range of $g$.
Hence $\calsf(\mathfrak{a})\leqslant\seq^{1\text{-}1}(\mathfrak{a})$,
and therefore, by Lemma~\ref{sh08b},
$\aleph_0\cdot\calsf(\mathfrak{a})\leqslant\aleph_0\cdot\seq^{1\text{-}1}(\mathfrak{a})=\seq(\mathfrak{a})$,
which implies that, by Fact~\ref{sh13} and Theorem~\ref{cbt},
$\seq(\mathfrak{a})=\aleph_0\cdot\calsf(\mathfrak{a})$.
Moreover, if $\mathfrak{a}$ is Dedekind finite,
then, by Fact~\ref{sh06}, $\seq^{1\text{-}1}(\mathfrak{a})$ is Dedekind finite,
and thus, since $g$ is not surjective,
$\calsf(\mathfrak{a})<\seq^{1\text{-}1}(\mathfrak{a})$ follows from Theorem~\ref{dedt}.
Finally, since $\mathfrak{a}\neq0$ and $\mathfrak{a}$ is Dedekind finite,
by Corollary~\ref{sh09}, $\seq^{1\text{-}1}(\mathfrak{a})<\seq(\mathfrak{a})$.
\end{proof}

\begin{proposition}\label{sh73}
The following statement is consistent with $\mathsf{ZF}$: There is an infinite cardinal $\mathfrak{a}$ such that
$2^\mathfrak{a}<\mathfrak{a}!<\seq^{1\text{-}1}(\mathfrak{a})<\seq(\mathfrak{a})=\aleph_0\cdot\mathfrak{a}!$.
\end{proposition}
\begin{proof}
Let $A$ be the set of atoms of $\mathcal{V}_\mathrm{M}$ and let $\mathfrak{a}=|A|$.
In $\mathcal{V}_\mathrm{M}$, $<_\mathrm{M}$ is an ordering of $A$.
Since $A$ is power Dedekind finite, by Lemma~\ref{sh72}, we have that
$\calsf(\mathfrak{a})<\seq^{1\text{-}1}(\mathfrak{a})<\seq(\mathfrak{a})=\aleph_0\cdot\calsf(\mathfrak{a})$,
and hence, by Fact~\ref{sh71},
$\mathfrak{a}!<\seq^{1\text{-}1}(\mathfrak{a})<\seq(\mathfrak{a})=\aleph_0\cdot\mathfrak{a}!$.
Finally, $2^\mathfrak{a}<\mathfrak{a}!$ was proved in \cite{DawsonHoward1976}.
\end{proof}

Proposition~\ref{sh73} is the consistency result stated after Corollary~\ref{sh39}.
For more cardinal relations that hold in $\mathcal{V}_\mathrm{M}$, see \cite[p.~249]{HalbeisenShelah2001}.

\subsection{A Shelah-type permutation model}
In \cite[\S1]{HalbeisenShelah1994}, Shelah constructed a permutation model in which
there exists an infinite cardinal $\mathfrak{a}$ such that $\seq(\mathfrak{a})<\fin(\mathfrak{a})$.
Later, in \cite[\S7.3]{HalbeisenShelah2001}, a similar model was constructed
in order to show that the existence of an infinite cardinal $\mathfrak{a}$
such that $\mathfrak{a}^2<[\mathfrak{a}]^2$ is consistent with $\mathsf{ZF}$.
Recently, in \cite{Halbeisen2018}, Halbeisen generalized these two results
by proving that the existence of an infinite cardinal $\mathfrak{a}$
such that $\seq(\mathfrak{a})<[\mathfrak{a}]^2$ and
$[\mathfrak{a}]^2\nleqslant_{\mathrm{fto}}\seq(\mathfrak{a})$
is consistent with $\mathsf{ZF}$. These permutation models are called
\emph{Shelah-type permutation models} (cf.~\cite[pp.~209--211]{Halbeisen2017}).
The atoms of Shelah-type permutation models are always constructed by recursion,
where every atom encodes certain sets of atoms on a lower level.

Here we construct a Shelah-type permutation model in which
there exists a Dedekind infinite cardinal $\mathfrak{a}$ such that $\mathfrak{a}!<[\mathfrak{a}]^3$,
$[\mathfrak{a}]^3\nleqslant_{\mathrm{dfto}}\mathfrak{a}!$, and $\mathfrak{a}!\leqslant^\ast\mathfrak{a}$.
The proof of Proposition~\ref{sh70} shows that, in the basic Fraenkel model,
there already exists a Dedekind infinite cardinal $\mathfrak{b}$ such that
$\seq(\mathfrak{b})<[\mathfrak{b}]^2$ and $[\mathfrak{b}]^2\nleqslant_{\mathrm{fto}}\seq(\mathfrak{b})$.
Hence, in such a case, we do not really need to construct new models.
However, for our purpose here, the proof of Proposition~\ref{sh70} does not work,
because, unlike the case for $\seq(\mathfrak{a})$, $(\mathfrak{a}!)!=\mathfrak{a}!$ does not hold;
in fact, by Corollary~\ref{sh30}, $\mathfrak{a}!<(\mathfrak{a}!)!$ for any infinite cardinal $\mathfrak{a}$.

In this subsection, we shall work in $\mathsf{ZFA}+\mathsf{AC}$.
For a set $x$, let $\calsc(x)$ be the set of
all permutations of $x$ which move only countably many elements.
The atoms of this Shelah-type permutation model are constructed as follows:
\begin{enumerate}[label=\upshape(\roman*)]
  \item $A_0$ is an arbitrary uncountable set of atoms.
  \item $\mathcal{G}_0=\cals(A_0)$.
  \item $A_{n+1}=A_n\cup\{(n,u,i)\mid u\in\calsc(A_n)\text{ and }i<3\}$.
  \item $\mathcal{G}_{n+1}$ is the subgroup of $\cals(A_{n+1})$ such that
        for all $h\in\cals(A_{n+1})$, $h\in\mathcal{G}_{n+1}$ if and only if
        there exists a $g\in\mathcal{G}_n$ such that
        \begin{itemize}
          \item $g=h\upharpoonright A_n$;
          \item for all $u\in\calsc(A_n)$, there exists a permutation $p$ of $\{0,1,2\}$ such that
                $h(n,u,i)=(n,g\circ u\circ g^{-1},p(i))$ for any $i<3$.
        \end{itemize}
\end{enumerate}
Let $A=\bigcup_{n\in\omega}A_n$. For each triple $(n,u,i)\in A$
we assign a new atom $a_{n,u,i}$ and define the set of atoms
by stipulating $\tilde{A}=A_0\cup\{a_{n,u,i}\mid(n,u,i)\in A\}$.
However, for the sake of simplicity we shall work with $A$
as the set of atoms rather than with $\tilde{A}$.
Now, let
\[
\mathcal{G}=\bigl\{\pi\in\cals(A)\bigm|\forall n\in\omega\bigl(\pi\upharpoonright A_n\in\mathcal{G}_n\bigr)\bigr\},
\]
and let
\[
\mathcal{I}=\bigl\{B\subseteq A\bigm|\exists n\in\omega\bigl(B\text{ is a countable subset of }A_n\bigr)\bigr\}.
\]
Obviously, $\mathcal{G}$ is a permutation group of $A$, and $\mathcal{I}$ is a normal ideal.
The permutation model determined by $\mathcal{G}$ and $\mathcal{I}$
is denoted by $\mathcal{V}_\mathrm{S}$ ($\mathrm{S}$ for Shelah).

We say that a subset $C$ of $A$ is \emph{closed} if for all triples $(n,u,i)\in C$,
$\mov(u)\subseteq C$ and $\{(n,u,j)\mid j<3\}\subseteq C$.
The \emph{closure} of a subset $B$ of $A$ is the least closed set that includes $B$.
Since we are working in $\mathsf{ZFA}+\mathsf{AC}$, it is easy to verify that
the closure of a countable subset of $A$ is also countable,
and therefore for all $B\in\mathcal{I}$, the closure of $B$ belongs to $\mathcal{I}$.

\begin{lemma}\label{sh74}
For all closed subsets $C$ of $A$ and all $m\in\omega$,
every $g\in\mathcal{G}_m$ fixing $C\cap A_m$ pointwise
extends to a permutation $\pi\in\fixg(C)$.
\end{lemma}
\begin{proof}
We define $h_n\in\mathcal{G}_{m+n}$ by recursion on $n$ as follows:
$h_0=g$; $h_{n+1}$ is the permutation of $A_{m+n+1}$
such that $h_n=h_{n+1}\upharpoonright A_{m+n}$
and such that for all $u\in\calsc(A_{m+n})$,
$h_{n+1}(m+n,u,i)=(m+n,h_n\circ u\circ h_n^{-1},i)$ for any $i<3$.
Now we prove by induction on $n$ that $h_n$ fixes $C\cap A_{m+n}$ pointwise.
By the assumption, $h_0$ fixes $C\cap A_m$ pointwise.
Assume, as an induction hypothesis, that $h_n$ fixes $C\cap A_{m+n}$ pointwise.
Then $h_{n+1}$ fixes $C\cap A_{m+n}$ pointwise, since $h_{n+1}$ extends $h_n$.
For any $(m+n,u,i)\in C$, since $C$ is closed,
we have that $\mov(u)\subseteq C\cap A_{m+n}$,
and therefore $h_n\circ u\circ h_n^{-1}=u$,
which implies that $h_{n+1}(m+n,u,i)=(m+n,u,i)$.
Hence $h_{n+1}$ fixes $C\cap A_{m+n+1}$ pointwise.
Let $\pi=\bigcup_{n\in\omega}h_n$.
Then $\pi\in\mathcal{G}$ extends $g$ and fixes $C$ pointwise.
\end{proof}

\begin{lemma}\label{sh75}
For all closed subsets $C$ of $A$ and all $n\in\omega$,
if $a$, $b$ are two distinct elements of $A$
such that $a\in A_{n+1}\setminus(A_n\cup C)$ and $b\in A_{n+1}\cup C$,
then there exists a permutation $\pi\in\fixg(C\cup A_n\cup\{b\})$ such that $\pi(a)\neq a$.
\end{lemma}
\begin{proof}
Let $a=(n,t,j)$, where $t\in\calsc(A_n)$ and $j<3$.
Let $l<3$ be the least natural number such that $(n,t,l)\notin\{a,b\}$ and let $p=(j;l)_3$.
Since $a\notin C$ and $C$ is closed, $(n,t,l)\notin C$.
Let $g$ be the permutation of $A_{n+1}$ such that $g$ fixes $A_n$ pointwise
and such that for all $u\in\calsc(A_n)$ and all $i<3$,
$g(n,u,i)=(n,u,p(i))$, if $u=t$, and $g(n,u,i)=(n,u,i)$, otherwise.
Then $g\in\mathcal{G}_{n+1}$ fixes $A_{n+1}\setminus\{a,(n,t,l)\}$ pointwise.
By Lemma~\ref{sh74}, $g$ extends to some $\pi\in\fixg(C)$.
Then $\pi\in\fixg(C\cup A_n\cup\{b\})$ and $\pi(a)=(n,t,l)\neq a$.
\end{proof}

\begin{lemma}\label{sh76}
In $\mathcal{V}_\mathrm{S}$, $\cals(A)=\{u\in\cals(A)\mid\mov(u)\in\mathcal{I}\}$.
\end{lemma}
\begin{proof}
Let $u\in\mathcal{V}_\mathrm{S}$ be a permutation of $A$,
and let $B\in\mathcal{I}$ be a support of $u$.
Let $C$ be the closure of $B$. Then $C\in\mathcal{I}$.
Assume towards a contradiction that there exists an $a\in\mov(u)\setminus C$. Let $b=u(a)\neq a$.

If $a\in A_0$ and $b\in A_0\cup C$,
then take an arbitrary $c\in A_0\setminus(C\cup\{a,b\})$ and let $g=(a;c)_{A_0}$.
By Lemma~\ref{sh74}, $g$ extends to some $\pi\in\fixg(C)$.
Then $\pi(a)=c\neq a$ and $\pi(b)=b$.
Hence $\pi$ moves $u$, contradicting the assumption that $B$ is a support of $u$.

If there is an $n\in\omega$ such that $a\in A_{n+1}\setminus A_n$ and $b\in A_{n+1}\cup C$,
then by Lemma~\ref{sh75}, there is a permutation $\sigma\in\fixg(C\cup\{b\})$ such that $\sigma(a)\neq a$.
Hence $\sigma$ moves $u$, contradicting the assumption that $B$ is a support of $u$.

Thus, $b\notin C$ and there exists an $m\in\omega$
such that $b\in A_{m+1}\setminus A_m$ and $a\in A_m$.
Again by Lemma~\ref{sh75}, there is a permutation $\tau\in\fixg(C\cup\{a\})$ such that $\tau(b)\neq b$.
Hence $\tau$ moves $u$, which is also a contradiction.

Therefore, we have $\mov(u)\subseteq C$, and hence $\mov(u)\in\mathcal{I}$.
\end{proof}

For all $n\in\omega$, since $\pi[A_n]=A_n$ for any $\pi\in\mathcal{G}$,
$A_n\in\mathcal{V}_\mathrm{S}$, and therefore
the function that maps each $n\in\omega$ to $A_n$ belongs to $\mathcal{V}_\mathrm{S}$.
For every $B\in\mathcal{I}$, let $k_B$ be the least $n\in\omega$ such that $B\subseteq A_n$.
Since for all $B\in\mathcal{I}$ and all $\pi\in\mathcal{G}$, $k_B=k_{\pi[B]}$,
the function that maps each $B\in\mathcal{I}$ to $k_B$ belongs to $\mathcal{V}_\mathrm{S}$.

\begin{lemma}\label{sh77}
Let $A$ be the set of atoms of $\mathcal{V}_\mathrm{S}$ and let $\mathfrak{a}=|A|$.
In $\mathcal{V}_\mathrm{S}$,
\begin{enumerate}[label=\upshape(\roman*)]
  \item $\mathfrak{a}$ is Dedekind infinite;
  \item $\mathfrak{a}!\leqslant[\mathfrak{a}]^3$ and $\mathfrak{a}!\leqslant^\ast\mathfrak{a}$;
  \item $[\mathfrak{a}]^3\nleqslant_{\mathrm{dfto}}\mathfrak{a}!$.
\end{enumerate}
\end{lemma}
\begin{proof}
(i) Let $q$ be an injection from $\omega$ into $A_0$.
Then $\ran(q)\in\mathcal{I}$, which implies that $q\in\mathcal{V}_\mathrm{S}$.
Hence, in $\mathcal{V}_\mathrm{S}$, $A$ is Dedekind infinite.

(ii) Let $\Phi$ be the function defined on $\{u\in\cals(A)\mid\mov(u)\in\mathcal{I}\}$ given by
\[
\Phi(u)=\bigl\{(k_{\mov(u)},u\upharpoonright A_{k_{\mov(u)}},i)\bigm|i<3\bigr\}.
\]
Then $\Phi$ is an injection from $\{u\in\cals(A)\mid\mov(u)\in\mathcal{I}\}$ into $[A]^3$
and the sets in the range of $\Phi$ are pairwise disjoint.
It is easy to verify that $\Phi\in\mathcal{V}_\mathrm{S}$.
In $\mathcal{V}_\mathrm{S}$, by Lemma~\ref{sh76},
$\cals(A)=\{u\in\cals(A)\mid\mov(u)\in\mathcal{I}\}$,
and thus $\Phi$ is an injection from $\cals(A)$ into $[A]^3$,
which implies that $\mathfrak{a}!\leqslant[\mathfrak{a}]^3$.
Since the sets in the range of $\Phi$ are pairwise disjoint,
we have $\mathfrak{a}!\leqslant^\ast\mathfrak{a}$.

(iii) Assume towards a contradiction that there exists a function $f\in\mathcal{V}_\mathrm{S}$
from $[A]^3$ into $\cals(A)$ such that
\begin{equation}\label{sh78}
\text{in $\mathcal{V}_\mathrm{S}$, $f$ is a Dedekind finite to one map.}
\end{equation}
Let $B\in\mathcal{I}$ be a support of $f$,
and let $C$ be the closure of $B$. Then $C\in\mathcal{I}$.

Let us now fix three distinct elements $a$, $b$, $c$ of
$A_0\setminus C$ and consider the permutation $u=f(\{a,b,c\})$.
We claim that
\begin{equation}\label{sh79}
\mov(u)\subseteq C\cup A_0.
\end{equation}
Assume towards a contradiction that there exists a $d\in\mov(u)\setminus(C\cup A_0)$.

If there is an $n\in\omega$ such that $d\in A_{n+1}\setminus A_n$ and $u(d)\in A_{n+1}\cup C$,
then by Lemma~\ref{sh75}, there exists a permutation
$\pi_0\in\fixg(C\cup A_0\cup\{u(d)\})$ such that $\pi_0(d)\neq d$.
Hence $\pi_0$ fixes $\{a,b,c\}$ but moves $u$,
contradicting the assumption that $B$ is a support of $f$.

Thus, $u(d)\notin C$ and there exists an $m\in\omega$
such that $u(d)\in A_{m+1}\setminus A_m$ and $d\in A_m$.
By Lemma~\ref{sh75}, there is a permutation
$\pi_1\in\fixg(C\cup A_0\cup\{d\})$ such that $\pi_1(u(d))\neq u(d)$.
Hence $\pi_1$ fixes $\{a,b,c\}$ but moves $u$,
contradicting again the assumption that $B$ is a support of $f$.
Thus \eqref{sh79} is proved.

If there exists an $e\in\mov(u)\setminus(C\cup\{a,b,c\})$,
then $u(e)\in\mov(u)$, and therefore, by \eqref{sh79},
$e\in A_0\setminus(C\cup\{a,b,c\})$ and $u(e)\in C\cup A_0$.
Take an arbitrary $v\in A_0\setminus(C\cup\{a,b,c,e,u(e)\})$ and let $g_0=(e;v)_{A_0}$.
By Lemma~\ref{sh74}, $g_0$ extends to a permutation $\sigma_0\in\fixg(C)$.
Then $\sigma_0\in\fixg(C\cup\{a,b,c\})$, $\sigma_0(e)=v\neq e$, and $\sigma_0(u(e))=u(e)$.
Hence $\sigma_0$ fixes $\{a,b,c\}$ but moves $u$,
contradicting that $B$ is a support of $f$.
Therefore $\mov(u)\subseteq C\cup\{a,b,c\}$.

If there exists a $z\in\mov(u)\cap\{a,b,c\}$,
then $u(z)\in\mov(u)$, and hence, by \eqref{sh79}, $u(z)\in C\cup A_0$.
Take a $w\in\{a,b,c\}\setminus\{z,u(z)\}$ and let $g_1=(z;w)_{A_0}$.
Again by Lemma~\ref{sh74}, $g_1$ extends to some $\sigma_1\in\fixg(C)$.
Then $\sigma_1(z)=w\neq z$ and $\sigma_1(u(z))=u(z)$.
Hence $\sigma_1(\{a,b,c\})=\{a,b,c\}$ but $\sigma_1(u)\neq u$, which is also a contradiction.
Therefore $\mov(u)\subseteq C$. Thus we have
\begin{equation}\label{sh80}
\forall t\in[A_0\setminus C]^3\bigl(\mov(f(t))\subseteq C\bigr).
\end{equation}

For any $t_0$, $t_1\in[A_0\setminus C]^3$,
it is easy to see that there exists an $h\in\mathcal{G}_0$ such that
$h$ fixes $C\cap A_0$ pointwise and such that $h[t_0]=t_1$.
By Lemma~\ref{sh74}, $h$ extends to a permutation $\tau\in\fixg(C)$.
Then $\tau(f)=f$ and $\tau(t_0)=t_1$, and hence, by \eqref{sh80}, $f(t_0)=f(t_1)$.
Therefore, $f$ maps all elements of $[A_0\setminus C]^3$ to the same element of $\cals(A)$.
Since $A_0$ is uncountable and $C$ is countable,
there exists an injection $p$ from $\omega$ into $A_0\setminus C$.
Then $\ran(p)\in\mathcal{I}$, which implies that $p\in\mathcal{V}_\mathrm{S}$.
Thus, in $\mathcal{V}_\mathrm{S}$, $A_0\setminus C$ is Dedekind infinite,
and hence $[A_0\setminus C]^3$ is Dedekind infinite,
contradicting \eqref{sh78}.
\end{proof}

\begin{theorem}\label{sh81}
The following statement is consistent with $\mathsf{ZF}$:
There exists a Dedekind infinite cardinal $\mathfrak{a}$ such that $\mathfrak{a}!<[\mathfrak{a}]^3$,
$[\mathfrak{a}]^3\nleqslant_{\mathrm{dfto}}\mathfrak{a}!$, and $\mathfrak{a}!\leqslant^\ast\mathfrak{a}$.
\end{theorem}
\begin{proof}
Let $A$ be the set of atoms of $\mathcal{V}_\mathrm{S}$ and let $\mathfrak{a}=|A|$.
Then by Lemma~\ref{sh77}, $\mathfrak{a}$ is Dedekind infinite,
$\mathfrak{a}!\leqslant[\mathfrak{a}]^3$, $\mathfrak{a}!\leqslant^\ast\mathfrak{a}$,
and $[\mathfrak{a}]^3\nleqslant_{\mathrm{dfto}}\mathfrak{a}!$.
Since $[\mathfrak{a}]^3\nleqslant_{\mathrm{dfto}}\mathfrak{a}!$,
we have that $[\mathfrak{a}]^3\nleqslant\mathfrak{a}!$,
and therefore $\mathfrak{a}!<[\mathfrak{a}]^3$.
\end{proof}

\section{A new permutation model}
In this section, we construct a permutation model in which
there exists an infinite cardinal $\mathfrak{a}$ such that
$\mathfrak{a}!\leqslant_{\mathrm{fto}}\mathfrak{a}$.
By Corollary~\ref{sh40} and Corollary~\ref{sh36},
such an infinite cardinal $\mathfrak{a}$ must be such that
$\mathfrak{a}!$ is Dedekind finite and such that
any permutation of a set of cardinality $\mathfrak{a}$ must fix at least one point.
Also, by Fact~\ref{sh10} and Fact~\ref{sh03},
such an infinite cardinal must be power Dedekind infinite.
The strategy of our construction is as follows:

We construct step-by-step an infinite lattice $A$ with a least element
such that every initial segment determined by an element of $A$ is finite.
The permutation model will then be determined by
the group of all automorphisms of $A$ and the normal ideal $\fin(A)$.
The lattice $A$ is constructed in a way such that it has enough automorphisms
(but not too much) to guarantee that every permutation of $A$
which has a finite support moves only finitely many elements.
Since the function that maps each finite subset of $A$ to
its least upper bound is a finite-to-one map from $\fin(A)$ into $A$,
by Fact~\ref{sh01}, it holds in the permutation model that
$\cals(A)=\calsf(A)\preccurlyeq_{\mathrm{fto}}\fin(A)\preccurlyeq_{\mathrm{fto}}A$.

In what follows, we consider a covering condition for partially ordered sets,
and then define the notion of a building block, which will be used in the construction of $A$.
Finally, we prove that $A$ has the desired properties.

\subsection{A covering condition}
Let $\langle P,<\rangle$ be a partially ordered set;
that is, $<$ is irreflexive and transitive.
For all $a$, $b\in P$, $a\leqslant b$ means that $a<b$ or $a=b$,
the \emph{initial segment} determined by $b$ is the set $\{c\in P\mid c\leqslant b\}$,
and the \emph{(closed) interval} from $a$ to $b$ is the set
$[a,b]=\{c\in P\mid a\leqslant c\leqslant b\}$.
We say that $\langle P,<\rangle$ is \emph{locally finite}
if for all $a$, $b\in P$, $[a,b]$ is finite.
Notice that if $\langle P,<\rangle$ has a least element,
then $\langle P,<\rangle$ is locally finite if and only if
every initial segment determined by an element of $P$ is finite.
For $a$, $b\in P$, we say that $a$ is \emph{covered} by $b$ (or $b$ \emph{covers} $a$),
denoted by $a\lessdot b$, if $a<b$ but $a<c<b$ for no $c\in P$.
For $b\in P$, we write $\cov(b)$ for the set $\{c\in P\mid c\lessdot b\}$
(i.e., the elements of $P$ covered by $b$).
A \emph{saturated chain} in an interval $[a,b]$ is a sequence $t\in\seq(P)$ of
length (i.e., the domain of $t$) $n>0$ such that
$t(0)=a$, $t(n-1)=b$, and $t(i)\lessdot t(i+1)$ for any $i<n-1$.

For all subsets $M$ of $P$, the least upper bound and the greatest lower bound of $M$,
if they exist, are denoted by $\sup M$ and $\inf M$, respectively.
Note that if $\langle P,<\rangle$ has a least element,
then the least upper bound of $\emptyset$ exists
and is the least element of $\langle P,<\rangle$.
We say that $\langle P,<\rangle$ is a \emph{lattice} if
any two elements of $P$ have a least upper bound and a greatest lower bound.
Note that if $\langle P,<\rangle$ is a lattice,
then any non-void finite subset $M$ of $P$ has
a least upper bound and a greatest lower bound.

\begin{fact}\label{sh82}
Let $\langle P,<\rangle$ be a locally finite lattice with a least element and let $\mathfrak{a}=|P|$.
Then $\fin(\mathfrak{a})\leqslant_{\mathrm{fto}}\mathfrak{a}$.
\end{fact}
\begin{proof}
The function that maps each $M\in\fin(P)$ to $\sup M$
is a finite-to-one map from $\fin(P)$ into $P$,
and hence $\fin(\mathfrak{a})\leqslant_{\mathrm{fto}}\mathfrak{a}$.
\end{proof}

\begin{definition}
A partially ordered set $\langle P,<\rangle$
satisfies the \emph{finitary lower covering condition}
if for all $M\in\fin(P)$ containing at least two elements,
\begin{equation}\label{sh83}
\exists b\in P\forall a\in M(a\lessdot b)\rightarrow\exists c\in P\forall a\in M(c\lessdot a).
\end{equation}
\end{definition}

\begin{xrem}
Let $\langle P,<\rangle$ be a lattice.
Then the statement that \eqref{sh83} holds for all $M\in[P]^2$ is equivalent to
the condition ($\xi'$) of \cite[p.~14]{Birkhoff1967}, which is in turn equivalent to
the usual \emph{lower covering condition} (cf.~\cite[p.~213]{Gratzer1998})
if $\langle P,<\rangle$ is locally finite. Locally finite lattices
satisfying the lower covering condition are often called \emph{Birkhoff lattices}.
\end{xrem}

\begin{lemma}\label{sh84}
Let $\langle P,<\rangle$ be a locally finite partially ordered set with a least element.
If $\langle P,<\rangle$ satisfies the finitary lower covering condition,
then the Jordan-Dedekind chain condition holds in $\langle P,<\rangle$;
that is, for any $a$, $b\in P$ such that $a\leqslant b$,
all saturated chains in $[a,b]$ have the same length.
\end{lemma}
\begin{proof}
Let $o$ be the least element of $\langle P,<\rangle$.
Clearly, it suffices to prove that for any $b\in P$,
all saturated chains in $[o,b]$ have the same length.
Now, we prove by induction on $n>0$ that for any $b\in P$,
if there exists a saturated chain in $[o,b]$ of length $n$,
then all saturated chains in $[o,b]$ have length $n$.
The case $n=1$ is obvious. For any $b\in P$,
let $t$ be a saturated chain in $[o,b]$ of length $n+1$ (where $n>0$),
let $u$ be an arbitrary saturated chain in $[o,b]$,
and let the length of $u$ be $m+1$. It suffices to show that $m=n$.
Clearly $m\neq0$. If $t(n-1)=u(m-1)$ then by the induction hypothesis
all saturated chains in $[o,u(m-1)]$ have length $n$, and therefore $m=n$.
Otherwise, since $b=t(n)=u(m)$ covers both $t(n-1)$ and $u(m-1)$,
by \eqref{sh83}, there exists a $c\in P$ covered by both $t(n-1)$ and $u(m-1)$.
Since $o\leqslant c$, we can find a saturated chain $s$ in $[o,c]$.
By the induction hypothesis, all saturated chains in $[o,t(n-1)]$ have length $n$,
and therefore, since $t(n-1)$ covers $c$, the length of $s$ is $n-1$.
Since $u(m-1)$ covers $c$, again by the induction hypothesis,
all saturated chains in $[o,u(m-1)]$ have length $n$, and therefore $m=n$.
\end{proof}

Let $\langle P,<\rangle$ be a locally finite partially ordered set with a least element~$o$,
and assume that $\langle P,<\rangle$ satisfies the finitary lower covering condition.
By Lemma~\ref{sh84}, for any $b\in P$,
all saturated chains in the interval $[o,b]$ have the same length $n>0$;
the \emph{height} of $b$, denoted by $\hght(b)$, is defined to be $n-1$.
Notice that for all $a$, $b\in P$, $a\lessdot b$ if and only if $a<b$ and $\hght(a)+1=\hght(b)$.

Clearly, if $\langle P,<\rangle$ is a locally finite lattice with a least element,
then $\langle P,<\rangle$ satisfies the finitary lower covering condition
if and only if for all $b\in P$ such that $\cov(b)$ contains at least two elements,
\begin{equation}\label{sh85}
\forall a\in\cov(b)\bigl(\inf\cov(b)\lessdot a\bigr).
\end{equation}

\begin{lemma}\label{sh86}
Let $\langle P,<\rangle$ be a locally finite lattice with a least element.
If $\langle P,<\rangle$ satisfies the finitary lower covering condition,
then for all $a$, $b\in P$ such that $a<b$ but not $a\leqslant\inf\cov(b)$,
\begin{enumerate}[label=\upshape(\roman*)]
  \item there exists a $c\leqslant\inf\cov(b)$ such that $c\lessdot a$
        and for all $d\leqslant\inf\cov(b)$, if $d<a$ then $d\leqslant c$;\label{sh87}
  \item $\inf\cov(a)\leqslant\inf\cov(b)$;\label{sh88}
  \item there exists a unique saturated chain in $[a,b]$.\label{sh91}
\end{enumerate}
\end{lemma}
\begin{proof}
(i) Fix an arbitrary $b\in P$.
We prove by induction on $n<\hght(b)$ that
for all $a\in P$ such that $a<b$ but not $a\leqslant\inf\cov(b)$,
if $\hght(a)=\hght(b)-n-1$, then there exists a $c\leqslant\inf\cov(b)$
such that $c\lessdot a$ and for all $d\leqslant\inf\cov(b)$, if $d<a$ then $d\leqslant c$.
If $n=0$, then $a\lessdot b$ and it suffices to take $c=\inf\cov(b)$.
Now, let $\hght(a)=\hght(b)-n-2$, where $n+1<\hght(b)$.
Let $v\in P$ be such that $v<b$ and $a\lessdot v$.
By the induction hypothesis, there exists a $w\leqslant\inf\cov(b)$
such that $w\lessdot v$ and for all $d\leqslant\inf\cov(b)$, if $d<v$ then $d\leqslant w$.
Take $c=\inf\{a,w\}$. Then $c\leqslant w\leqslant\inf\cov(b)$ and, by \eqref{sh83}, $c\lessdot a$.
Moreover, for all $d\leqslant\inf\cov(b)$, if $d<a$, then, since $d<a\leqslant v$,
we get $d\leqslant w$, and therefore $d\leqslant c$, which completes the proof of (i).

(ii) By \ref{sh87}, there exists a $c\leqslant\inf\cov(b)$ such that $c\lessdot a$,
and therefore, $\inf\cov(a)\leqslant c\leqslant\inf\cov(b)$.

(iii) Assume towards a contradiction that
there are two distinct saturated chains $t$, $u$ in $[a,b]$.
Then by Lemma~\ref{sh84}, $t$ and $u$ have the same length $n>0$.
Let $m=\max\{i<n\mid t(i)\neq u(i)\}$. Since $b=t(n-1)=u(n-1)$, $m<n-1$.
Let $c=t(m+1)=u(m+1)$. Then $c$ covers both $t(m)$ and $u(m)$.
By \eqref{sh85}, $\inf\cov(c)$ is covered by $t(m)$ and $u(m)$.
Thus $\inf\cov(c)=\inf\{t(m),u(m)\}$, and hence $a\leqslant\inf\cov(c)$.
If $c=b$ or $c\leqslant\inf\cov(b)$, then $a\leqslant\inf\cov(b)$, which is a contradiction.
Otherwise, $c<b$ but not $c\leqslant\inf\cov(b)$, and thus, by \ref{sh88},
$a\leqslant\inf\cov(c)\leqslant\inf\cov(b)$, which is also a contradiction.
\end{proof}

\subsection{Building blocks}
We define the notion of a building block as follows:

\begin{definition}
A \emph{building block} is a non-void finite lattice $\langle P,<\rangle$
satisfying the finitary lower covering condition
and such that for all $b\in P$, if $\hght(b)=2$ then $|\cov(b)|=4$,
and if $\hght(b)>2$ then
\begin{equation}\label{sh89}
\text{for all }c\lessdot\inf\cov(b),\ \bigl|\{a\in\cov(b)\mid\inf\cov(a)=c\}\bigr|=4.
\end{equation}
\end{definition}

Let $\langle P,<\rangle$ be a building block, let $e$ be the greatest element of $\langle P,<\rangle$,
and let $o$ be the least element of $\langle P,<\rangle$.
Clearly, for all $b\in P$, if $\hght(b)=0$ then $b=o$ and $\cov(b)=\emptyset$,
if $\hght(b)=1$ then $\cov(b)=\{o\}$, and if $\hght(b)=2$ then $\inf\cov(b)=o$.
Note also that for all $b\in P$ such that $\hght(b)\geqslant2$, $|\cov(b)|\geqslant4$,
and hence, by \eqref{sh85}, $\inf\cov(b)$ is covered by every $a\in\cov(b)$.

Let $Q=\{c\in P\mid c\leqslant\inf\cov(e)\}$. Note that $\langle Q,<\rangle$ is a building block.
Let $a\in P\setminus(Q\cup\{e\})$.
By Lemma~\ref{sh86}\ref{sh91}, there exists a unique saturated chain in $[a,e]$,
and therefore there exists a unique $c\in P$ such that $a\lessdot c$;
we~use $\scc(a)$ to denote the unique $c\in P$ such that $a\lessdot c$. Clearly,
\begin{equation}\label{sh92}
\scc(a)\in P\setminus Q\wedge a\lessdot\scc(a)\wedge\forall b\in P\bigl(a<b\leftrightarrow\scc(a)\leqslant b\bigr).
\end{equation}
Let $\prd(a)=\inf\cov(\scc(a))$. We claim that
\begin{equation}\label{sh90}
\prd(a)\in Q\wedge\prd(a)\lessdot a\wedge\forall d\in Q\bigl(d<a\leftrightarrow d\leqslant\prd(a)\bigr).
\end{equation}
In fact, by Lemma~\ref{sh86}\ref{sh88}, $\prd(a)\in Q$.
Since $a\notin Q$, we have that $a\neq o$,
and hence $\hght(a)\geqslant1$ and $\hght(\scc(a))\geqslant2$,
which implies that $\prd(a)\lessdot a$.
On the other hand, by Lemma~\ref{sh86}\ref{sh87}, there is a $c\in Q$
such that $c\lessdot a$ and for all $d\in Q$, if $d<a$ then $d\leqslant c$.
Since $\prd(a)\in Q$ and $\prd(a)\lessdot a$, we have $\prd(a)=c$,
and hence for all $d\in Q$, $d<a$ if and only if $d\leqslant\prd(a)$.
Thus \eqref{sh90} is proved. By \eqref{sh90},
$\prd(a)$ is the unique $c\in Q$ such that $c\lessdot a$.
Notice that if $\hght(a)\geqslant2$ then $\inf\cov(a)\lessdot\prd(a)$,
and hence if $\hght(a)>2$ then $\inf\cov(\prd(a))\lessdot\inf\cov(a)$.

Let $C=\{b\in P\setminus Q\mid\hght(b)=2\}$, and let
\[
D=\bigl\{(b,c)\in(P\setminus Q)\times P\bigm|\hght(b)>2\text{ and }c\lessdot\inf\cov(b)\bigr\}.
\]
For any $b\in C$, let $k_b=|\cov(b)\setminus Q|$, and for any $(b,c)\in D$, let
\[
l_{b,c}=\bigl|\{a\in\cov(b)\mid\inf\cov(a)=c\}\setminus Q\bigr|.
\]
Then it is easy to verify that for all $b\in C$,
\begin{equation}\label{sh93}
k_b=
\begin{cases}
  3, & \text{if $b\neq e$;}\\
  4, & \text{if $b=e$,}
\end{cases}
\end{equation}
and that for all $(b,c)\in D$,
\begin{equation}\label{sh94}
l_{b,c}=
\begin{cases}
  3, & \text{if $b\neq e$ and $\inf\cov(\prd(b))=c$;}\\
  4, & \text{if $b\neq e$ and $\inf\cov(\prd(b))\neq c$;}\\
  4, & \text{if $b=e$.}
\end{cases}
\end{equation}

Let $\sigma$ be a function defined on $C$ such that for all $b\in C$,
$\sigma(b)$ is a bijection from $\cov(b)\setminus Q$ onto $k_b$,
and let $\tau$ be a function defined on $D$ such that for all $(b,c)\in D$,
$\tau(b,c)$ is a bijection from $\{a\in\cov(b)\mid\inf\cov(a)=c\}\setminus Q$ onto $l_{b,c}$.
Such functions $\sigma$ and $\tau$ exist since $P$ is finite.
Let $p$ be a function on $C$ such that for all $b\in C$, $p(b)$ is a permutation of $k_b$,
and let $q$ be a function on $D$ such that for all $(b,c)\in D$, $q(b,c)$ is a permutation of $l_{b,c}$.
Let $f$ be an automorphism of $\langle Q,<\rangle$.
We shall define an automorphism $g$ of $\langle P,<\rangle$ extending $f$ as follows:

For each $d\in Q$, let $g(d)=f(d)$. Take $g(e)=e$.
Now, we assume that $a\in P\setminus(Q\cup\{e\})$
and that for all $b\in P\setminus Q$ such that $\hght(b)=\hght(a)+1$,
$g(b)$ is defined and we have:
\begin{gather}
g(b)\in P\setminus Q\wedge\hght(g(b))=\hght(b);\label{sh95}\\
\inf\cov(g(b))=f(\inf\cov(b));\label{sh96}\\
b\neq e\rightarrow\prd(g(b))=f(\prd(b)).\label{sh97}
\end{gather}
Take $b=\scc(a)$. Then by \eqref{sh92}, $a\lessdot b\in P\setminus Q$,
and thus $\hght(b)=\hght(a)+1$,
which implies that $g(b)$ is defined and \eqref{sh95}--\eqref{sh97} hold.
We consider the following two cases:

If $\hght(a)=1$, then $\hght(b)=2$ and hence $b\in C$.
By \eqref{sh95}, $g(b)\in P\setminus Q$ and $\hght(g(b))=\hght(b)=2$, and thus $g(b)\in C$.
Since $g(b)=e$ if and only if $b=e$, by \eqref{sh93}, $k_{g(b)}=k_b$.
Now we define
\begin{equation}\label{sh98}
g(a)=\bigl(\sigma(g(b))\bigr)^{-1}\Bigl(p(b)\bigl(\sigma(b)(a)\bigr)\Bigr).
\end{equation}
Then $g(a)\in\cov(g(b))\setminus Q$, and thus $\hght(g(a))=\hght(g(b))-1=1=\hght(a)$.
Since $\cov(g(a))=\cov(a)=\{o\}$ and $\prd(g(a))=\prd(a)=o$,
we get that \eqref{sh95}--\eqref{sh97} hold with $b$ replaced by $a$. Notice that
\begin{equation}\label{sh99}
\scc(g(a))=g(\scc(a)).
\end{equation}

If $\hght(a)>1$, then $\hght(b)>2$. Let $c=\inf\cov(a)$.
Then by \eqref{sh90}, we have $c\lessdot\prd(a)\in Q$, and hence $c\in Q$ and $(b,c)\in D$.
By \eqref{sh95}, $g(b)\in P\setminus Q$ and $\hght(g(b))=\hght(b)>2$.
Since $f$ is an automorphism of $\langle Q,<\rangle$,
we have $f(c)\lessdot f(\inf\cov(b))$, and thus, by \eqref{sh96},
$f(c)\lessdot\inf\cov(g(b))$, which implies that $(g(b),f(c))\in D$.
Since $\hght(g(b))=\hght(b)$, $g(b)=e$ if and only if $b=e$.
By \eqref{sh97} and the assumption that $f$ is an automorphism of $\langle Q,<\rangle$,
if $b\neq e$, then $\inf\cov(\prd(g(b)))=f(c)$ if and only if
$f(\inf\cov(\prd(b)))=f(c)$ if and only if $\inf\cov(\prd(b))=c$.
Thus by \eqref{sh94}, $l_{g(b),f(c)}=l_{b,c}$. Now we define
\begin{equation}\label{sh100}
g(a)=\bigl(\tau(g(b),f(c))\bigr)^{-1}\Bigl(q(b,c)\bigl(\tau(b,c)(a)\bigr)\Bigr).
\end{equation}
Then $g(a)\in\{v\in\cov(g(b))\mid\inf\cov(v)=f(c)\}\setminus Q$;
that is, $g(a)\in P\setminus Q$, $g(a)\lessdot g(b)$, and $\inf\cov(g(a))=f(\inf\cov(a))$.
Thus $\hght(g(a))=\hght(a)$ and
\begin{equation}\label{sh101}
\scc(g(a))=g(\scc(a)).
\end{equation}
Hence, by \eqref{sh96}, $\prd(g(a))=\inf\cov(g(b))=f(\inf\cov(b))=f(\prd(a))$,
and therefore we get that \eqref{sh95}--\eqref{sh97} hold with $b$ replaced by $a$.

Therefore, for all $b\in P\setminus Q$,
$g(b)$ is defined and \eqref{sh95}--\eqref{sh97} hold.
Also, by \eqref{sh99} and \eqref{sh101}, for all $a\in P\setminus(Q\cup\{e\})$,
\begin{equation}\label{sh104}
\scc(g(a))=g(\scc(a)).
\end{equation}

We still have to prove that $g$ is an automorphism of $\langle P,<\rangle$.
For this, we first prove that $g$ is injective.
Since $g\upharpoonright Q=f$ is injective, by \eqref{sh95},
it suffices to show that $g\upharpoonright(P\setminus Q)$ is injective.
We prove by induction on $n<\hght(e)$ that for all $a_0$, $a_1\in P\setminus Q$
such that $\hght(a_0)=\hght(e)-n$, if $g(a_0)=g(a_1)$ then $a_0=a_1$.
The case $n=0$ is obvious.
Let $n<\hght(e)-1$ and let $a_0$, $a_1\in P\setminus Q$ be such that
$\hght(a_0)=\hght(e)-n-1$ and $g(a_0)=g(a_1)$.
We have to prove that $a_0=a_1$.
By \eqref{sh95}, $\hght(a_1)=\hght(g(a_1))=\hght(g(a_0))=\hght(a_0)<\hght(e)$,
and hence $a_0$, $a_1\in P\setminus(Q\cup\{e\})$.
Let $b_0=\scc(a_0)$ and let $b_1=\scc(a_1)$.
Then by \eqref{sh92}, $b_0$, $b_1\in P\setminus Q$ and $\hght(b_0)=\hght(a_0)+1=\hght(e)-n$,
and hence, by the induction hypothesis, if $g(b_0)=g(b_1)$ then $b_0=b_1$.
By \eqref{sh104}, $g(b_0)=\scc(g(a_0))=\scc(g(a_1))=g(b_1)$, and thus $b_0=b_1$.
We consider the following two cases:

If $\hght(a_0)=1$, then $\hght(a_1)=1$.
Since $g(a_0)=g(a_1)$ and $b_0=b_1$, by \eqref{sh98},
$p(b_0)(\sigma(b_0)(a_0))=p(b_1)(\sigma(b_1)(a_1))$,
and therefore $\sigma(b_0)(a_0)=\sigma(b_1)(a_1)$, which implies that $a_0=a_1$.

If $\hght(a_0)>1$, then we have that $\hght(a_1)>1$.
Let $c_0=\inf\cov(a_0)$ and let $c_1=\inf\cov(a_1)$.
By \eqref{sh96}, $f(c_0)=\inf\cov(g(a_0))=\inf\cov(g(a_1))=f(c_1)$, and therefore $c_0=c_1$.
Since $g(a_0)=g(a_1)$, $b_0=b_1$, and $c_0=c_1$, by \eqref{sh100},
we have that $q(b_0,c_0)(\tau(b_0,c_0)(a_0))=q(b_1,c_1)(\tau(b_1,c_1)(a_1))$,
and therefore $\tau(b_0,c_0)(a_0)=\tau(b_1,c_1)(a_1)$, which implies that $a_0=a_1$.

Thus $g$ is injective, and hence $g$ is a permutation of $P$ since $P$ is finite.
It remains to show that for all $a$, $b\in P$,
\begin{equation}\label{sh102}
a<b\leftrightarrow g(a)<g(b).
\end{equation}
Let $a$, $b\in P$. If $b\in Q\cup\{e\}$, then obviously \eqref{sh102} holds.
Suppose that $b\in P\setminus(Q\cup\{e\})$.
Then by \eqref{sh95}, we have $g(b)\in P\setminus(Q\cup\{e\})$.
If $a\in Q$, then $g(a)=f(a)\in Q$, and therefore, by \eqref{sh90} and \eqref{sh97},
\[
a<b\leftrightarrow a\leqslant\prd(b)\leftrightarrow g(a)\leqslant\prd(g(b))\leftrightarrow g(a)<g(b).
\]
Thus if $a\in Q$ then \eqref{sh102} holds.
Also, if $a=e$, then \eqref{sh102} holds trivially.
Assume that $a\in P\setminus(Q\cup\{e\})$ and that
for all $c\in P\setminus Q$ such that $\hght(c)=\hght(a)+1$,
$c<b$ if and only if $g(c)<g(b)$.
Then by \eqref{sh92} and the injectivity of $g$,
we get that $\scc(a)\leqslant b$ if and only if $g(\scc(a))\leqslant g(b)$.
By~\eqref{sh95}, we have $g(a)\in P\setminus(Q\cup\{e\})$,
and therefore, by \eqref{sh92} and \eqref{sh104},
\[
a<b\leftrightarrow \scc(a)\leqslant b\leftrightarrow\scc(g(a))\leqslant g(b)\leftrightarrow g(a)<g(b).
\]
Thus \eqref{sh102} is proved.
We use $\Phi(P,<,\sigma,\tau,p,q,f)$ to denote the function $g$.
Hence we have proved that
\begin{equation}\label{sh103}
\text{$\Phi(P,<,\sigma,\tau,p,q,f)$ is an automorphism of $\langle P,<\rangle$ extending $f$.}
\end{equation}

Now, let $p_0$ be the function on $C$ such that for all $b\in C$, $p_0(b)=\mathrm{id}_{k_b}$,
and let $q_0$ be the function on $D$ such that for all $(b,c)\in D$, $q_0(b,c)=\mathrm{id}_{l_{b,c}}$.
Let $\Psi(P,<,\sigma,\tau,f)=\Phi(P,<,\sigma,\tau,p_0,q_0,f)$. Hence, by \eqref{sh103}, we have that
\begin{equation}\label{sh106}
\text{$\Psi(P,<,\sigma,\tau,f)$ is an automorphism of $\langle P,<\rangle$ extending $f$.}
\end{equation}

\begin{lemma}\label{sh105}
Let $\langle P,<\rangle$ be a building block,
let $e$ be the greatest element of~$\langle P,<\rangle$,
and let $Q=\{c\in P\mid c\leqslant\inf\cov(e)\}$.
For all $a\in P\setminus(Q\cup\{e\})$ and all $d\in P\setminus\{a\}$
such that either $\hght(d)\geqslant\hght(a)$ or $d\in Q$,
there exists an automorphism $g$ of $\langle P,<\rangle$
fixing $Q\cup\{d\}$ pointwise and such that $g(a)\neq a$.
\end{lemma}
\begin{proof}
Let $\sigma$ and $\tau$ be functions as above.
Let $b_0=\scc(a)$. We consider the following two cases:

If $\hght(a)=1$, then let $i=\sigma(b_0)(a)<k_{b_0}$ and
let $j<k_{b_0}$ be the least natural number such that $(\sigma(b_0))^{-1}(j)\notin\{a,d\}$.
Let $p$ be the function on $C$ such that for all $b\in C$,
\[
p(b)=
\begin{cases}
  (i;j)_{k_b},     & \text{if $b=b_0$;}\\
  \mathrm{id}_{k_b}, & \text{otherwise,}
\end{cases}
\]
and let $q$ be the function on $D$ such that for all $(b,c)\in D$,
$q(b,c)$ is the identity permutation of $l_{b,c}$.
Let $g=\Phi(P,<,\sigma,\tau,p,q,\mathrm{id}_Q)$.
Then by \eqref{sh103}, $g$ is an automorphism of $\langle P,<\rangle$ fixing $Q$ pointwise.
By \eqref{sh100} and a routine induction, we get that
for all $v\in P\setminus Q$ such that $\hght(v)>1$, $g(v)=v$.
Therefore, by \eqref{sh98}, $g(a)=(\sigma(b_0))^{-1}(j)\neq a$ and
for all $w\in P\setminus Q$ such that $\hght(w)=1$,
if $w\notin\{a,(\sigma(b_0))^{-1}(j)\}$, then $g(w)=w$. Hence $g(d)=d$.

If $\hght(a)>1$, then let $c_0=\inf\cov(a)$, let $i=\tau(b_0,c_0)(a)<l_{b_0,c_0}$,
and let~$j<l_{b_0,c_0}$ be the least natural number such that $(\tau(b_0,c_0))^{-1}(j)\notin\{a,d\}$.
Let $p$ be the function on $C$ such that for all $b\in C$, $p(b)$ is the identity permutation of $k_b$,
and let $q$ be the function on $D$ such that for all $(b,c)\in D$,
\[
q(b,c)=
\begin{cases}
  (i;j)_{l_{b,c}},     & \text{if $b=b_0$ and $c=c_0$;}\\
  \mathrm{id}_{l_{b,c}}, & \text{otherwise.}
\end{cases}
\]
Let $g=\Phi(P,<,\sigma,\tau,p,q,\mathrm{id}_Q)$.
By \eqref{sh103}, $g$ is an automorphism of $\langle P,<\rangle$ fixing $Q$ pointwise.
By \eqref{sh100} and a routine induction, we get that
for all $v\in P\setminus Q$ such that $\hght(v)>\hght(a)$, $g(v)=v$.
Therefore, again by \eqref{sh100}, $g(a)=(\tau(b_0,c_0))^{-1}(j)\neq a$ and
for all $w\in P\setminus Q$ such that $\hght(w)=\hght(a)$,
if~$w\notin\{a,(\tau(b_0,c_0))^{-1}(j)\}$, then $g(w)=w$.
Since $d\notin\{a,(\tau(b_0,c_0))^{-1}(j)\}$ and
either $\hght(d)\geqslant\hght(a)$ or $d\in Q$, we have $g(d)=d$.
\end{proof}

\subsection{The permutation model}
For any quintuple $(x_0,x_1,x_2,x_3,x_4)$ and for any $j<5$, let $\proj_j(x_0,x_1,x_2,x_3,x_4)=x_j$.
Let $o$ be an arbitrary atom. The atoms of the permutation model are constructed as follows:
\begin{enumerate}[label=\upshape(\roman*)]
  \item $e_0=o$, $A_0=\{o\}$, and ${\lessdot_0}=\emptyset$.
  \item $e_1=(0,0,\emptyset,o,3)$, $A_1=\{o,e_1\}$, and ${\lessdot_1}=\{(o,e_1)\}$.
  \item For any $n\geqslant1$, $e_{n+1}=(n,n,\emptyset,e_{n-1},3)$
        and $A_{n+1}=A_n\cup\bigcup_{i\leqslant n}B_{n,i}$,
        where $B_{n,i}$ is defined by recursion on $i\leqslant n$ as follows:
        \begin{itemize}
          \item $B_{n,0}=\{e_{n+1}\}$;
          \item $B_{n,n}=\{(n,0,b,o,j)\mid b\in B_{n,n-1}\wedge j<3\}$;
          \item $B_{n,i}=\{(n,n-i,b,c,j)\mid b\in B_{n,i-1}\wedge c\lessdot_n\proj_3 b\wedge j<L_{b,c}\}$,
                where $0<i<n$ and
                \[
                L_{b,c}=
                \begin{cases}
                  3, & \text{if $b\neq e_{n+1}$ and $\proj_3\proj_3\proj_2b=c$;}\\
                  4, & \text{if $b\neq e_{n+1}$ and $\proj_3\proj_3\proj_2b\neq c$;}\\
                  3, & \text{if $b=e_{n+1}$ and $c=e_{n-2}$;}\\
                  4, & \text{if $b=e_{n+1}$ and $c\neq e_{n-2}$.}
                \end{cases}
                \]
        \end{itemize}
  \item For any $n\geqslant1$, $\lessdot_{n+1}$ is defined as follows:
        \begin{align*}
        {\lessdot_{n+1}}={\lessdot_n} & \cup\bigl\{(e_n,e_{n+1})\bigr\} \\
                                      & \cup\bigl\{(\proj_3\proj_2a,a)\bigm|a\in A_{n+1}\setminus(A_n\cup\{e_{n+1}\})\bigr\} \\
                                      & \cup\bigl\{(a,\proj_2a)\bigm|a\in A_{n+1}\setminus(A_n\cup\{e_{n+1}\})\bigr\}.
        \end{align*}
  \item For any $n\in\omega$, $<_n$ is the transitive closure of $\lessdot_n$;
        that is, for all $a$, $b$, $a<_nb$ if and only if there exists a sequence $t$ of length $m>1$
        such that $t(0)=a$, $t(m-1)=b$, and $t(j)\lessdot_nt(j+1)$ for any $j<m-1$.
        Such a sequence $t$ is called a $\lessdot_n$-chain from $a$ to $b$.
\end{enumerate}
Let $A=\bigcup_{n\in\omega}A_n$ and let ${<}=\bigcup_{n\in\omega}{<_n}$.
For the sake of simplicity we shall work with $A$ as the set of atoms.
Let $\mathcal{G}$ be the group of all automorphisms of $\langle A,<\rangle$ and let $\mathcal{I}=\fin(A)$.
The permutation model determined by $\mathcal{G}$ and $\mathcal{I}$
is denoted by $\mathcal{V}_\mathcal{S}$ ($\mathcal{S}$ for the operator $\mathcal{S}$).

Clearly, for all $n\in\omega$, $e_n\in A_n$ and for all $a\in A_{n+1}\setminus A_n$,
$\proj_0a=n$ and if $n\geqslant1$ then $n-\proj_1a$ is the unique $i\leqslant n$ such that $a\in B_{n,i}$.
Therefore, for all $n\geqslant1$, $A_n$ and $\bigcup_{i\leqslant n}B_{n,i}$ are disjoint,
and the sets $B_{n,i}$ ($i\leqslant n$) are pairwise disjoint.
Notice that for all $n\geqslant1$ and all $a\in A_{n+1}\setminus(A_n\cup\{e_{n+1}\})$,
\begin{gather}
\proj_2a\in A_{n+1}\setminus A_n\wedge\proj_1\proj_2a=\proj_1a+1;\label{sh107}\\
\proj_1a>0\rightarrow\proj_3a\lessdot_n\proj_3\proj_2a.\label{sh108}
\end{gather}

\begin{lemma}\label{sh109}
For all $n\in\omega$, $\langle A_n,<_n\rangle$ is a building block,
$\lessdot_n$ is the covering relation of $<_n$,
$o$ is the least element of $\langle A_n,<_n\rangle$,
$e_n$ is the greatest element of $\langle A_n,<_n\rangle$,
and for all $a\in A_n\setminus\{o\}$, $\hght(a)=\proj_1a+1$ and $\inf\cov(a)=\proj_3a$.
\end{lemma}
\begin{proof}
We prove this lemma by induction on $n$. The cases $n=0$ and $n=1$ are obvious.
Next, for the inductive step, let $n\geqslant1$ and
assume that the assertion holds for $n$.
We prove that the assertion holds for $n+1$ as follows:

We first make some basic observations about $\lessdot_{n+1}$-chains.
Note that $\proj_3b\in A_n$ for any $b\in A_{n+1}\setminus A_n$, and hence,
by \eqref{sh107}, $\proj_3\proj_2a\in A_n$ for any $a\in A_{n+1}\setminus(A_n\cup\{e_{n+1}\})$.
Now, let $a<_{n+1}b$, and let $t$ be a $\lessdot_{n+1}$-chain of length $m>1$ from $a$ to $b$.
By \eqref{sh107} and the definition of $\lessdot_{n+1}$, if $b\in A_n$,
then $\ran(t)\subseteq A_n$ and $t$ is a $\lessdot_n$-chain from $a$ to $b$. Thus we have
\begin{equation}\label{sh114}
b\in A_n\wedge a<_{n+1}b\rightarrow a\in A_n\wedge a<_nb.
\end{equation}
If $a\in A_{n+1}\setminus A_n$, then $t[m-1]\subseteq A_{n+1}\setminus(A_n\cup\{e_{n+1}\})$, $t(m-1)\in A_{n+1}\setminus A_n$,
and for all $j<m-1$, $t(j+1)=\proj_2t(j)$ and thus $\proj_1t(j+1)=\proj_1t(j)+1$,
which implies that $m=\proj_1b-\proj_1a+1$ and hence $t$ is uniquely determined by $a$ and $b$. Therefore, we have that
\begin{equation}\label{sh113}
\begin{split}
&\text{if $a\in A_{n+1}\setminus A_n$ and $a<_{n+1}b$, then $\proj_1a<\proj_1b$,}\\
&\text{and there exists a unique $\lessdot_{n+1}$-chain from $a$ to $b$.}
\end{split}
\end{equation}
If $a$, $b\in A_{n+1}\setminus(A_n\cup\{e_{n+1}\})$, then $\ran(t)\subseteq A_{n+1}\setminus(A_n\cup\{e_{n+1}\})$
and $t(j+1)=\proj_2t(j)$ for any $j<m-1$, and therefore, by \eqref{sh107} and \eqref{sh108},
$\proj_3\proj_2t(j)\lessdot_n\proj_3\proj_2t(j+1)$ for any $j<m-1$. Thus we have
\begin{equation}\label{sh116}
a,b\in A_{n+1}\setminus(A_n\cup\{e_{n+1}\})\wedge a<_{n+1}b\rightarrow\proj_3\proj_2a<_n\proj_3\proj_2b.
\end{equation}

We claim that for all $a\in A_{n+1}\setminus(A_n\cup\{e_{n+1}\})$,
\begin{equation}\label{sh111}
\forall d\in A_n\bigl(d<_{n+1}a\rightarrow d\leqslant_n\proj_3\proj_2a\bigr).
\end{equation}
In fact, let $d\in A_n$, let $t$ be a $\lessdot_{n+1}$-chain of length $m>1$ from $d$ to $a$,
and let $i=\min\{j<m\mid t(j)\in A_{n+1}\setminus A_n\}$. Clearly, $i>0$, $t[i]\subseteq A_n$,
$t[m\setminus i]\subseteq A_{n+1}\setminus(A_n\cup\{e_{n+1}\})$, $d\leqslant_n t(i-1)$,
$t(i-1)=\proj_3\proj_2t(i)$, and $t(i)\leqslant_{n+1}a$, which implies that,
by \eqref{sh116}, $\proj_3\proj_2t(i)\leqslant_n\proj_3\proj_2a$.
Therefore, $d\leqslant_n t(i-1)=\proj_3\proj_2t(i)\leqslant_n\proj_3\proj_2a$. Thus \eqref{sh111} is proved.

Now we prove that $\langle A_{n+1},<_{n+1}\rangle$ is a partially ordered set.
Since $<_{n+1}$ is the transitive closure of $\lessdot_{n+1}$,
it suffices to prove that $<_{n+1}$ is irreflexive.
Assume towards a contradiction that there is a $b\in A_{n+1}$ such that $b<_{n+1}b$.
If $b\in A_n$, then, by \eqref{sh114}, $b<_nb$, contradicting the assumption that $<_n$ is irreflexive.
Otherwise, by \eqref{sh113}, $\proj_1b<\proj_1b$, which is also a contradiction.

$A_{n+1}$ is finite since $A_n$ is finite.
Since $o$ is the least element of $\langle A_n,<_n\rangle$,
$o\leqslant_ne_n\lessdot_{n+1}e_{n+1}$ and
$o\leqslant_n\proj_3\proj_2a\lessdot_{n+1}a$ for any $a\in A_{n+1}\setminus(A_n\cup\{e_{n+1}\})$,
which implies that $o$ is also the least element of $\langle A_{n+1},<_{n+1}\rangle$.
Since $e_n$ is the greatest element of $\langle A_n,<_n\rangle$,
we have $d\leqslant_ne_n\lessdot_{n+1}e_{n+1}$ for any $d\in A_n$.
For any $a\in A_{n+1}\setminus(A_n\cup\{e_{n+1}\})$,
the sequence $t$ of length $n-\proj_1a+1$ such that $t(0)=a$ and
$t(j+1)=\proj_2t(j)$ for any $j<n-\proj_1a$ is a $\lessdot_{n+1}$-chain from $a$ to $e_{n+1}$,
and therefore $a<_{n+1}e_{n+1}$, which implies that $e_{n+1}$ is the greatest element of $\langle A_{n+1},<_{n+1}\rangle$.

We prove that $\lessdot_{n+1}$ is the covering relation of $<_{n+1}$; that is, for all $a$, $b$,
\[
a\lessdot_{n+1}b\leftrightarrow a<_{n+1}b\wedge\neg\exists c\bigl(a<_{n+1}c<_{n+1}b\bigr).
\]
Clearly, if $a<_{n+1}b$ but $a<_{n+1}c<_{n+1}b$ for no $c\in A_{n+1}$, then $a\lessdot_{n+1}b$.
For the other direction, assume towards a contradiction that
$a\lessdot_{n+1}b$ and $a<_{n+1}c<_{n+1}b$ for some $c\in A_{n+1}$.
We consider the following four cases:
If $a\lessdot_nb$, then $b\in A_n$ and thus, by \eqref{sh114}, $a<_nc<_nb$,
contradicting the assumption that $\lessdot_n$ is the covering relation of $<_n$.
If $a=e_n$ and $b=e_{n+1}$, then we have that $c\in A_{n+1}\setminus(A_n\cup\{e_{n+1}\})$ and
the sequence $t$ of length $n-\proj_1c+1$ such that $t(0)=c$ and
$t(j+1)=\proj_2t(j)$ for any $j<n-\proj_1c$ is a $\lessdot_{n+1}$-chain from $c$ to $e_{n+1}$,
and thus, by \eqref{sh111} and \eqref{sh116}, we get that
$e_n\leqslant_n\proj_3\proj_2c\leqslant_n\proj_3\proj_2t(n-\proj_1c-1)=\proj_3e_{n+1}=e_{n-1}$, which is absurd.
If $b\in A_{n+1}\setminus(A_n\cup\{e_{n+1}\})$ and $a=\proj_3\proj_2b$,
then, since $c\in A_n$ implies that, by \eqref{sh111}, $c\leqslant_n\proj_3\proj_2b=a$,
we have $c\in A_{n+1}\setminus(A_n\cup\{e_{n+1}\})$,
and thus, by \eqref{sh111} and \eqref{sh116}, $a\leqslant_n\proj_3\proj_2c<_n\proj_3\proj_2b=a$, which is absurd.
Finally, if $a\in A_{n+1}\setminus(A_n\cup\{e_{n+1}\})$ and $b=\proj_2a$, then, by \eqref{sh113} and \eqref{sh107},
we have that $\proj_1a<\proj_1c<\proj_1b=\proj_1a+1$, which is also a contradiction.

Now we prove that $\langle A_{n+1},<_{n+1}\rangle$ is a lattice.
Since $A_{n+1}$ is finite and $\langle A_{n+1},<_{n+1}\rangle$ has a greatest element,
we only need to prove that any two elements of $A_{n+1}$ have a greatest lower bound.
Let $a$, $b\in A_{n+1}$. If $a\leqslant_{n+1}b$ or $b\leqslant_{n+1}a$,
then obviously $a$ and $b$ have a greatest lower bound.
Now, suppose that $a$ and $b$ are incomparable. If $a$, $b\in A_n$, then,
by \eqref{sh114}, the greatest lower bound of $a$ and $b$ in $\langle A_n,<_n\rangle$
is also their greatest lower bound in $\langle A_{n+1},<_{n+1}\rangle$.
If $a\in A_n$ and $b\in A_{n+1}\setminus(A_n\cup\{e_{n+1}\})$, then, by \eqref{sh114} and \eqref{sh111},
the greatest lower bound of $a$ and $\proj_3\proj_2b$ in $\langle A_n,<_n\rangle$
is also the greatest lower bound of $a$ and $b$ in $\langle A_{n+1},<_{n+1}\rangle$.
Finally, we claim that if $a$, $b\in A_{n+1}\setminus(A_n\cup\{e_{n+1}\})$,
then the greatest lower bound of $\proj_3\proj_2a$ and $\proj_3\proj_2b$ in $\langle A_n,<_n\rangle$
is the greatest lower bound of $a$ and $b$ in $\langle A_{n+1},<_{n+1}\rangle$.
By \eqref{sh111}, it suffices to show that for all $d\in A_{n+1}$,
if $d\leqslant_{n+1}a$ and $d\leqslant_{n+1}b$, then $d\in A_n$.
In fact, for all $c\in A_{n+1}\setminus(A_n\cup\{e_{n+1}\})$, by \eqref{sh113},
there exists a unique $\lessdot_{n+1}$-chain from $c$ to $e_{n+1}$,
and thus, since $a$ and $b$ are incomparable,
it cannot happen that $c\leqslant_{n+1}a$ and $c\leqslant_{n+1}b$ simultaneously.

We prove that $\langle A_{n+1},<_{n+1}\rangle$ satisfies the finitary lower covering condition.
Since $\langle A_n,<_n\rangle$ satisfies the finitary lower covering condition, by \eqref{sh114},
it suffices to prove that for all $b\in A_{n+1}\setminus A_n$
such that $\cov(b)$ contains at least two elements, \eqref{sh85} holds.
Since $\cov(b)$ contains at least two elements, by the definition of $\lessdot_{n+1}$,
$\cov(b)$ contains some element of $A_{n+1}\setminus(A_n\cup\{e_{n+1}\})$.
Let $a$ be an arbitrary element of $A_{n+1}\setminus(A_n\cup\{e_{n+1}\})$ such that $a\lessdot_{n+1}b$.
Then $b=\proj_2a$ and thus $\proj_3b=\proj_3\proj_2a\lessdot_{n+1}a$.
Note that if $b=e_{n+1}$ then $\cov(b)\cap A_n=\{e_n\}$ and $\proj_3b=e_{n-1}\lessdot_ne_n$,
and if $b\in A_{n+1}\setminus(A_n\cup\{e_{n+1}\})$ then $\cov(b)\cap A_n=\{\proj_3\proj_2b\}$ and,
by \eqref{sh107} and \eqref{sh108}, $\proj_3b\lessdot_n\proj_3\proj_2b$.
Therefore, $\proj_3b=\inf\cov(b)$ and \eqref{sh85} holds.
Hence we have proved that $\langle A_{n+1},<_{n+1}\rangle$ satisfies the finitary lower covering condition and
\begin{equation}\label{sh112}
\forall b\in A_{n+1}\setminus A_n\bigl(|\cov(b)|\geqslant2\rightarrow\inf\cov(b)=\proj_3b\bigr).
\end{equation}

Now, by Lemma~\ref{sh84}, in $\langle A_{n+1},<_{n+1}\rangle$,
the height of $b$ is well-defined for any $b\in A_{n+1}$.
Notice that for all $d\in A_n$, by \eqref{sh114},
the height of $d$ in $\langle A_{n+1},<_{n+1}\rangle$ is
the same as its height in $\langle A_n,<_n\rangle$. We claim that
\begin{equation}\label{sh115}
\forall a\in A_{n+1}\setminus\{o\}\bigl(\hght(a)=\proj_1a+1\bigr).
\end{equation}
Since in $\langle A_n,<_n\rangle$, $\hght(a)=\proj_1a+1$ for any $a\in A_n\setminus\{o\}$,
it suffices to prove that for all $b\in A_{n+1}\setminus A_n$, $\hght(b)=\proj_1b+1$.
Let $b\in A_{n+1}\setminus A_n$. If $b=e_{n+1}$, then, since $e_n\lessdot_{n+1}e_{n+1}$,
$\hght(b)=\hght(e_n)+1=\proj_1e_n+2=n+1=\proj_1b+1$.
Otherwise, the sequence $t$ of length $n-\proj_1b+1$ such that $t(0)=b$ and
$t(j+1)=\proj_2t(j)$ for any $j<n-\proj_1b$ is a $\lessdot_{n+1}$-chain from $b$ to $e_{n+1}$,
which implies that $\hght(b)+n-\proj_1b=\hght(e_{n+1})$ and hence $\hght(b)=\proj_1b+1$. Thus \eqref{sh115} is proved.

Finally, we prove that $\langle A_{n+1},<_{n+1}\rangle$ is a building block
and that for all $a\in A_{n+1}\setminus\{o\}$, $\inf\cov(a)=\proj_3a$.
Since $\langle A_n,<_n\rangle$ is a building block and in $\langle A_n,<_n\rangle$,
$\inf\cov(a)=\proj_3a$ for any $a\in A_n\setminus\{o\}$, by \eqref{sh114},
it suffices to prove that for all $b\in A_{n+1}\setminus A_n$, $\inf\cov(b)=\proj_3b$,
if $\hght(b)=2$ then $|\cov(b)|=4$, and if $\hght(b)>2$ then \eqref{sh89} holds.
Let $b\in A_{n+1}\setminus A_n$. We consider the following three cases:

If $\hght(b)=1$, then, by \eqref{sh115}, $\proj_1b=0$, and hence $\inf\cov(b)=o=\proj_3b$.

If $\hght(b)=2$, then, by \eqref{sh115}, $\proj_1b=1$, which implies that $b\in B_{n,n-1}$
and $\cov(b)\cap(A_{n+1}\setminus A_n)=\{(n,0,b,o,j)\mid j<3\}$.
Since $\cov(b)\cap A_n$ is a singleton, we have $|\cov(b)|=4$,
and therefore, by \eqref{sh112}, $\inf\cov(b)=\proj_3b$.

If $\hght(b)>2$, then, by \eqref{sh115}, $\proj_1b>1$. We further consider two subcases:

If $b=e_{n+1}$, then we have $n=\proj_1b>1$ and therefore $\proj_3b=e_{n-1}\neq o$.
For all $c\lessdot_n\proj_3b$, if $c=e_{n-2}$ then $L_{b,c}=3$ and hence
\[
\bigl|\{a\in\cov(b)\mid\proj_3a=c\}\bigr|=\bigl|\{e_n\}\cup\{(n,n-1,b,c,j)\mid j<L_{b,c}\}\bigr|=4,
\]
and if $c\neq e_{n-2}$ then $L_{b,c}=4$ and hence
\[
\bigl|\{a\in\cov(b)\mid\proj_3a=c\}\bigr|=\bigl|\{(n,n-1,b,c,j)\mid j<L_{b,c}\}\bigr|=4.
\]
Thus we have $|\cov(b)|\geqslant4$, which implies that, by \eqref{sh112}, $\inf\cov(b)=\proj_3b$.

If $b\in A_{n+1}\setminus(A_n\cup\{e_{n+1}\})$, then,
by \eqref{sh108}, $\proj_3b\lessdot_n\proj_3\proj_2b\lessdot_{n+1}b$,
which implies that $\hght(\proj_3b)=\hght(b)-2>0$ and hence we have $\proj_3b\neq o$.
For all $c\lessdot_n\proj_3b$, if $\proj_3\proj_3\proj_2b=c$ then $L_{b,c}=3$ and hence
\begin{multline*}
\bigl|\{a\in\cov(b)\mid\proj_3a=c\}\bigr|\\
=\bigl|\{\proj_3\proj_2b\}\cup\{(n,\proj_1b-1,b,c,j)\mid j<L_{b,c}\}\bigr|=4,
\end{multline*}
and if $\proj_3\proj_3\proj_2b\neq c$ then $L_{b,c}=4$ and hence
\[
\bigl|\{a\in\cov(b)\mid\proj_3a=c\}\bigr|=\bigl|\{(n,\proj_1b-1,b,c,j)\mid j<L_{b,c}\}\bigr|=4.
\]
Thus we have $|\cov(b)|\geqslant4$, which implies that, by \eqref{sh112}, $\inf\cov(b)=\proj_3b$.

Now, since in all cases we have $\inf\cov(b)=\proj_3b$,
we can replace $\proj_3b$ by $\inf\cov(b)$ and $\proj_3a$ by $\inf\cov(a)$ in the above two subcases,
and hence \eqref{sh89} holds in both subcases, which completes the proof.
\end{proof}

\begin{corollary}\label{sh110}
$\langle A,<\rangle$ is a locally finite lattice with a least element.
\end{corollary}
\begin{proof}
By Lemma~\ref{sh109}, for all $n\in\omega$,
$\langle A_n,<_n\rangle$ is a finite lattice
and $o$ is the least element of $\langle A_n,<_n\rangle$.
Hence, by \eqref{sh114}, $\langle A,<\rangle$ is a locally finite lattice
and $o$ is the least element of $\langle A,<\rangle$.
\end{proof}

\begin{lemma}\label{sh117}
For all $m\in\omega$, every automorphism of $\langle A_m,<_m\rangle$
extends to an automorphism of $\langle A,<\rangle$.
\end{lemma}
\begin{proof}
Let $m\in\omega$ and let $g$ be an automorphism of $\langle A_m,<_m\rangle$.
We define an automorphism $\pi$ of $\langle A,<\rangle$ extending $g$ as follows:

For each $n\in\omega$, let
\[
C_n=\bigl\{b\in A_{m+2n+2}\setminus A_{m+2n}\bigm|\proj_1b=1\bigr\},
\]
let
\[
D_n=\bigl\{(b,c)\bigm|b\in A_{m+2n+2}\setminus A_{m+2n}\wedge\proj_1b>1\wedge c\lessdot_{m+2n}\proj_3b\bigr\},
\]
let $\sigma_n$ be the function on $C_n$ such that for all $b\in C_n$,
$\sigma_n(b)$ is the function defined on $\{a\in A_{m+2n+2}\setminus A_{m+2n}\mid a\lessdot_{m+2n+2}b\}$
given by $\sigma_n(b)(a)=\proj_4a$,
and let $\tau_n$ be the function on $D_n$ such that for all $(b,c)\in D_n$,
$\tau_n(b,c)$ is the function defined on $\{a\in A_{m+2n+2}\setminus A_{m+2n}\mid a\lessdot_{m+2n+2}b\wedge\proj_3a=c\}$
given by $\tau_n(b,c)(a)=\proj_4a$.
We define $h_n$ by recursion on $n$ as follows:
\begin{align*}
  h_0     & =g; \\
  h_{n+1} & =\Psi(A_{m+2n+2},<_{m+2n+2},\sigma_n,\tau_n,h_n),
\end{align*}
where $\Psi$ is defined before Lemma~\ref{sh105}.
Therefore, by Lemma~\ref{sh109} and \eqref{sh106}, a routine induction shows that for all $n\in\omega$,
$h_{n+1}$ is an automorphism of $\langle A_{m+2n+2},<_{m+2n+2}\rangle$ extending $h_n$.
It suffices to take $\pi=\bigcup_{n\in\omega}h_n$.
\end{proof}

\begin{lemma}\label{sh118}
In $\mathcal{V}_\mathcal{S}$, $\cals(A)=\calsf(A)$.
\end{lemma}
\begin{proof}
Let $u\in\mathcal{V}_\mathcal{S}$ be a permutation of $A$,
and let $B\in\fin(A)$ be a support of $u$.
Let $k$ be the least natural number such that $B\subseteq A_k$.
We claim that
\[
\mov(u)\subseteq A_k.
\]
In fact, assume towards a contradiction that there is an $a\in\mov(u)\setminus A_k$.
Let $n=\proj_0a$ and let $b=u(a)\neq a$. Then $a\in A_{n+1}\setminus A_n$ and hence $k\leqslant n$.

If $b\in A_n$, or if $b\in A_{n+1}\setminus A_n$ and $\proj_1b\geqslant\proj_1a$,
then, by Lemma~\ref{sh109} and Lemma~\ref{sh105},
there exists an automorphism $g$ of $\langle A_{n+2},<_{n+2}\rangle$
fixing $A_n\cup\{b\}$ pointwise and such that $g(a)\neq a$.
By Lemma~\ref{sh117}, $g$ extends to an automorphism $\pi$ of $\langle A,<\rangle$.
Then we have $\pi\in\fixg(B\cup\{b\})$ and $\pi(a)\neq a$.
Hence $\pi$ moves $u$, contradicting the assumption that $B$ is a support of $u$.

Thus, $b\notin A_n$, and if $b\in A_{n+1}\setminus A_n$ then $\proj_1b<\proj_1a$.
Let $m=\proj_0b$. Then $b\in A_{m+1}\setminus A_m$ and hence $n\leqslant m$,
which implies that either $a\in A_m$ or $a\in A_{m+1}\setminus A_m$ and $\proj_1a>\proj_1b$.
Hence, by Lemma~\ref{sh109} and Lemma~\ref{sh105},
there exists an automorphism $h$ of $\langle A_{m+2},<_{m+2}\rangle$
fixing $A_m\cup\{a\}$ pointwise and such that $h(b)\neq b$.
By Lemma~\ref{sh117}, $h$ extends to an automorphism $\sigma$ of $\langle A,<\rangle$.
Then we have $\sigma\in\fixg(B\cup\{a\})$ and $\sigma(b)\neq b$.
Hence $\sigma$ moves $u$, contradicting again the assumption that $B$ is a support of $u$.

Thus $\mov(u)\subseteq A_k$. Since $A_k$ is finite, we have $u\in\calsf(A)$.
\end{proof}

\begin{corollary}\label{sh119}
Let $A$ be the set of atoms of $\mathcal{V}_\mathcal{S}$ and let $\mathfrak{a}=|A|$.
In $\mathcal{V}_\mathcal{S}$, we have $\mathfrak{a}!\leqslant_{\mathrm{fto}}\mathfrak{a}$.
\end{corollary}
\begin{proof}
By Lemma~\ref{sh118}, $\mathfrak{a}!=\calsf(\mathfrak{a})$,
and by Fact~\ref{sh01}, $\calsf(\mathfrak{a})\leqslant_{\mathrm{fto}}\fin(\mathfrak{a})$.
Also, by Corollary~\ref{sh110} and Fact~\ref{sh82}, we have $\fin(\mathfrak{a})\leqslant_{\mathrm{fto}}\mathfrak{a}$.
Therefore, we get that $\mathfrak{a}!=\calsf(\mathfrak{a})\leqslant_{\mathrm{fto}}\fin(\mathfrak{a})\leqslant_{\mathrm{fto}}\mathfrak{a}$.
\end{proof}

Now the following theorem immediately follows from Corollary~\ref{sh119} and the Jech-Sochor theorem.

\begin{theorem}\label{sh120}
The following statement is consistent with $\mathsf{ZF}$:
There exists an infinite cardinal $\mathfrak{a}$ such that $\mathfrak{a}!\leqslant_{\mathrm{fto}}\mathfrak{a}$.
\end{theorem}

\section{Conclusion}
Three open problems posed in \cite{Tachtsis2018} are solved in this paper:

(i) As a special case of Corollary~\ref{sh36}, we get that for all infinite sets $x$,
if there exists a permutation of $x$ without fixed points, then there are no finite-to-one maps from $\cals(x)$ into $x$.
This answers one of the open problems posed in (2) of \cite[\S4]{Tachtsis2018}.

(ii) Theorem~\ref{sh120} shows that it is not provable in $\mathsf{ZF}$ that
for all infinite sets $x$, there are no finite-to-one maps from $\cals(x)$ into $x$.
This answers the open problem (7) of \cite[\S4]{Tachtsis2018},
since it is obvious that for all infinite sets $x$, there exists a finite-to-one map from $\cals(x)$ into $x$
if and only if there exists a finite-to-one surjection from $\cals(x)$ onto $x$.

(iii) It follows from Lemma~\ref{sh77} that in the Shelah-type permutation model $\mathcal{V}_\mathrm{S}$,
there exists an infinite set $A$ such that there exists a surjection from $A$ onto $\cals(A)$.
Since it follows from Cantor's theorem that there are no surjections from $A$ onto $\wp(A)$,
we get that there are no surjections from $\cals(A)$ onto $\wp(A)$.
This answers the open problem (8) of \cite[\S4]{Tachtsis2018}.

Also, it follows from Corollary~\ref{sh40} that for all infinite sets $x$,
if $\cals(x)$ is Dedekind infinite, then there are no finite-to-one maps from $\cals(x)$ into $x$.
This is a generalization of Theorem~3.2 of \cite{Tachtsis2018}.

In what follows, we list some open problems which are of interest for future work,
and then summarize the relationships between $\mathfrak{a}!$ and some other cardinals considered in this paper.
Finally, we make a comparison of these relationships with those between $2^\mathfrak{a}$ and some other cardinals.

\subsection{Open problems}
Now, we propose four open problems as follows:

\begin{question}\label{sh121}
Is it consistent with $\mathsf{ZF}$ that there is an infinite cardinal $\mathfrak{a}$
such that $\mathfrak{a}!<\aleph_0\cdot\mathfrak{a}$?
\end{question}

By Lemma~\ref{sh37}, an affirmative answer to this question
would yield a generalization of Theorem~\ref{sh120}.

\begin{question}\label{sh122}
Is it consistent with $\mathsf{ZF}$ that there exists an infinite cardinal $\mathfrak{a}\leqslant2^{\aleph_0}$
such that $\mathfrak{a}!\leqslant_{\mathrm{fto}}\mathfrak{a}$?
\end{question}

By Theorem~\ref{sh38}, an affirmative answer to this question
would give an affirmative answer to Question~\ref{sh121}.

\begin{question}\label{sh123}
Does $\mathsf{ZF}$ prove that $\mathfrak{a}!\neq\seq^{1\text{-}1}(\mathfrak{a})$ for any cardinal $\mathfrak{a}\neq0$?
\end{question}

In \cite[Theorem~4]{HalbeisenShelah1994}, Halbeisen and Shelah proved in $\mathsf{ZF}$ that
$2^\mathfrak{a}\neq\seq^{1\text{-}1}(\mathfrak{a})$ for any cardinal $\mathfrak{a}\geqslant2$.
It is natural to ask whether we can replace $2^\mathfrak{a}$ by $\mathfrak{a}!$ in this theorem.

\begin{question}\label{sh124}
Is it consistent with $\mathsf{ZF}$ that there exists a cardinal $\mathfrak{a}$
such that $\mathfrak{a}!=[\mathfrak{a}]^3$?
\end{question}

Note that, by Theorem~\ref{sh81}, the existence of an infinite cardinal $\mathfrak{a}$
such that $\mathfrak{a}!<[\mathfrak{a}]^3$ is consistent with $\mathsf{ZF}$.

\subsection{Summary}
Now we summarize the results obtained in the previous sections.
For all cardinals $\mathfrak{a}$, if $\mathfrak{a}!$ is Dedekind infinite,
then $\mathfrak{a}!$ cannot be too small, in the following sense:
\begin{itemize}
  \item $\mathfrak{a}!\nleqslant_{\mathrm{dfto}}\seq(\calspdf(\mathfrak{a}))$ (cf.~Theorem~\ref{sh47});
  \item $2^{\aleph_0}\cdot\mathfrak{a}\leqslant2^{\aleph_0}\cdot\seq(\mathfrak{a})
         \leqslant2^{\aleph_0}\cdot\calspdf(\mathfrak{a})\leqslant\mathfrak{a}!$ (cf.~Theorem~\ref{sh48});
  \item $\aleph_0\cdot\mathfrak{a}\leqslant\seq(\mathfrak{a})\leqslant\aleph_0\cdot\calspdf(\mathfrak{a})<\mathfrak{a}!$
        (cf.~Corollary~\ref{sh50}).
\end{itemize}
However, if we replace the requirement that $\mathfrak{a}!$ is Dedekind infinite
by the requirement that $\mathfrak{a}$ is infinite,
then it may consistently happen that $\mathfrak{a}!\leqslant_{\mathrm{fto}}\mathfrak{a}$ (cf.~Theorem~\ref{sh120})
and that $\mathfrak{a}!<\seq^{1\text{-}1}(\mathfrak{a})<\seq(\mathfrak{a})$ (cf.~Proposition~\ref{sh73}).
It is an open problem whether it may consistently happen that $\mathfrak{a}!<\aleph_0\cdot\mathfrak{a}$ (cf.~Question~\ref{sh121}).
Nevertheless, for all infinite cardinals $\mathfrak{a}$, we have:
\begin{itemize}
  \item $\mathfrak{a}!\neq\seq(\mathfrak{a})$ (cf.~Corollary~\ref{sh39});
  \item $\mathfrak{a}!\neq\aleph_0\cdot\mathfrak{a}$ (cf.~Corollary~\ref{sh41});
  \item $\mathfrak{a}^n<\mathfrak{a}!$ (cf.~Corollary~\ref{sh34});
  \item $[\mathfrak{a}]^2\leqslant\bigl[[\mathfrak{a}]^2\bigr]^2<\mathfrak{a}!$ (cf.~Corollary~\ref{sh32}).
\end{itemize}
It is an open problem whether $\mathfrak{a}!\neq\seq^{1\text{-}1}(\mathfrak{a})$ is provable (cf.~Question~\ref{sh123}).
Even for Dedekind infinite cardinals $\mathfrak{a}$,
it is not provable that $\Bigl[\bigl[[\mathfrak{a}]^2\bigr]^2\Bigr]^2\leqslant\mathfrak{a}!$
nor that $([\mathfrak{a}]^2)^2\leqslant\mathfrak{a}!$ (cf.~Proposition~\ref{sh60}),
and it may consistently happen that $\mathfrak{a}!<[\mathfrak{a}]^3$
and $\mathfrak{a}!\leqslant^\ast\mathfrak{a}$ (cf.~Theorem~\ref{sh81}).
It is an open problem whether or not it may consistently happen
that $\mathfrak{a}!=[\mathfrak{a}]^3$ (cf.~Question~\ref{sh124}).

For cardinals $\mathfrak{b}$, $\mathfrak{c}$,
we write $\mathfrak{b}\parallel\mathfrak{c}$ to express that $\mathfrak{b}$ and $\mathfrak{c}$ are incomparable.
For infinite cardinals $\mathfrak{a}$, it may consistently happen that $\mathfrak{a}!\parallel\seq^{1\text{-}1}(\mathfrak{a})$,
$\mathfrak{a}!\parallel\seq(\mathfrak{a})$, $\mathfrak{a}!\parallel[\mathfrak{a}]^3$,
and $\mathfrak{a}!\parallel2^\mathfrak{a}$ (cf.~Lemma~\ref{sh55} and Proposition~\ref{sh60}).
Also, by Lemma~\ref{sh37}, Fact~\ref{sh10}, Fact~\ref{sh03}, and Fact~\ref{sh11},
we have that $\mathfrak{a}!\parallel\aleph_0\cdot\mathfrak{a}$ for any infinite but power Dedekind finite cardinal $\mathfrak{a}$,
and therefore it may consistently happen that $\mathfrak{a}!\parallel\aleph_0\cdot\mathfrak{a}$.

Now, for infinite cardinals $\mathfrak{a}$, we list all the possible relationships between $\mathfrak{a}!$ and
$\mathfrak{a}$, $\aleph_0\cdot\mathfrak{a}$, $\seq^{1\text{-}1}(\mathfrak{a})$, $\seq(\mathfrak{a})$,
$[\mathfrak{a}]^3$, or $2^\mathfrak{a}$ in the following table.

\begin{table}[b]
{\renewcommand{\arraystretch}{1.5}
\begin{tabular}{ | l | c | c | c | c | c | c | } \hline
                    & $\mathfrak{a}$                & $\aleph_0\cdot\mathfrak{a}$   & $\seq^{1\text{-}1}(\mathfrak{a})$
                    & $\seq(\mathfrak{a})$          & $[\mathfrak{a}]^3$            & $2^\mathfrak{a}$                   \\ \hline
\,$\mathfrak{a}!>$  & \makebox[43pt][c]{\checkmark} & \makebox[43pt][c]{\checkmark} & \makebox[43pt][c]{\checkmark}
                    & \makebox[43pt][c]{\checkmark} & \makebox[43pt][c]{\checkmark} & \makebox[43pt][c]{\checkmark}      \\ \hline
\,$\mathfrak{a}!=$                         & \textsf{X} & \textsf{X} & \textbf{?} & \textsf{X} & \textbf{?} & \checkmark \\ \hline
\,$\mathfrak{a}!<$                         & \textsf{X} & \textbf{?} & \checkmark & \checkmark & \checkmark & \checkmark \\ \hline
\,$\mathfrak{a}!\,\parallel$               & \textsf{X} & \checkmark & \checkmark & \checkmark & \checkmark & \checkmark \\ \hline
\,$\mathfrak{a}!\leqslant_{\mathrm{fto}}$  & \checkmark & \checkmark & \checkmark & \checkmark & \checkmark & \checkmark \\ \hline
\,$\mathfrak{a}!\leqslant^\ast$            & \checkmark & \checkmark & \checkmark & \checkmark & \checkmark & \checkmark \\ \hline
\end{tabular}
}
\end{table}

We should also mention that it is consistent with $\mathsf{ZF}$
that there exists a power Dedekind infinite cardinal $\mathfrak{a}$
such that $\mathfrak{a}!\leqslant_{\mathrm{fto}}\aleph_0$.
The sketch of the proof is as follows:
Consider the permutation model $\mathcal{N}2(3)$ in \cite{HowardRubin1998}.
In this permutation model, the set $A$ of atoms is the union of
a denumerable set $B$ of pairwise disjoint $3$-element sets,
$\mathcal{G}$ is the group of all permutations of $A$ that leave $B$ pointwise fixed,
and $\mathcal{I}$ is the normal ideal $\fin(A)$.
It is easy to verify that in $\mathcal{N}2(3)$,
we have $\cals(A)=\calsf(A)$ and there exists a three-to-one surjection from $A$ onto $\omega$.
Therefore, if we take $\mathfrak{a}=|A|$, then we have that $\mathfrak{a}$ is power Dedekind infinite,
$\mathfrak{a}\leqslant_{\mathrm{fto}}\aleph_0$, and $\mathfrak{a}!=\calsf(\mathfrak{a})$,
which implies that, by Fact~\ref{sh01},
$\mathfrak{a}!=\calsf(\mathfrak{a})\leqslant_{\mathrm{fto}}\fin(\mathfrak{a})\leqslant_{\mathrm{fto}}\fin(\aleph_0)=\aleph_0$.

\subsection{Comparison with powers}
The relationships between $2^\mathfrak{a}$ and some other cardinals are studied in
\cite{Specker1954,HalbeisenShelah1994,HalbeisenShelah2001,Forster2003,VejjajivaPanasawatwong2014,Shen2017}.
In \cite[Proposition~3.13]{Shen2017}, the first author proved that
$2^\mathfrak{a}\nleqslant_{\mathrm{dfto}}\seq^{1\text{-}1}(\pdfin(\mathfrak{a}))$
for any power Dedekind infinite cardinal $\mathfrak{a}$.
In fact, for power Dedekind infinite cardinals $\mathfrak{a}$, we have:
\begin{itemize}
  \item $2^\mathfrak{a}\nleqslant_{\mathrm{dfto}}\seq(\pdfin(\mathfrak{a}))$, $\pdfin(\seq(\mathfrak{a}))$,
        $\fin(\pdfin(\mathfrak{a}))$, $\pdfin(\fin(\mathfrak{a}))$;
  \item $2^{\aleph_0}\cdot\mathfrak{a}\leqslant2^{\aleph_0}\cdot\pdfin(\mathfrak{a})\leqslant2^\mathfrak{a}$ (cf.~\cite[Lemma~3.18]{Shen2017});
  \item $\aleph_0\cdot\mathfrak{a}\leqslant\aleph_0\cdot\pdfin(\mathfrak{a})<2^\mathfrak{a}$ (cf.~\cite[Proposition~3.19]{Shen2017}).
\end{itemize}
We shall omit the proof here. It is an open problem whether it is provable in $\mathsf{ZF}$ that
$2^\mathfrak{a}\nleqslant\pdfin(\pdfin(\mathfrak{a}))$ for any power Dedekind infinite cardinal $\mathfrak{a}$.
Hence, $2^\mathfrak{a}$ has stronger properties than $\mathfrak{a}!$, in the sense that
the requirement that $\mathfrak{a}$ is power Dedekind infinite is weaker than
the requirement that $\mathfrak{a}!$ is Dedekind infinite (cf.~Fact~\ref{sh11}),
and $2^\mathfrak{a}\nleqslant_{\mathrm{dfto}}\seq(\pdfin(\mathfrak{a}))$ is stronger than
$2^\mathfrak{a}\nleqslant_{\mathrm{dfto}}\seq(\calspdf(\mathfrak{a}))$ (cf.~Fact~\ref{sh14}).
Also, for infinite cardinals $\mathfrak{a}$, we have:
\begin{itemize}
  \item $2^\mathfrak{a}\nleqslant_{\mathrm{pdfto}}\mathfrak{a}^n$ (cf.~\cite[Proposition~3.11]{Shen2017});
  \item $2^\mathfrak{a}\nleqslant_{\mathrm{fto}}\aleph_0\cdot\mathfrak{a}$;
  \item $\fin(\mathfrak{a})<2^\mathfrak{a}$ (cf.~\cite[Theorem~3]{HalbeisenShelah1994});
  \item $2^\mathfrak{a}\neq\seq^{1\text{-}1}(\mathfrak{a})$ (cf.~\cite[Theorem~4]{HalbeisenShelah1994});
  \item $2^\mathfrak{a}\neq\seq(\mathfrak{a})$ (cf.~\cite[Theorem~5]{HalbeisenShelah1994}).
\end{itemize}
We also omit the proof here.
Notice that, even for infinite cardinals $\mathfrak{a}$,
$2^\mathfrak{a}$ has stronger properties than $\mathfrak{a}!$,
in the sense that it may consistently happen that
$\mathfrak{a}!\leqslant_{\mathrm{fto}}\mathfrak{a}$ (cf.~Theorem~\ref{sh120})
and that $\calsf(\mathfrak{a})=\mathfrak{a}!$ (cf.~Fact~\ref{sh71}).
Nevertheless, it may consistently happen that
$2^\mathfrak{a}<\calsf(\mathfrak{a})=\mathfrak{a}!<\seq^{1\text{-}1}(\mathfrak{a})<\seq(\mathfrak{a})$
(cf.~Fact~\ref{sh71} and Proposition~\ref{sh73})
and hence, by Fact~\ref{sh01}, $2^\mathfrak{a}\leqslant_{\mathrm{fto}}\fin(\mathfrak{a})$.

For the relation $\leqslant^\ast$, on the one hand,
it may consistently happen that $\mathfrak{a}!\leqslant^\ast\mathfrak{a}$ (cf.~Theorem~\ref{sh81}),
and on the other hand, by Cantor's theorem,
we have $2^\mathfrak{a}\nleqslant^\ast\mathfrak{a}$ for any cardinal $\mathfrak{a}$.
Moreover, for all infinite cardinals $\mathfrak{a}$ and all cardinals $\mathfrak{b}\leqslant_{\mathrm{pdfto}}\mathfrak{a}$,
we have $2^\mathfrak{a}\nleqslant^\ast\mathfrak{b}$ (cf.~\cite[Theorem~5.3]{Shen2017}).
Note that in \cite[Theorem~1]{HalbeisenShelah1994}, Halbeisen and Shelah proved that
the existence of an infinite cardinal $\mathfrak{a}$ such that
$2^\mathfrak{a}\leqslant^\ast\fin(\mathfrak{a})$ is consistent with $\mathsf{ZF}$.
Now we propose three open problems concerning the relation $\leqslant^\ast$ as follows:

\begin{question}\label{sh125}
Is it consistent with $\mathsf{ZF}$ that there is an infinite cardinal $\mathfrak{a}$
such that $2^\mathfrak{a}\leqslant^\ast\mathfrak{a}^2$?
\end{question}

This question is known as the dual Specker problem and is asked in \cite{Truss1973}
(cf.~also \cite[p.~133]{Halbeisen2017} or \cite[Problem~5.8]{Shen2017}).

\begin{question}\label{sh126}
Is it consistent with $\mathsf{ZF}$ that there is an infinite cardinal $\mathfrak{a}$
such that $2^\mathfrak{a}\leqslant^\ast[\mathfrak{a}]^2$?
\end{question}

This question is asked in \cite{Halbeisen2018}.
Notice that an affirmative answer to this question
would give an affirmative answer to Question~\ref{sh125}.

\begin{question}\label{sh127}
Is it consistent with $\mathsf{ZF}$ that there is an infinite cardinal $\mathfrak{a}$
such that $2^\mathfrak{a}\leqslant^\ast\aleph_0\cdot\mathfrak{a}$?
\end{question}

In fact, an affirmative answer to this question
would give an affirmative answer to Question~\ref{sh126}.
The sketch of the proof is as follows:
Note that for all power Dedekind infinite cardinals $\mathfrak{a}$,
$\aleph_0\cdot\mathfrak{a}\leqslant^\ast[\mathfrak{a}]^2$.
Hence, we only need to prove that for all infinite cardinals $\mathfrak{a}$,
if $2^\mathfrak{a}\leqslant^\ast\aleph_0\cdot\mathfrak{a}$,
then $\mathfrak{a}$ is power Dedekind infinite.
Let $x$ be a set such that $|x|=\mathfrak{a}$,
and let $f$ be a surjection from $\omega\times x$ onto $\wp(x)$.
Then we can explicitly define a surjection $g\subseteq f$ from a subset of $\omega\times x$ onto $\wp(x)$
such that for all $z\in x$, $g\upharpoonright(\omega\times\{z\})$ is injective.
If $\wp(x)$ is Dedekind finite, then for all $z\in x$, $\dom(g)\cap(\omega\times\{z\})$ is finite,
and hence there is a finite-to-one map from $\dom(g)$ into $x$,
contradicting Theorem~5.3 of \cite{Shen2017}. Therefore, we get that $x$ is power Dedekind infinite,
and hence $\mathfrak{a}$ is power Dedekind infinite.

Finally, for infinite cardinals $\mathfrak{a}$, we list all the possible relationships
between $2^\mathfrak{a}$ and $\mathfrak{a}$, $\aleph_0\cdot\mathfrak{a}$, $\mathfrak{a}^2$,
$\fin(\mathfrak{a})$, $\seq^{1\text{-}1}(\mathfrak{a})$, or $\seq(\mathfrak{a})$ in the following table.

\begin{table}[b]
{\renewcommand{\arraystretch}{1.5}
\begin{tabular}{ | l | c | c | c | c | c | c | } \hline
                     & $\mathfrak{a}$                & $\aleph_0\cdot\mathfrak{a}$       & $\mathfrak{a}^2$
                     & $\fin(\mathfrak{a})$          & $\seq^{1\text{-}1}(\mathfrak{a})$ & $\seq(\mathfrak{a})$           \\ \hline
\,$2^\mathfrak{a}>$  & \makebox[43pt][c]{\checkmark} & \makebox[43pt][c]{\checkmark} & \makebox[43pt][c]{\checkmark}
                     & \makebox[43pt][c]{\checkmark} & \makebox[43pt][c]{\checkmark} & \makebox[43pt][c]{\checkmark}      \\ \hline
\,$2^\mathfrak{a}=$                         & \textsf{X} & \textsf{X} & \textsf{X} & \textsf{X} & \textsf{X} & \textsf{X} \\ \hline
\,$2^\mathfrak{a}<$                         & \textsf{X} & \textsf{X} & \textsf{X} & \textsf{X} & \checkmark & \checkmark \\ \hline
\,$2^\mathfrak{a}\,\parallel$               & \textsf{X} & \checkmark & \checkmark & \textsf{X} & \checkmark & \checkmark \\ \hline
\,$2^\mathfrak{a}\leqslant_{\mathrm{fto}}$  & \textsf{X} & \textsf{X} & \textsf{X} & \checkmark & \checkmark & \checkmark \\ \hline
\,$2^\mathfrak{a}\leqslant^\ast$            & \textsf{X} & \textbf{?} & \textbf{?} & \checkmark & \checkmark & \checkmark \\ \hline
\end{tabular}
}
\end{table}

\subsection*{Acknowledgements}
We would like to give thanks to Professor Qi Feng and Professor Liuzhen Wu
for their advice and encouragement during the preparation of this paper.
Both authors were partially supported by NSFC No.~11871464.


\end{document}